\documentclass[11pt,a4paper]{article}%
\usepackage[centertags]{amsmath}
\usepackage{amsfonts}
\usepackage{amssymb}
\usepackage{amsthm}
\usepackage{epsfig}
\usepackage{setspace}
\usepackage{ae}
\usepackage{eucal}
\usepackage[usenames]{color}%
\usepackage{graphicx}

\theoremstyle{plain}
\newtheorem{thm}{Theorem}[section]
\newtheorem{lem}[thm]{Lemma}

\newtheorem{prop}[thm]{Proposition}
\theoremstyle{definition}

\theoremstyle{remark}

\newtheorem{rem}[thm]{Remark}

\renewcommand{\div}{\operatorname{div}}

\begin{document}

\title{Constructive analysis of the incompressible Navier-Stokes equation}
\author{J\"org Kampen }
\maketitle

\begin{abstract}
 A global time-discretized scheme for the Navier-Stokes equation system in its Leray projection form is defined.  It is shown that the scheme converges to a bounded global classical solution for smooth data which have polynomial decay at infinity. Furthermore, the algorithm proposed is extended to the situation of initial-boundary value problems. Algorithms constructed in a different context (cf. \cite{FrKa, KKS, FrKa2, K2}) may be used within the proposed scheme in order to compute the solution of Leray's form of the Navier-Stokes system. The main idea for global existence is to define a control function dynamically and show explicitly that the scheme which solves a controlled Navier-Stokes type equation can control the modulus of velocity and the first derivatives of velocity to be bounded. The method described here can be extended to Navier-Stokes equations on compact manifolds which is done in a subsequent paper.

\end{abstract}


\section{Introduction}
In its Leray projection form the incompressible Navier-Stokes equation is a semilinear partial integro-differential equation system, where the integral term is quadratic with respect to the gradient of velocity. This integral term requires a careful treatment in order to control the growth of the solution (whatever scheme you propose). This is the main difference to the multivariate Burgers equation.
Indeed, for the multivariate Burgers equation, i.e., the Cauchy problem
\begin{equation}\label{qparasyst1}
\left\lbrace \begin{array}{ll}
\frac{\partial u_i}{\partial t}=\nu\sum_{j=1}^n \frac{\partial^2 u_i}{\partial x_j^2} 
-\sum_{j=1}^n u_j\frac{\partial u_i}{\partial x_j},\\
\\
\mathbf{u}(0,.)=\mathbf{h},
\end{array}\right.
\end{equation}
on $[0,\infty)\times {\mathbb R}^n$ (where $\nu$ is some strictly positive constant, i.e., $\nu>0$), and $1\leq i\leq n$)  global solutions may be constructed via the
a priori estimate 
\begin{equation}\label{apriorimax}
\max_{j}\sup_{x\in \Omega}|u_j(t,x)|\leq \max_{j}\sup_{x\in \Omega}|h_j(x)|,
\end{equation}
which may be obtained from estimates of the form
\begin{equation}\label{aprioriest}
\frac{\partial}{\partial t}\|u(t,.)\|_{H^s}\leq \|u(t,.)\|_{H^{s+1}}\sum_{i,j}\sum_{|\alpha|+|\beta|\leq s}\|D^{\alpha}u_iD^{\beta}u_j\|_{L^2}-2\|\nabla u\|^2_{H^s}
\end{equation}
for some positive $s\in {\mathbb R}$ (this is the standard notation for Sobolev spaces $H^s$ which will be defined below). This type of estimate is valid also for initial-value problems   with periodic boundary conditions which are are essentially initial-value problems defined on the $n$-torus ${\mathbb T}^n$. From this point of view the incompressible Navier-Stokes equation system is a multivariate Burgers equation system (of Cauchy type) with an external force, and with the negative gradient pressure as source terms, where it is the pressure term which makes the equation global and difficult to control. The pressure may be determined by an incompressibility condition for the velocity which is satisfied if the divergence of the velocity field is zero. Indeed, classically, the Navier-Stokes equation system is defined via three equations, the nonlinear diffusion equation
  \begin{equation}\label{nav}
    \frac{\partial\mathbf{v}}{\partial t}-\nu \Delta \mathbf{v}+ (\mathbf{v} \cdot \nabla) \mathbf{v} = - \nabla p + \mathbf{f}_{ex}~~~~t\geq 0,~~x\in{\mathbb R}^n,
\end{equation}
the incompressibility condition 
\begin{equation}\label{navdiv}
\nabla \cdot \mathbf{v} = 0,~~~~t\geq 0,~~x\in{\mathbb R}^n, 
\end{equation}
and the initial condition
\begin{equation}\label{navinit}
\mathbf{v}(0,x)=\mathbf{h}(x),~~x\in{\mathbb R}^n.
\end{equation}
The Leray projection for this equation system is obtained from a Poisson equation for the pressure, i.e., we have
\begin{equation}\label{poissonp}
\begin{array}{ll} 
- \Delta p=\sum_{j,k=1}^n \left( \frac{\partial}{\partial x_k}v_j\right) \left( \frac{\partial }{\partial x_j}v_k\right),
\end{array}
\end{equation}
with the solution
\begin{equation}\label{pp}
p(t,x)=-\int_{{\mathbb R}^n}K_n(x-y)\sum_{j,k=1}^n\left( \frac{\partial v_k}{\partial x_j}\frac{\partial v_j}{\partial x_k}\right) (t,y)dy.
\end{equation}
Here,
\begin{equation}
K_n(x):=\left\lbrace \begin{array}{ll}\frac{1}{2\pi}\ln |x|,~~\mbox{if}~~n=2,\\
\\ \frac{1}{(2-n)\omega_n}|x|^{2-n},~~\mbox{if}~~n\geq 3\end{array}\right.
\end{equation}
is the Poisson kernel. Furthermore, $|.|$ denotes the Euclidean norm and $\omega_n$ denotes the area of the unit $n$-sphere. Since the function $x\rightarrow |x|^{\mu}$ is locally integrable if and only if $\mu >-n$ we observe that for $n\geq 3$ the functions
\begin{equation}\label{Kkernel}
x\rightarrow \frac{\partial}{\partial x_l}K_n(x)=\omega_n^{-1}\frac{x_l}{|x|^n}
\end{equation}
are integrable around $x=0$ for $1\leq l\leq n$ (you may observe this explicitly by writing the derivative of the kernel $K_n$ in (\ref{Kkernel}) in polar coordinates). Note that  the derivatives of the velocity in the expression for the pressure in (\ref{pp}) should have some decay at spatial infinity in order that the Leray projection form makes sense. 
Later, we shall observe that a solution $\mathbf{v}$ can be constructed which decays at spatial infinity such that the integral on the right side of (\ref{pp}) exists globally. More precisely, we shall see that for initial data $\mathbf{h}=(h_1,\cdots ,h_n)^T$ which are in Sobolev spaces $\left[ H^s\left( {\mathbb R}^n\right)\right]^n$ for $s$ large enough (i.e., for $h_i\in H^s\left({\mathbb R}^n\right)$ for $1\leq i\leq n$) we have a certain decay at spatial infinity of the solution such that
\begin{equation}\label{vnorm}
\int_{{\mathbb R}^n}\sum_{j,k=1}^n{\Big |}\left( \frac{\partial v_k}{\partial x_j}\frac{\partial v_j}{\partial x_k}\right) (t,y){\Big |}dy<\infty .
\end{equation}
An $L^1$-(resp. $H^1$)-estimate can be naturally obtained for (\ref{vnorm}) if we have $H^1$-(resp. $H^2$)-estimates for the velocity, since
\begin{equation}
\sum_{j,k=1}^n{\Big |}\left( \frac{\partial v_k}{\partial x_j}\frac{\partial v_j}{\partial x_k}\right) (t,y){\Big |}\leq \frac{n^2}{2}\left( \left( \frac{\partial v_k}{\partial x_j}\right)^2(t,y)+\left( \frac{\partial v_j}{\partial x_k}\right)^2(t,y)\right).
\end{equation}

It is well-known that a global classical solution justifies that we speak of 'the' solution. Hence, it makes sense to define the Navier-Stokes Cauchy problem for divergence free velocity fields by
\begin{equation}\label{Navleray}
\left\lbrace \begin{array}{ll}
\frac{\partial v_i}{\partial t}-\nu\sum_{j=1}^n \frac{\partial^2 v_i}{\partial x_j^2} 
+\sum_{j=1}^n v_j\frac{\partial v_i}{\partial x_j}=\\
\\ \hspace{1cm}\int_{{\mathbb R}^n}\left( \frac{\partial}{\partial x_i}K_n(x-y)\right) \sum_{j,k=1}^n\left( \frac{\partial v_k}{\partial x_j}\frac{\partial v_j}{\partial x_k}\right) (t,y)dy,\\
\\
\mathbf{v}(0,.)=\mathbf{h},
\end{array}\right.
\end{equation}
for $1\leq i\leq n$, where external forces $\mathbf{f}_{ex}$ are set to zero for simplicity. In order to have polynomial decay at spatial infinity of the data $\mathbf{h}$ we take a standard assumption that we have $\mathbf{h}\in \left[ H^s\right]^n$ for all $s\in {\mathbb R}$. This assumption may be weakened a bit (this may depend on the regularity which you want to achieve- cf. below), but from an algorithmic or physical perspective it is satisfying. Note that here and in the following we write
\begin{equation}
\left( \frac{\partial v_k}{\partial x_j}\frac{\partial v_j}{\partial x_k}\right) (t,y)
:= \frac{\partial v_k}{\partial x_j}(t,y)\frac{\partial v_j}{\partial x_k}(t,y)
\end{equation}
for the sake of brevity (similar for sums of functions etc.). We look for classical solutions in the space of divergence free vector fields, i.e., in the space
\begin{equation}
\left\lbrace \mathbf{v}\in \left[ C^{1,2}\left( \left[0,\infty\right)\times {\mathbb R}^n\right) \right]^n|\div \mathbf{v}=0\right\rbrace ,
\end{equation}
where the solution $\mathbf{v}$ should be bounded and regular (i.e. derivatives should exist in a classical sense such that uniqueness is guaranteed), and where some decay at spatial infinity guarantees that (\ref{vnorm}) is satisfied. 
Here, the function space  $C^{1,2}\left( \left[0,\infty\right)\times {\mathbb R}^n\right)$ denotes the classical function space of scalar functions which have continuous first derivatives with respect to time and continuous spatial derivatives up to second order, and if we add a subscript $b$, then the function space  $C^{1,2}_b\left( \left[0,\infty\right)\times {\mathbb R}^n\right)\subset C^{1,2}\left( \left[0,\infty\right)\times {\mathbb R}^n\right)$ is the space of scalar functions which have bounded time derivatives up to first order and bounded spatial derivatives up to second order.
Furthermore, the function space $\left[ C^{1,2}_b\left( \left[0,\infty\right)\times {\mathbb R}^n\right) \right]^n$ is the space of vector-valued functions $\mathbf{v}:=\left(v_1,\cdots ,v_n\right)^T$ with components $v_i\in C^{1,2}_b\left( \left[0,\infty\right)\times {\mathbb R}^n\right)$ for all $1\leq i\leq n$. Here, and in the following we write vector-valued functions in boldface letters.
Prima facie it seems reasonable to measure the construction of the Navier-Stokes solution $\mathbf{v}=(v_1,\cdots ,v_n)$ in a $|.|_{1,2}$-norm of functions with globally bounded time derivatives up to first order and spatial derivatives up to second order (for each component function $v_i$) plus an integral norm related to (\ref{vnorm}). Note that the norm $|.|_{1,2}$ does not lead to a Banach space for infinite domains, i.e., the limit of successive approximations with finite $|.|_{1,2}$ norms may not have a finite $|.|_{1,2}$-norm. However, this is different for finite domains, and may be exploited when Navier Stokes equations are considered on a compact manifold, for example (cf. \cite{KM}).

 The incompressible Navier-Stokes equation cannot be solved by a simple global fixed point iteration. Similar as in the case of the multidimensional Burgers equation (cf. \cite{K}) we choose a time-discretized scheme and construct fixed points which are local in time. However, compared to the multidimensional Burgers equation controlling the growth of the integral term in (\ref{Navleray}) is an  additional difficulty. We considered this problem in \cite{KB2}.  For $l\geq 1$ let
\begin{equation}
\begin{array}{ll}
\mathbf{v}^{\rho,l}=\left(v^{\rho,l}_1,\cdots, v^{\rho,l}_n\right)^T:[l-1,l]\times{\mathbb R}^n\rightarrow {\mathbb R}^n
\end{array}
\end{equation}
be a solution of (\ref{Navleray}) on the domain $[l-1,l]\times{\mathbb R}^n$ in transformed time coordinates $\tau$ with
\begin{equation}
t=\rho_l\tau
\end{equation}
and with initial data $\mathbf{v}^{\rho,l}(l-1,.)=\mathbf{v}^{\rho,l-1}(l-1,.)$ being the final data of the previous time step number $l-1$, where $\mathbf{v}^{\rho,1}(0,.)=\mathbf{h}(,.)$.
 Note that $\mathbf{v}^{\rho,l}(l-1,.)=\mathbf{v}\left(\sum_{m=1}^{l-1}\rho_m,.\right)$ for $l\geq 1$, where $\mathbf{v}$ denotes the velocity in original time coordinates. 
Here, the idea is that choices $\rho_l<1$ small enough lead to a local converging iteration scheme in transformed coordinates $\tau$ on domains $[l-1,l]\times {\mathbb R}^n$, 
while on the other side the numbers $\rho_l$ may be chosen large enough such that the sum of time-step size $\rho_l$ diverges.  The latter condition implies that the scheme is global in time (the discussion here is preliminary; cf. below for a more extensive discussion of this time discretization). However, it seems difficult to control the growth of the solutions $v^{\rho,l}_i$ directly if the time step size series $(\rho_l)_l$ is large enough such that
\begin{equation}\label{rhosum}
 \sum_{l=1}^{N}\rho_l\uparrow \infty~\mbox{ as }~N\uparrow \infty.
\end{equation}
Note that the requirement (\ref{rhosum}) for a global scheme is weaker than the requirement of a scheme where the time step size is bounded from below independently of the time step number $l$. Indeed the requirement (\ref{rhosum}) allows us to choose time step sizes $\rho_l$ of order
\begin{equation}\label{rhosize}
\rho_l\sim \frac{1}{l}
\end{equation}
in order to define a global scheme. This does not mean that we shall construct a solution which has a linear bound with respect to time. Indeed we shall construct a bounded solution $\mathbf{v}$ where all components $v_i$ are gobally bounded with respect to the $|.|_{1,2}$ norm. However, choosing time step sizes of order (\ref{rhosize}) may be useful in order to control the 
integral magnitude (\ref{vnorm}). This depends on the control function which is chosen - as we shall see below. Linear growth of the solution with respect to a supremum norm up to a certain order of derivatives is not sufficient in order to define a global scheme for problems which are defined on an infinite domain.

Therefore we shall construct a bounded   solution of an equivalent problem with solution $\mathbf{v}^{r,\rho}$ defined recursively via a series $\mathbf{v}^{r,\rho,l}$. Here 'bounded' means bounded with respect to the supremum norm. For the integral magnitude (\ref{vnorm}) it suffices to have linear growth in time. The choice of the time step size (\ref{rhosize}) will ensure the local convergence of the scheme.
Here for each $l\geq 1$ the functions $\mathbf{v}^{r,\rho,l},~\mathbf{v}^{\rho,l},~\mathbf{r}^l$ are all defined on the domain $[l-1,l]\times {\mathbb R}^n$. For $\tau=l-1$ we have $\mathbf{v}^{r,\rho,l}=\mathbf{v}^{r,\rho,l-1}$ for all $l\geq 1$.
The series $\left( \mathbf{r}^l\right)_l$ is designed in order to control the source terms of the equations for $\mathbf{v}^{r,\rho,l}$.
For integers $l\geq 1$ (time-steps) we shall construct a series of real numbers $\left(\rho_l\right)$ and a family of recursively defined functions $\left(r^l_i\right),~l\in {\mathbb N},~1\leq i\leq n$, where
\begin{equation}
\begin{array}{ll}
r^l_i:[l-1,l]\times {\mathbb R}^n\rightarrow {\mathbb R},
\end{array}
\end{equation}
and consider the equation system for the functions
\begin{equation}
v^{r,\rho,l}_i=v^{\rho ,l}_i+r^l_i.
\end{equation}
It is not sufficient to construct a global upper bound for $\mathbf{v}^{\rho}+\mathbf{r}$ of course.
We shall construct a global upper bound for $\mathbf{r}=(\mathbf{r}^l)_{l\geq 1}=(r^l_1,\cdots ,r^l_n)_{l\geq 1}$ and a global upper bound for $\mathbf{v}^{r,\rho}$ and shall conclude that $\mathbf{v}=\mathbf{v}^{\rho}=\mathbf{v}^{r,\rho}-\mathbf{r}$ has itself a global upper bound. The global upper bound for $v_i(\tau,.)$ with respect to the $|.|_{2}$ and with respect to the $|.|_{H^2}$-norm. Moreover, we shall construct $\mathbf{r}^l$ such that the spatial derivatives of $\mathbf{v}^{r,\rho}$ and of $\mathbf{r}$ are bounded up to first order (up to first order is essential). 
For $l\geq 1$ the functions $r^l_i,~1\leq i\leq n$ (which determine a global function $\mathbf{r}$) satisfy themselves for each time step number $l$ a multivariate partial differential equation system which is determined dynamically within the solution scheme for $\mathbf{v}^{r,\rho}$ and $\mathbf{r}$. In order to explain this in more detail we first consider the equations for the functional series  $\mathbf{v}^{r,\rho,l},~l\geq 1$ for a rather arbitrary class of functions $r^l_i\in C^{1,2}_b\left(\left[ l-1,l\right] \times {\mathbb R}^n \right) $, i.e., for all $1\leq i\leq n$ the functions $(\tau,x)\rightarrow r^{l}_{i}(\tau,x)$ are assumed to be bounded functions with a bounded time derivative ( in a weak sense at the integer values $l\geq 1$ and in a classical sense elsewhere) and bounded spatial derivatives up to second order on the domain $\left(l-1,l\right] \times {\mathbb R}^n$. Furthermore, the functions are defined recursively with respect to the time step number $l$. For $\mathbf{r}=(r_1,\cdots ,r_n)^T:[0,\infty)\times {\mathbb R}^n\rightarrow {\mathbb R}^n$ we have
\begin{equation}
r_i(\tau,x)=r^l_i(\tau,x)~\mbox{iff}~(\tau,x)\in [l-1,l]\times {\mathbb R}^n.
\end{equation}
We construct global functions $\mathbf{v}^{r,\rho}:[0,\infty)\times {\mathbb R}^n\rightarrow {\mathbb R}^n$ and $\mathbf{r}:[0,\infty)\times {\mathbb R}^n$ time step by time step. First we prove that for some bounded function $\mathbf{r}$ the function $(t,x)\rightarrow \mathbf{v}^r(t,x)=\mathbf{v}^{r,\rho}(\tau,x)$ is bounded and a solution of an equation system which is equivalent to the Navier-Stokes equation system. This solution is classical locally for all domains $\left[l-1,l \right) \times {\mathbb R}^n$ and is a weak solution at the points $\tau =l$ in transformed time coordinates. Then we conclude that there is a solution of the Navier-Stokes equation system $\mathbf{v}=\mathbf{v}^r+\mathbf{r}$ which is bounded and in $C^{1,2}\left(\left[0,\infty \right)\times {\mathbb R}^n  \right) $. Bounded classical solutions with decay to zero at infinity lead to further regularity and to uniqueness. Indeed, it is well-known that energy estimates imply that the Navier-Stokes equation 'determines' $\mathbf{v}$ among the smooth solutions. Next for each time step number $l$ the restriction $\mathbf{v}^{r,\rho,l}$ of the function $\mathbf{v}^{r,\rho}$ to the domain $[l-1,l]\times {\mathbb R}^n$, i.e., the function
\begin{equation}\label{vvr}
\begin{array}{ll}
\mathbf{v}^{r,\rho,l}=\left(v^{r,\rho,l}_1,\cdots, v^{r,\rho,l}_n\right)^T:[l-1,l]\times{\mathbb R}^n\rightarrow {\mathbb R}^n\\
\\
v^{r,\rho,l}_i= v^{\rho,l}_i+r^l_i,~
\end{array}
\end{equation}
satisfies
\begin{equation}\label{Navlerayrrhol}
\left\lbrace \begin{array}{ll}
\frac{\partial v^{r,\rho,l}_i}{\partial \tau}-\rho_l\nu\sum_{j=1}^n \frac{\partial^2 v^{r,\rho,l}_i}{\partial x_j^2} 
+\rho_l\sum_{j=1}^n v^{r,\rho,l}_j\frac{\partial v^{r,\rho,l}_i}{\partial x_j}=L^{\rho ,l}_i(\mathbf{r}^l;\mathbf{v}^{r,\rho,l})+\\
\\ \hspace{1cm}\rho_l\int_{{\mathbb R}^n}\left( \frac{\partial}{\partial x_i}K_n(x-y)\right) \sum_{j,k=1}^n\left( \frac{\partial v^{r,\rho,l}_k}{\partial x_j}\frac{\partial v^{r,\rho,l}_j}{\partial x_k}\right) (\tau,y)dy+r^l_{i,\tau},\\
\\
\mathbf{v}^{r,\rho,l}(l-1,.)=\mathbf{v}^{r,\rho,l-1}(l-1,.),
\end{array}\right.
\end{equation}
 where 
\begin{equation}
\begin{array}{ll}
L^{\rho ,l}_i(\mathbf{r}^l;\mathbf{v}^{r,\rho,l})\equiv -\rho_l\nu \Delta r^l_i+\rho_l\sum_{j=1}^n r^{l}_j\frac{\partial r^{l}_i}{\partial x_j}\\
\\
+\rho_l\sum_{j=1}^n r^l_j\frac{\partial v^{r,\rho,l}_i}{\partial x_j}+\rho_l\sum_{j=1}^n v^{r,\rho,l}_j\frac{\partial r^{l}_i}{\partial x_j}\\
\\ -2\rho_l\int_{{\mathbb R}^n}\left( \frac{\partial}{\partial x_i}K_n(x-y)\right) \sum_{j,k=1}^n\left( \frac{\partial r^{l}_k}{\partial x_j}\frac{\partial v^{r,\rho,l}_j}{\partial x_k}\right) (\tau,y)dy\\
\\
-\rho_l\int_{{\mathbb R}^n}\left( \frac{\partial}{\partial x_i}K_n(x-y)\right) \sum_{j,k=1}^n\left( \frac{\partial r^{l}_k}{\partial x_j}\frac{\partial r^{l}_j}{\partial x_k}\right) (\tau,y)dy.
\end{array}
\end{equation}
Note that the functional $L^{\rho ,l}_i(\mathbf{r}^l;\mathbf{v}^{r,\rho,l})$ is affine with respect to  $\mathbf{v}^{r,\rho,l}$ and its first derivatives. Sometimes it is useful to consider the linear part $L^{\rho ,l,0}_i(\mathbf{r}^l;\mathbf{v}^{r,\rho,l})$ of the functional $L^{\rho ,l}_i(\mathbf{r}^l;\mathbf{v}^{r,\rho,l})$. Therefore we write 
\begin{equation}
\begin{array}{ll}
L^{\rho ,l}_i(\mathbf{r}^l;\mathbf{v}^{r,\rho,l})
=:-\rho_l\nu \Delta r^l_i+\rho_l\sum_{j=1}^n r^{l}_j\frac{\partial r^{l}_i}{\partial x_j}\\
\\
-\rho_l\int_{{\mathbb R}^n}\left( \frac{\partial}{\partial x_i}K_n(x-y)\right) \sum_{j,k=1}^n\left( \frac{\partial r^{l}_k}{\partial x_j}\frac{\partial r^{l}_j}{\partial x_k}\right) (\tau,y)dy+L^{\rho ,l,0}_i(\mathbf{r}^l;\mathbf{v}^{r,\rho,l})\\
\\
=:S^{\rho,l}_i\left(\mathbf{r}^l\right)+L^{\rho ,l,0}_i(\mathbf{r}^l;\mathbf{v}^{r,\rho,l}).
\end{array}
\end{equation}
 In (\ref{Navlerayrrhol}) we have $\mathbf{v}^{r,\rho,0}(0,.)=\mathbf{h}$ in case $l=1$. Note that $L^{\rho ,l,0}_i(\mathbf{r}^l;\mathbf{v}^{r,\rho,l})$ is a linear differential operator with respect to the function $\mathbf{v}^{r,\rho,l}$. Note that all terms of the functional $L^{\rho,l}_i$ have a factor $\rho_l$.
The numbers $\rho_l$ measure the time step size in original coordinates and will be chosen such that at each time step we can construct a local time solution of the incompressible Navier-Stokes equation by an iteration procedure where we use $H^2$-estimates implying a decay at spatial infinity of the iterative approximations of the solution. Furthermore, the functions $\mathbf{r}^l$ are chosen  such that at each time step $l$ we can control  the source terms on the right side of (\ref{Navlerayrrhol}) depending on the computation of the data from the last time step. The control functions $\mathbf{r}^l$ are defined such that we do not only control the absolute value of the functions  $\mathbf{r}^l;\mathbf{v}^{r,\rho,l}$ but also the absolute value of the first derivatives. This is a crucial idea of our method. More specifically
\begin{itemize}
 \item[$\alpha )$] the choice of $\rho_l$ ensures the time-local convergence of an iteration scheme which computes approximations $\mathbf{v}^{r,\rho,k,l}$ for the local solution $\mathbf{v}^{r,\rho,l}$ of the incompressible Navier-Stokes equation at each time step $l$, i.e., the solution on the domain $[l-1,l]\times {\mathbb R}^n$ with respect to transformed time coordinates $\tau$. For each time step number $l\geq 1$ the local iteration determines the function $\mathbf{v}^{r,\rho,l}$ as a limit of the functional series $\mathbf{v}^{r,\rho,k,l}, k\geq 0$. For given $\mathbf{r}$ and each $k\geq 0$ the function $\mathbf{v}^{r,\rho,k,l}$ satisfies
\begin{equation}\label{Navlerayrrhol+}
\left\lbrace \begin{array}{ll}
\frac{\partial v^{r,\rho,k,l}_i}{\partial \tau}-\rho_l\nu\sum_{j=1}^n \frac{\partial^2 v^{r,\rho,k,l}_i}{\partial x_j^2} 
+\rho_l\sum_{j=1}^n v^{r,\rho,k-1,l}_j\frac{\partial v^{r,\rho,k,l}_i}{\partial x_j}\\
\\
=L^{\rho ,l}_i(\mathbf{r}^l;\mathbf{v}^{r,\rho,k-1,l})+\\
\\
\rho_l\int_{{\mathbb R}^n}\left( \frac{\partial}{\partial x_i}K_n(x-y)\right) \sum_{j,m=1}^n\left( \frac{\partial v^{r,\rho,k-1,l}_m}{\partial x_j}\frac{\partial v^{r,\rho,k-1,l}_j}{\partial x_m}\right) (\tau,y)dy+r^l_{i,\tau},\\
\\
\mathbf{v}^{r,\rho,k,l}(l-1,.)=\mathbf{v}^{r,\rho,l-1}(l-1,.).
\end{array}\right.
\end{equation}
Here, in case of $k=0$ (resp. $k-1=-1$) we define
\begin{equation}
\mathbf{v}^{r,\rho,-1,l}(\tau,x)=\mathbf{v}^{r,\rho,l-1}(l-1,x)\mbox{ for all }(\tau,x)\in [l-1,l]\times {\mathbb R}^n.
\end{equation}
Note that in \cite{KB2} we determined a the time step size 
\begin{equation}\label{rhoCl}
 \rho_l\geq \frac{C}{l}
\end{equation}
for some constant $C>0$ where all local schemes for $\mathbf{v}^{r,\rho,k,l}$ converge while $\sum_{l= 1}^M\rho_l\uparrow \infty$ as $M\uparrow \infty$. The latter property of the time step sizes makes the scheme global. There is a difference to the previous scheme for the multivariate Burgers equation which has a time step size of order $1/l$ (cf. \cite{K}) in order to compensate for a bound on the coefficients which is linear in time. In the present scheme we show that we can use a constant time step size $\rho_>0$
 in order to keep the additional integral term of the Navier Stokes equation ('additional' compared to the multidimensional Burgers equation) under control. Note that it is only for analytical reasons that a decreasing time step size is used in \cite{KB2} in order to make the scheme global (because the sum of time step sizes goes to infinity as the number of time steps goes to infinity). A scheme with constant or even increasing time step size is preferable from a numerical/computational point of view. Once we have proved that the solution function $\mathbf{v}$ is bounded itself we can alter the scheme for numerical purposes to a scheme with increasing time step size. 
 \item[$\beta )$] 
   In order to determine the function $\mathbf{r}^l$ for each $l$ recursively we assume that $\mathbf{v}^{r,l-1}(l-1,.)$ and $\mathbf{r}^{l-1}(l-1,.)$ have been computed (in case $l=1$ the function $\mathbf{v}^{r,0}(l-1,.)$ equals the initial data function $\mathbf{h}$ which is given, and $\mathbf{r}^{0}$ will be set to zero). Given this information of the previous time step we first determine a function $\phi^l_i$ which is constructed in order to control the growth of the functions $r^l_i$ and of the function $\mathbf{v}^{r,\rho,l}$. The functions $\phi^l_i$ are introduced as source terms of a linear equation for $r^l_i$ at each time step. This implies the control of the function $\mathbf{v}^{\rho,l}$. It is essential to control a spatial H\"{o}lder norm (which includes the supremum norm). For this purpose it is sufficient to control the supremum and the first derivatives to be bounded, of course. The functions $\phi^l_i$ determine the functions $r^l_i$ first in terms of the functions $\mathbf{r}^{l-1}$ and $\mathbf{v}^{r,\rho,l-1}$ and then in turn  control the growth the local Navier-Stokes solution via the control of growth of the functions  $v^{r,\rho,l}_i$ and $r^l_i$ at time step $l$. At each time step we control the growth of the local solution $\mathbf{v}^{r,\rho,l}$ as a  limit of an iteration procedure of successive approximations
   \begin{equation}\label{vkapprox0}
   \mathbf{v}^{r,\rho,k,l}=\mathbf{v}^{r,\rho,0,l}+\sum_{m=1}^k\left( \mathbf{v}^{r,\rho,m,l}-\mathbf{v}^{r,\rho,m-1,l}\right) ,
   \end{equation}
where we control the growth of the successive approximations of this series, and then show that  this control is preserved in the limit $k\uparrow \infty$. The advantage of this approach is that all members of the series in (\ref{vkapprox0}) are solutions of linear equations. 
At time step $l\geq 1$ we compute the function $\mathbf{r}^l$ in terms of the data of the previous time step where we look at the values itself and at the values of the derivatives of the functions $\mathbf{r}^{l-1}$ and $\mathbf{v}^{r,\rho,l-1}$. Having determined this function $\mathbf{r}^l$  appropriately we show that the growth of the  function  $\mathbf{v}^{r,\rho,0,l}$ is controlled in terms of a  a norm which includes second order derivatives and that the functions $v^{r,\rho,0,l}_i(\tau,.)$ can be estimated with repsect to ab $|.|_{H^2}$-norm independent of time $\tau \in [l-1,l]$. Then we shall show that for small time step size $\rho>0$ the sum with the summands $\delta\mathbf{v}^{r,\rho,k,l}=\left( \mathbf{v}^{r,\rho,m,l}-\mathbf{v}^{r,\rho,m-1,l}\right)$ of (\ref{vkapprox0})  is small enough such that some growth property established for $\mathbf{v}^{r,\rho,0,l}$ is essentially preserved for all approximations $\mathbf{v}^{r,\rho,k,l}$  and in the limit for $\mathbf{v}^{r,\rho,l}$. Moreover we show that there is a time step size $\rho$ which is independent of the time step number $l>0$ - this is an effect of the control function $r^l_i$- in \cite{KB2} we have a lower bound of form (\ref{rhoCl}) with a fixed constant $C>0$. In any case, the scheme is global. We exploit that we have some freedom in the choice of the functions $\mathbf{r}^l$.  The construction is such that the global function $\mathbf{r}$ which equals the local function $\mathbf{r}^l$ on each domain $[l-1,l]\times {\mathbb R}^n$ is bounded. Each $\mathbf{r}^l$ is itself a solution of a linear equation 
with certain source terms $\phi^l_i$ which serve as 'consumption terms' in certain regions of the domain $[l-1,l]\times {\mathbb R}^n$ and control the growth of  the functions for $\mathbf{v}^{r,\rho,0,l}$ and $\mathbf{r}^{l}$. This linear equation is essentially a linearized equation related to the terms on the right side of the equation for $\mathbf{v}^r$ above.  For each $1\leq i\leq n$ the function $\phi^l_i$ is determined in terms of the information gained at the previous time step, i.e., in terms of the functions $x\rightarrow r^{l-1}_i(l-1,x)$ and $x\rightarrow v^{r,\rho,l-1}_i(l-1,x)$. For simplicity of notation we denote these functions by $r^{l-1}_i(l-1,.)$ and $ v^{r,\rho,l-1}_i(l-1,.)$ sometimes.
   The functions $\phi^l_i$ are constructed grosso modo as follows (details will be given in the proof of the main theorem below). For each $1\leq i\leq n$ and at each time step $l$ the function $\phi^l_i$ is constructed as a sum $\phi^l_i=\phi^{v,l}_i+\phi^{r,l}_i$. 
 We define
\begin{equation}
\phi^{r,l}_i(\tau,x)=\phi^{r,s,l}_i(x):=
 -\frac{2}{C}r^{l-1}_i(l-1,.).
\end{equation}
Similarly for the source term $\phi^{v,l}_i$, but with the difference, that we put it on a diiferent scale, i.e., make it smaller, such that it does not interfer essentially with the growth of the control function. We define
\begin{equation}\label{phiv}
\phi^{r,l}_i(\tau,x)=\phi^{r,s,l}_i(x):=
 -\frac{2}{C^2}v^{r,l}_i(l-1,.).
\end{equation}
The main idea is that $C>0$ may be large such that the effect of the control source term for $\phi^{v,s,l}_i$ does not interfer with the effect of  the control source term for $\phi^{r,s,l}_i$. Therefore the square in the denominator of (\ref{phiv}). Nevertheless, we may have  $\rho\sim \frac{1}{C^3}$ such that for $C>0$ large enough even the intergation pf $\phi^{v,\rho,l}_i$ of one unit (in $\tau$-coordinates) is large compared to all the other source terms in the scheme which all have a factor $\rho$. The choice of the generic $C>0$ is step by step. 
We shall construct $\mathbf{r}^l$ such that $|r^l_i|_0\leq C$ for some $C>0$ independent of the time step number $l\geq 1$ (a bound with respect to the supremum norm $|.|_0$), and $|r^l_i|_1\leq C$ (a bound with respect to the classical norm including the first order derivatives) for all $l\geq 1$, and all $1\leq i\leq n$. Moreover, we shall ensure that
\begin{equation}
|r^{l}_i(\tau,.)|_{H^2}\leq C,
\end{equation}
independent of the time step number $l\geq 1$.

\begin{rem}
The definition of the source term functions $\phi^l_i$ is such that we may have bounded jumps at the time discretization points $l$ of the scheme. However, in the representation of the functions $v^{r,\rho,0,l}_i$ and $r^{l}_i$ the bounded source terms become integrated over time starting from $\tau=l-1$ at each time step $l$. This will ensure that the construction of the global solution function $\mathbf{v}^{r,\rho}$ and $\mathbf{r}^l$ are uniformly bounded H\"{o}lder -continuous over time. We shall see that this is sufficient for us. However it is also possible to construct global functions $\phi_i$ which equal $\phi^l_i$ on each domain $(l-1,l]\times {\mathbb R}^n$ which are differentiable from the beginning defining
\begin{equation}
\begin{array}{ll}
\phi^{v,l}_i:[l-1,l]\times {\mathbb R}^n\rightarrow {\mathbb R}\\
\\
\phi^{v,l}_i(\tau,x):=\sin^2\left(\pi(\tau-(l-1))\right)\phi^{v,s,l}_i(x),
\end{array}
\end{equation}
and 
\begin{equation}
\begin{array}{ll}
\phi^{r,l}_i:[l-1,l]\times {\mathbb R}^n\rightarrow {\mathbb R}\\
\\
\phi^{r,l}_i(\tau,x):=\sin^2\left(\pi(\tau-(l-1))\right)\phi^{r,s,l}_i(x),
\end{array}
\end{equation}
Then we get global bounded solution functions $\mathbf{v}^{r,\rho}$ and $\mathbf{r}$ which are differentiable across time points $\tau=l$. However the growth control of the local functions $\mathbf{v}^{r,\rho,l}$ and $\mathbf{r}^l$ becomes a little bit cumbersome. Hence we proceed with the definition which leads to global solutions which are bounded continuous at time points $\tau=l$ first.  
\end{rem}
We continue to outline the main ideas within item $\beta)$. We use the fact that we have some freedom in order to define the control functions $r^l_i$.
The functions $\phi^l_i$ serve as source terms or 'consumption terms' in the following parabolic equation for $\mathbf{r}^l$ ( which is a linearising of a time-local Navier-Stokes type equation with source term $\phi^l_i$ among other source terms), where for $1\leq i\leq n$ we have
   \begin{equation}\label{Navlerayrrholr++}
\left\lbrace \begin{array}{ll}
r^l_{i,\tau}-\rho_l\nu \Delta r^l_i+\rho_l\sum_{j=1}^n r^{l-1}_j(l-1,.)\frac{\partial r^{l}_i}{\partial x_j}=\\
\\
+\rho_l\int_{{\mathbb R}^n}\left( \frac{\partial}{\partial x_i}K_n(x-y)\right) \sum_{j,k=1}^n\left( \frac{\partial r^{l-1}_k}{\partial x_j}\frac{\partial r^{l-1}_j}{\partial x_k}\right) (l-1,y)dy\\
\\
-L^{\rho ,l,0}_i(\mathbf{r}^{l-1}(l-1,.);\mathbf{v}^{r,\rho,l-1}(l-1,.))\\
\\-\rho_l\int_{{\mathbb R}^n}\left( \frac{\partial}{\partial x_i}K_n(x-y)\right) \sum_{j,m=1}^n\left( \frac{\partial v^{r,\rho,l-1}_m}{\partial x_j}\frac{\partial v^{r,\rho,l-1}_j}{\partial x_m}\right) (l-1,y)dy+\phi^l_i\\
\\
r^l_i(l-1,.)=r^{l-1}_i(l-1,.),
\end{array}\right.
\end{equation}
and where
\begin{equation}
\begin{array}{ll}
L^{\rho ,l,0}_i(\mathbf{r}^{l-1};\mathbf{v}^{r,\rho,l-1})\equiv 
+\rho_l\sum_{j=1}^n r^{\rho,l-1}_j\frac{\partial v^{r,\rho,l-1}_i}{\partial x_j}+\rho_l\sum_{j=1}^n v^{r,\rho,l-1}_j\frac{\partial r^{l-1}_i}{\partial x_j}\\
\\ -2\rho_l\int_{{\mathbb R}^n}\left( \frac{\partial}{\partial x_i}K_n(x-y)\right) \sum_{j,k=1}^n\left( \frac{\partial r^{l-1}_k}{\partial x_j}\frac{\partial v^{r,\rho,l-1}_j}{\partial x_k}\right) (\tau,y)dy.
\end{array}
\end{equation}
Note the signs on the right side of (\ref{Navlerayrrholr++}). 
Note that dependence on the function $\mathbf{v}^{r,\rho}$ involves only the function $x\rightarrow \mathbf{v}^{r,\rho,l-1}(l-1,.)$ which is known at time step $l$.
Assume that this equation has a classical solution $\mathbf{r}^l=(r^l_1,\cdots ,r^l_n)$ in $C^{1,2}$ on the domain $(l-1,l]\times{\mathbb R}^n$ (we shall show existence of such classical solutions for time-local Navier-Stokes type systems below). Plugging (\ref{Navlerayrrholr++}) in (\ref{Navlerayrrhol+}) for $k=0$ (note that for $k=0$ $v^{r,\rho,k-1,l}_j=v^{r,\rho,-1,l}_j=v^{r,\rho,l-1}_j$) we get
the equation
\begin{equation}\label{Navlerayrrhol+++}
\left\lbrace \begin{array}{ll}
\frac{\partial v^{r,\rho,0,l}_i}{\partial \tau}-\rho_l\nu\sum_{j=1}^n \frac{\partial^2 v^{r,\rho,0,l}_i}{\partial x_j^2} 
+\rho_l\sum_{j=1}^n v^{r,\rho,l-1}_j(l-1,.)\frac{\partial v^{r,\rho,0,l}_i}{\partial x_j}=\phi^l_i+\delta^l_i,\\
\\
\mathbf{v}^{r,\rho,0,l}(l-1,.)=\mathbf{v}^{r,\rho,l-1}(l-1,.),
\end{array}\right.
\end{equation}
where we shall point out that $\delta^l_i$ is small compared to $\phi^l_i$ with respect to a $|.|_1$-norm (provided that $\rho_l$ is small enough). 

The choice of $\phi_l$ is such that in these relevant regions the behavior of the solutions $\mathbf{v}^{r,\rho,0,l}$ and $\mathbf{r}^{l}$ is controlled. Moreover, we know that the function  
\begin{equation}
\mathbf{v}^{r,\rho,l}=\mathbf{v}^{r,\rho,0,l}+\sum_{k\geq 1}\delta \mathbf{v}^{r,\rho,k,l},
\end{equation}
is small where the sum of functions
\begin{equation}
\delta \mathbf{v}^{r,\rho,k,l}=\mathbf{v}^{r,\rho,k,l}-\mathbf{v}^{r,\rho,k-1,l},
\end{equation}
is a 'perturbation' of the function $\mathbf{v}^{r,\rho,0,l}$ if $\rho_l$ is small. We may choose $\rho_l$ in such a way that a certain control of the behavior established for $\mathbf{v}^{r,\rho,0,l}$ is also valid for all functions $\mathbf{v}^{r,\rho,k,l},~k\geq 1$ and for the limit $\mathbf{v}^{r,\rho,l}$. Proceeding with time step numbers $l$ in such a way while the sum $\sum_{l=1}^m\rho_l$ of the time step numbers $\rho_l$ goes to infinity ensures boundedness of the function $\mathbf{v}^{r,\rho}$ and of its first spatial derivatives. Having obtained this you obtain global classical solution in each time interval which is H\"{o}lder continuous across the integers $l\geq 1$. Note that  boundedness of the function $\mathbf{v}^{r,\rho}$ does not guarantee boundedness of the difference $\mathbf{v}-\mathbf{r}=\mathbf{v}^{\rho}-\mathbf{r}$. We have to ensure that $\mathbf{r}$ is bounded, i.e., there is a global bound of all functions $\mathbf{r}^l$ independently of the time step number $l$. Furthermore we shall ensure that the control function $\mathbf{r}$ is globally H\"{o}lder contiuous. Then we may substract $\mathbf{r}$ from $\mathbf{v}^{r,\rho}$ and get a global H\"{o}lder continuous solution $\mathbf{v}^{\rho}$ to the time-transformed incompressible Navier-Stokes equation. These H\"{o}lder-coninuous functions figure as first order coefficients of a system which is linear in terms of these coefficients. Then classical theory of linear parbolic equations tells us that we have a global classical solution. 
We shall see later in detail that the construction outlined can be implemented such that the functions $\mathbf{v}^{r,\rho,l}$ and $\mathbf{r}^l$ are uniformly H\"{o}lder-bounded, i.e., there is a number $C>0$ such that $|v^{r,\rho,l}_i|_{\alpha}\leq C$ and $|r^l_i|_{\alpha}\leq C$ for all $l\geq 1$ and $1\leq i\leq n$, and such that the H\"{o}lder constant $\alpha$ is independent of the time-step number $l$, of course. The constant $C>0$ can be computed a priori and depends only on the dimension $n$, the viscosity $\nu >0$, and the initial data $\mathbf{h}$. Moreover, the sequence $\mathbf{r}^l$ defines a globally bounded function $\mathbf{r}:[0,\infty)\times {\mathbb R}^n$ which is (weakly) differentiable with respect to time and has spatial derivatives up to second order in a classical sense.
  The functions $\mathbf{v}^{r,\rho,l}:[l-1,l]\times {\mathbb R}^n\rightarrow {\mathbb R}^n,~l\geq 1$ which solve the equations (\ref{Navlerayrrhol}) for each $l$ with $\mathbf{v}^{\rho,r,l}(l-1,.)=\mathbf{v}^{\rho,r,l-1}(l-1,.)$ for $l\geq 2$ and  $\mathbf{v}^{\rho,r,l}(l-1,.)=\mathbf{h}(l-1,.)$ for $l=1$ define a global function
\begin{equation}
\mathbf{v}^{r,\rho}:=\mathbf{v}^{\rho}+\mathbf{r}:[0,\infty)\times {\mathbb R}^n\rightarrow {\mathbb R}^n,
\end{equation}
which equals $\mathbf{v}^{r,\rho,l}$ on each subdomain $[l-1,l]\times {\mathbb R}^n,~l\geq 1$. Since $\mathbf{r}$ is bounded this function $\mathbf{v}^{\rho,r}$ satisfies a system of equations equivalent to the Navier-Stokes equation system and a bounded global solution of the Navier-Stokes equation system is given by
\begin{equation}
\mathbf{v}(t,.):=\mathbf{v}^{r,\rho}(\tau,.)-\mathbf{r}(\tau,.).
\end{equation}
We have already pointed our that the regularity $\mathbf{v}\in \left[ C^{1,2}\left(\left[0,\infty\right) \times {\mathbb R}^n\right)\right]^n$ is easily obtained if $\mathbf{v}^{r,\rho}$ and $\mathbf{r}$ are globally H\"{o}lder continuous and have a certain decay at infinity. This leads to uniqueness.
\end{itemize}

However, before we get involved with the construction of the sequence $(\rho_l)$ and $(\mathbf{r}^l)$ in detail let us consider (\ref{Navleray}) again.
From the point of view that the incompressible Navier-Stokes equation is an extension of the multivariate Burgers equation a solution of the equation (\ref{Navleray}) is an element of the class of divergence-free solutions of a certain family of multivariate Burgers equations where for each function $\mathbf{f}$ we add a certain source term to (\ref{qparasyst1}) involving functions $\mathbf{f}\in\left[ C^{1,2}_b\left( \left[0,\infty\right)\times {\mathbb R}^n\right)\right]^n$. More precisely, consider the family $\left( \mathbf{v}^{f}\right)_{\mathbf{f}}$ where for each fixed $\mathbf{f}$ the function $\mathbf{v}^{f}$ satisfies the Cauchy  problem
\begin{equation}\label{Navleraymultburg}
\left\lbrace \begin{array}{ll}
\frac{\partial v^{f}_i}{\partial t}-\nu\sum_{j=1}^n \frac{\partial^2 v^{f}_i}{\partial x_j^2} 
+\sum_{j=1}^n v^{f}_j\frac{\partial v^{f}_i}{\partial x_j}=\\
\\ \hspace{3cm}\int_{{\mathbb R}^n}\left( \frac{\partial}{\partial x_i}K_n(x-y)\right) \sum_{j,k=1}^n\left( \frac{\partial f_k}{\partial x_j}\frac{\partial f_j}{\partial x_k}\right) (t,y)dy,\\
\\
\mathbf{v}^{f}(0,.)=\mathbf{h}.
\end{array}\right.
\end{equation}
Then a fixed point of the map (defined on regular divergence free vector fields with values in divergence free vector fields, all depending on time)
\begin{equation}
F_{\mbox{\tiny loc}}:\mathbf{f}\rightarrow \mathbf{v}^{f},~\mbox{where}~\div\mathbf{f}=0,~\mbox{and}~\div\mathbf{v}^{f}=0
\end{equation}
is a solution of the incompressible Navier-Stokes system. Let us denote a fixed point of the latter map by $\mathbf{v}^{*}$.
Formally, such a fixed point has a representation in terms of the fundamental solution $\Gamma^*$ of the scalar equation
\begin{equation}
\frac{\partial \Gamma^*}{\partial t}-\nu\sum_{j=1}^n \frac{\partial^2 \Gamma^*}{\partial x_j^2} 
+\sum_{j=1}^n v^{*}_j\frac{\partial \Gamma^*}{\partial x_j}=0,
\end{equation}
where we denote $\mathbf{v}^{*}=\left(v^{*}_1,\cdots ,v^{*}_n\right)^T$, i.e., we
have the formal representation ($1\leq i\leq n$)
\begin{equation}\label{navstokesburgrep}
\begin{array}{ll}
v^{*}_i(t,x)=\int_{{\mathbb R}^n}h_i(y)\Gamma^*(t,x,0,y)dy+\\
\\ \int_0^t\int_{{\mathbb R}^n}\left( \frac{\partial}{\partial x_i}K_n(y-z)\right) \sum_{j,k=1}^n\left( \frac{\partial v^{*}_k}{\partial x_j}\frac{\partial v^{*}_j}{\partial x_k}\right) (s,z)\Gamma^*(t,x,s,y)dzdyds .
\end{array}
\end{equation}
Note that for all $1\leq i\leq n$ we have the same fundamental solution $\Gamma^*$ of a scalar equation involving first order coefficients $v^{*}_j, 1\leq j\leq n$ which are the same for all $1\leq i\leq n$.
Well, if we construct such fixed points in a time-discretized scheme, then we may use the a priori estimates of the multivariate Burgers equation of type (\ref{apriorimax}) and a representation of type (\ref{navstokesburgrep}) at each time step in order analyze the growth of the solution. Note that we cannot control the integral terms involving the Poisson kernel in the Leray projection form this way. Nevertheless the representation (\ref{navstokesburgrep}) gives us a first hint how the integral term can be controlled in a time-discretized scheme using our considerations above. There is a mutual dependence of the  choice of the functions $r^l_{i}$ and the choice of the numbers $\rho_l$ which ensure the local convergence in time. The choice of the latter numbers depends on the local scheme. It is not necessary to start each time step with the solution of the corresponding multidimensional Burgers equation.
Indeed, an alternative way of constructing a divergence-free fixed point
\begin{equation}
F_{\mbox{\tiny glob}}:\mathbf{f}\rightarrow \mathbf{v}^f
\end{equation}
via a time-discretized scheme may be the following: locally in time the family $\left( \mathbf{v}^f\right)_{\mathbf{f}}$
of vector-valued functions  $\mathbf{v}^f=\left(v^f_1,\cdots ,v^f_n \right)^T$ may satisfy for each $\mathbf{f}=\left(f_1,\cdots ,f_n \right) $ the equation
\begin{equation}\label{Navleraymultburgf}
\left\lbrace \begin{array}{ll}
\frac{\partial v^f_i}{\partial t}-\nu\sum_{j=1}^n \frac{\partial^2 v^f_i}{\partial x_j^2} 
+\sum_{j=1}^n f_j\frac{\partial v^f_i}{\partial x_j}=\\
\\ \hspace{3cm}\int_{{\mathbb R}^n}\left( \frac{\partial}{\partial x_i}K_n(x-y)\right) \sum_{j,k=1}^n\left( \frac{\partial f_k}{\partial x_j}\frac{\partial v^f_j}{\partial x_k}\right) (t,y)dy,\\
\\
\mathbf{v}^f(0,.)=\mathbf{h}.
\end{array}\right.
\end{equation}
Here, we search for a fixed point solution $\mathbf{f}^*$ with $F_{\mbox{\tiny glob}}(\mathbf{f}^*)=\mathbf{v}^{f^*}$, where 
\begin{equation}
\begin{array}{ll}
\div \mathbf{f}^*=\div \mathbf{v}^{f^*}=0.
\end{array}
\end{equation}
Let us consider this map more closely. Start with the equation
\begin{equation}\label{navf}
    \frac{\partial\mathbf{v}^f}{\partial t}-\nu \Delta \mathbf{v}^f+ (\mathbf{f} \cdot \nabla) \mathbf{v}^f = - \nabla p^f~~~~t\geq 0,~~x\in{\mathbb R}^n,
\end{equation}
for some scalar function $p^f$, and where
\begin{equation}
\mathbf{f}\in \left\lbrace \mathbf{u}\in \left[ C^{1,2}_b\left( \left[0,\infty\right)\times {\mathbb R}^n\right)\right]^n|\div \mathbf{u}=0,~\mathbf{u}(0,.)=\mathbf{h}\right\rbrace .
\end{equation}
In the following we denote ordinary spatial derivatives in the form $f_{i,j}:=\frac{\partial f_i}{\partial x_j}$ and $f_{i,j,k}:=\frac{\partial f^2_i}{\partial x_j\partial x_k}$ etc. as is usual for example in the literature on the theory of general relativity. Sometimes we feel free to denote time derivatives in the form $f_{i,t}=\frac{\partial f_i}{\partial t}$. For the divergence
$\div \mathbf{v}^f$ we have 
\begin{equation}\label{diveq}
\frac{\partial}{\partial t}\div \mathbf{v}^f+\nu\Delta \div\mathbf{v}^f+\sum_j f_j\frac{\partial}{\partial x_j}\div \mathbf{v}^f=-\sum_{i,j=1}^n f_{i,j}v^f_{j,i}-\Delta p^f,
\end{equation}
where $\div \mathbf{v}^f(0,.)=\div \mathbf{h}=0$. Now let $\Gamma^f$ be the fundamental solution of
\begin{equation}
\frac{\partial}{\partial t}\Gamma^f-\nu\Delta \Gamma^f+\sum_{j=1}^n f_j\Gamma^f=0. 
\end{equation}
Then the solution to equation (\ref{diveq}) with zero initial data has the representation
\begin{equation}
\div \mathbf{v}^f(t,x)=\int_0^t\int_{{\mathbb R}^n}\left( -\sum_{i,j=1}^n f_{i,j}v^f_{j,i}-\Delta p^f\right)(s,y)\Gamma^f(t,x,s,y)dyds .
\end{equation}
Now let $\mathbf{v}^f$ be a solution of (\ref{Navleraymultburgf}). Then we have the representation
\begin{equation}\label{virep*+}
\begin{array}{ll}
v^f_i(t,x)=\int_{{\mathbb R}^n}h_i(y)\Gamma^f(t,x,0,y)dy\\
\\
+\int_0^t\int_{{\mathbb R}^n}\int_{{\mathbb R}^n}\left( \frac{\partial}{\partial x_i}K_n(x-z)\right) \sum_{j,k=1}^n\left( \frac{\partial f_k}{\partial x_j}\frac{\partial v^f_j}{\partial x_k}\right) (s,z)\Gamma^f(t,x,s,y)dydzds
\end{array}
\end{equation}
Hence if we solve (\ref{Navleraymultburgf}) for $\mathbf{v}^f$ in the form (\ref{virep*+}), then this is the same is solving (\ref{navf}) in the form
\begin{equation}\label{virep*++}
\begin{array}{ll}
v^f_i(t,x)=\int_{{\mathbb R}^n}h_i(y)\Gamma^f(t,x,0,y)dy\\
\\
+\int_0^t\int_{{\mathbb R}^n}\left( \frac{\partial}{\partial x_i}p\right) (s,y)\Gamma^f(t,x,s,y)dyds
\end{array}
\end{equation}
along with
\begin{equation}
p^f(t,x)=-\int_{{\mathbb R}^n}\left( K_n(x-y)\right) \sum_{j,k=1}^n\left( \frac{\partial f_k}{\partial x_j}\frac{\partial v^f_j}{\partial x_k}\right) (t,y)dy,
\end{equation}
and such that
\begin{equation}
\div \mathbf{v}^f=0
\end{equation}
is ensured. Hence, if we find a fixed point $\mathbf{v}^{*}$ of the map $\mathbf{f}\rightarrow \mathbf{v}^f$, then this function satisfies a) the equation (\ref{Navleraymultburgf}), and b) the equation (\ref{nav}) (with $\mathbf{f}_{ex}\equiv 0$) along with
\begin{equation}\label{pressurefix}
p(t,x)=-\int_{{\mathbb R}^n}\left( K_n(x-y)\right) \sum_{j,k=1}^n\left( \frac{\partial v^*_k}{\partial x_j}\frac{\partial v^*_j}{\partial x_k}\right) (t,y)dy,
\end{equation} 
and $\div \mathbf{v}^*=0$, and $\mathbf{v}^*(0,.)=\mathbf{h}(.)$. The latter ansatz has the advantage that we can easily preserve the incompressibility condition at each step of approximation of the local solution, i.e., the solution at each time step. The disadvantage is that we need a priori estimates for partial-integral linear differential equations, and these are a lttle more cumbersome.  Note that we speak of 'the' local solution here in the sense of our construction, i.e., the  approximating equations involved in each time step have a unique solution. Moreover, we shall prove that the local solution is a classical solution, and then it can be shown that the classical solution is indeed unique. In any case (whether we start with the multivariate Burgers equation or with a linearized Navier-Stokes equation), it seems that such a fixed point cannot be obtained by a global iteration scheme.

Recall that it is indeed sufficient to solve (\ref{Navleray}) (this is well-known). A regular solution $\mathbf{v}$ of equation (\ref{Navleray}) satisfies
\begin{equation}\label{Navleraydiv}
\left\lbrace \begin{array}{ll}
\frac{\partial v_i}{\partial t}-\nu\sum_{j=1}^n \frac{\partial^2 v_i}{\partial x_j^2} 
+\sum_{j=1}^n v_j\frac{\partial v_i}{\partial x_j}=-\frac{\partial}{\partial x_i}p
\\
\mathbf{v}(0,.)=\mathbf{h},
\end{array}\right.
\end{equation}
along with $p$ of form (\ref{pressurefix}). The associated equation for the divergence $\div\mathbf{v}$ has zero initial data and a source term
\begin{equation}
-\sum_{i,j=1}^n \frac{\partial v_j}{\partial x_i}\frac{\partial v_i}{\partial x_j}-\Delta_xp,
\end{equation}
which becomes zero, since
\begin{equation}\label{pressure}
\sum_{i,j=1}^n \left( \frac{\partial v_j}{\partial x_i}\frac{\partial v_i}{\partial x_j}\right) (t,x)=\Delta_x\int_{{\mathbb R}^n}K_n(x-y) \sum_{j,k=1}^n\left( \frac{\partial v_k}{\partial x_j}\frac{\partial v_j}{\partial x_k}\right) (t,y)dy
\end{equation}
for regular $\mathbf{v}$ with the indicated decay at infinity. Here, $\Delta_x$ denotes the Laplacian (where derivatives are with respect to the variables $x_i$), and where some regularity of the function $\mathbf{v}$ is sufficient. Note that higher order differentiability is easily obtained if a solution in $C^{1,2}_b$ (the space with a bounded time derivative of first order and bounded second order dervatives up to second order) has been obtained. You just differentiate the Navier-Stokes equations. You get some additional terms upon each step of differntiation, but you can treat them as source terms (as if they were known functions), because you know the function as a function of lower regularity (cf. Section 4 for details).

Back to the time discretization, the time step sizes are given by a series of positive real numbers $(\rho_l)_{l\in {\mathbb N}}$ where ${\mathbb N}$ denotes the natural numbers (starting with $l=1$) and the size will be chosen in such a way that the iteration at each time step $l$ converges and such that we have global convergence, i.e., the time step sizes are large enough. In order to ensure global convergence it is sufficient to have a time step size of order 
$\rho_l=\frac{c}{l}$
for a constant $c>0$ which depends only on the viscosity $\nu$, the dimension $n$ and the initial data $\mathbf{h}$. Prima facie it seems that our recursive construction of a bounded function $\mathbf{r}$ which equals the functions $r^l_i$ locally on $[l-1,l]\times {\mathbb R}^n$ leads us to the conclusion that a uniform lower bound for the time step size $\rho_l$ may be possible (note that the numbers $\rho_l$ are time step sizes from the point of view of the original time coordinates). However, we need to have a bound for the integral magnitude in (\ref{vnorm}) too, and this leads us to our  choice of a decreasing time step size. From a numerical point of view we hope that the opposite may be possible, i.e., that smoothing effects may allow to increase the time step size $\rho_l$ as time goes by, i.e., as the time step number $l$ increases. Once we have proved that there is a bounded classical solution we may discard the control function $\mathbf{r}$ and consider this possibility. We shall ensure that the global solution $\mathbf{v}^{r,\rho}$ (which equals $\mathbf{v}^{r,\rho,l}$ on each domain $[l-1,l]\times {\mathbb R}^n$) is bounded. Since $\mathbf{r}$ is bounded, this implies that we have a bounded global solution $\mathbf{v}^{\rho}=\mathbf{v}^{r,\rho}+\mathbf{r}$ (here, $\mathbf{v}^{\rho}$ equals $\mathbf{v}^{\rho ,l}$ on each domain $[l-1,l]\times {\mathbb R}^n$). Note that $\mathbf{v}(t,.)=\mathbf{v}^{\rho}(\tau,.)$ where the former function is a solution of the Navier Stokes equation. Note that we construct the function $\mathbf{r}$ time step by time step together with the function $\mathbf{v}^{r,\rho}$. Furthermore, note that $\mathbf{v}^{r}=\mathbf{v}^{r,\rho}$, where the first function refers to the equivalent system in the original time coordinate $t$. Accordingly $\mathbf{v}^{r}$ satisfies the equation
\begin{equation}\label{Navlerayrrhol****}
\left\lbrace \begin{array}{ll}
\frac{\partial v^{r}_i}{\partial t}-\nu\sum_{j=1}^n \frac{\partial^2 v^{r}_i}{\partial x_j^2} 
+\sum_{j=1}^n v^{r}_j\frac{\partial v^{r}_i}{\partial x_j}=L_i(\mathbf{r};\mathbf{v}^r)+\\
\\ \hspace{1cm}\rho_l\int_{{\mathbb R}^n}\left( \frac{\partial}{\partial x_i}K_n(x-y)\right) \sum_{j,k=1}^n\left( \frac{\partial v^{r}_k}{\partial x_j}\frac{\partial v^{r}_j}{\partial x_k}\right) (t,y)dy+r_{i,t},\\
\\
\mathbf{v}^{r}(0,.)=\mathbf{h}(0,.),
\end{array}\right.
\end{equation}
 where 
\begin{equation}
\begin{array}{ll}
L_i(\mathbf{r};\mathbf{v}^{r})\equiv
 -\nu \Delta r^l_i+\sum_{j=1}^n r_j\frac{\partial r_i}{\partial x_j}\\
\\
+\sum_{j=1}^n r_j\frac{\partial v^{r}_i}{\partial x_j}+\sum_{j=1}^n v^{r}_j\frac{\partial r_i}{\partial x_j}\\
\\ -2\int_{{\mathbb R}^n}\left( \frac{\partial}{\partial x_i}K_n(x-y)\right) \sum_{j,k=1}^n\left( \frac{\partial r_k}{\partial x_j}\frac{\partial v^{r}_j}{\partial x_k}\right) (t,y)dy\\
\\
-\int_{{\mathbb R}^n}\left( \frac{\partial}{\partial x_i}K_n(x-y)\right) \sum_{j,k=1}^n\left( \frac{\partial r_k}{\partial x_j}\frac{\partial r_j}{\partial x_k}\right) (t,y)dy\\
\\
=:\nu \Delta r_i+\sum_{j=1}^n r_j\frac{\partial r_i}{\partial x_j}\\
\\
\int_{{\mathbb R}^n}\left( \frac{\partial}{\partial x_i}K_n(x-y)\right) \sum_{j,k=1}^n\left( \frac{\partial r_k}{\partial x_j}\frac{\partial r_j}{\partial x_k}\right) (t,y)dy+L^{0}_i(\mathbf{r};\mathbf{v}^{r}).
\end{array}
\end{equation}
Our solution scheme computes $\mathbf{v}^{r}$ and $\mathbf{r}$ (in original coordinates) simultaneously, and 
a global solution of the Navier-Stokes equation system is then obtained by addition $\mathbf{v}=\mathbf{v}^r+\mathbf{r}$ (we use the same symbol $\mathbf{r}$ for simplicity). Note also that $\mathbf{r}(t,.)=\mathbf{r}(\tau,.)$, hence $\mathbf{v}$ is bounded since $\mathbf{v}^r$ is bounded and $\mathbf{r}$ is bounded. Note that on the domain $\left[\sum_{m=1}^{l-1}\rho_m,\sum_{m=1}^l\rho_m \right]\times {\mathbb R}^n$ the solution of the Navier-Stokes equation $\mathbf{v}$ equals $\mathbf{v}^{\rho,l}$ on the domain $[l-1,l]\times {\mathbb R}^n$. As already remarked dependence of the numbers $\rho_l$ on the time step number $l$ is due to infinite domains from an analytic point of view and may also be due to numerical purposes (with opposite consequences for the time step size) because the scheme with flexible time step may take advantage of the smoothing effect of the scalar densities involved in our scheme (implying larger time step sizes as time goes by (cf. remarks in section 5)).
\begin{rem}
Note that in \cite{K} we used a scheme with time step sizes  
\begin{equation}
\rho_l=\frac{1}{C^*_nl},
\end{equation}
in order to solve the multidimensional Burgers equation,
where $C^*_n>0$ is a constant which does not depend on the time step number $l$ (it depends only on the data $\mathbf{h}$, and the dimension $n$ , and the viscosity constant $\nu >0$ and will be determined below). The scheme considered there is not sufficient in order to prove global existence of solutions for the incompressible Navier-Stokes equation since we have an additional source term which is quadratic with respect to the gradient of the velocity. The introduction of the functions $r^l_i$ for $l\geq 1$ and $1\leq i\leq n$ is the crucial difference which makes it possible to control the source terms in the Leray projection form of the incompressible Navier-Stokes equation.
\end{rem}
Let us have a closer look at the time transformations $t\rightarrow \tau$, and the introduction of the function $\mathbf{r}$. 
We consider the Cauchy problem (\ref{Navleray}) on the domain $D=\left[0,\infty\right) \times{\mathbb R}^n$ where our interest is in the case $n\geq 3$. We shall solve the Cauchy problem in subsequent time steps $l\geq 1$ on the domains
\begin{equation}
D_l=[T_{l-1},T_{l}]\times {\mathbb R}^n,
\end{equation}
where $T_0=0$ and $T_{l}=T_{l-1}+\rho_l$ for $l\geq 1$. Instead of considering a scheme with small time step size in original coordinates we may consider a equidistant scheme in transformed coordinates with time step size $1$, i.e., 
\begin{equation}\label{timeder}
\tau\rightarrow t_l(\tau)=\rho_l\tau~~\mbox{if}~~\tau\in [l-1,l].
\end{equation}
The transformed domains are denoted by $D^{\tau}_l=[l-1,l]\times {\mathbb R}^n$. The original time coordinate $t$ has an index here in (\ref{timeder}) in order to indicate that we are actually considering infinitely many different time transformations (one for each time step number $l$). However, in order to keep notation simple we drop such indices in the following. The index $l$ makes clear which domain we consider. On each domain $D^{\tau}_l$ we have a transformed Cauchy problem
\begin{equation}\label{Navlerayl}
\left\lbrace \begin{array}{ll}
\frac{\partial v^{\rho,l}_i}{\partial \tau}-\rho_l\nu\sum_{j=1}^n \frac{\partial^2 v^{\rho,l}_i}{\partial x_j^2} 
+\rho_l\sum_{j=1}^n v^{\rho,l}_j\frac{\partial v^{\rho,l}_i}{\partial x_j}=\\
\\ \hspace{3cm}\rho_l\int_{{\mathbb R}^n}\left( \frac{\partial}{\partial x_i}K_n(x-y)\right) \sum_{j,k=1}^n\left( \frac{\partial v^{\rho,l}_k}{\partial x_j}\frac{\partial v^{\rho,l}_j}{\partial x_k}\right) (\tau,y)dy,\\
\\
\mathbf{v}^{\rho,l}(l-1,.)=\mathbf{v}^{\rho,l-1}(l-1,.).
\end{array}\right.
\end{equation}
Hence the initial data of the Cauchy problem for $\mathbf{v}^{\rho,l}$ on the domain $D^{\tau}_l$ (the $l$th time step) are the final data of the Cauchy problem for $\mathbf{v}^{\rho,l-1}$ on the domain $D^{\tau}_{l-1}$. Note that $\mathbf{v}^{\rho,1}(0,.)=\mathbf{h}$.
However, as indicated above we introduce another sequence of real functions $(\mathbf{r}^l)$ and consider for each time step $l\geq 1$ and on the domain $D^{\tau}_l$ the function
\begin{equation}
(\tau,x)\rightarrow \mathbf{v}^{r,\rho,l}(\tau,x):=\mathbf{v}^{\rho,l}(\tau,x)+\mathbf{r}_l.
\end{equation}
If for time step number $l$ the function $\mathbf{v}^{\rho,l}$ solves (\ref{Navlerayl}), then   the function $\mathbf{v}^{r,\rho,l}$ solves the equation
\begin{equation}\label{Navleraylr}
\left\lbrace \begin{array}{ll}
\frac{\partial v^{r,\rho,l}_i}{\partial \tau}-\rho_l\nu\sum_{j=1}^n \frac{\partial^2 v^{r,\rho,l}_i}{\partial x_j^2} 
+\rho_l\sum_{j=1}^n v^{r,\rho,l}_j\frac{\partial v^{r,\rho,l}_i}{\partial x_j}=L^{\rho ,l}_i(\mathbf{r}^l,\mathbf{v}^{r,\rho,l})+\\
\\ \hspace{1cm}\rho_l\int_{{\mathbb R}^n}\left( \frac{\partial}{\partial x_i}K_n(x-y)\right) \sum_{j,k=1}^n\left( \frac{\partial v^{r,\rho,l}_k}{\partial x_j}\frac{\partial v^{r,\rho,l}_j}{\partial x_k}\right) (\tau,y)dy+r^l_{i,\tau},\\
\\
\mathbf{v}^{r,\rho,l}(l-1,.)=\mathbf{v}^{r,\rho,l-1}(l-1,.).
\end{array}\right.
\end{equation}
The sequence $(\mathbf{r}^l)$ will be chosen (recursively with respect to $l$) in such a way that the global growth is controlled, i.e., that the components of $\mathbf{v}^{r,\rho,l}$ are bounded and have bounded first orde spatial derivatives independent of $l$. Furthermore, since $\mathbf{r}$ is bounded and H\"{o}lder continuous  this implies a bound  and H\"{o}lder continuity of $\mathbf{v}^{\rho,l}$. Note that the global H\"{o}lder contuity then implies the existence of a global classical solution. The sequence $(\rho_l)$ is chosen in such a way that for each time step number $l$ the local scheme for $\mathbf{v}^{r,\rho,l}$ converges.
At each time step $l$ we approximate $\mathbf{v}^{r,\rho,l}$ iteratively by functions  $\mathbf{v}^{r,\rho,k,l}$ for $k\geq 0$. Indeed, at each time step $l$ we shall construct a solution $\mathbf{v}^{r,\rho,l}$ of (\ref{Navleraylr}) in form of a functional series
\begin{equation}
\mathbf{v}^{r,\rho,l}=\mathbf{v}^{r,\rho,0,l}+\sum_{k\geq 1}\delta \mathbf{v}^{r,\rho,k,l},
\end{equation}
where for each $k\geq 1$ we shall have a contraction of the successive approximations
\begin{equation}
 \mathbf{v}^{r,\rho,k,l}=\mathbf{v}^{r,\rho,0,l}+\sum_{m= 1}^k\delta \mathbf{v}^{r,\rho,m,l},
\end{equation}
where
\begin{equation}\label{contract0}
|\delta \mathbf{v}^{r,\rho,k,l}|=|\mathbf{v}^{r,\rho,k,l}-\mathbf{v}^{r,\rho,k-1,l}|\leq c |\mathbf{v}^{r,\rho,k-1,l}-\mathbf{v}^{r,\rho,k-2,l}|.
\end{equation}
Here the symbol $|.|$ represents some appropriate norm and we have a contraction constant $c\in (0,1)$ which turns out to be independent of the time step number $l$. In general we shall prove a contraction property of type (\ref{contract0}) for specific series $(\mathbf{v}^{\rho,k,l})_k$ defined  as in (iii)  below with a specific first element $\mathbf{v}^{\rho,0,l}$ for each time step $l\geq 1$ (alternatively, we could use the methods in  (i) or (ii)) below to set up a local scheme). 
We have (at least) three possibilities for the local iteration (given $(\mathbf{r}^l)$ which is computed first at each step).
\begin{itemize}
 \item[(i)] We start for each $l$ the iteration with the corresponding multivariate Burgers equation, i.e. for each $l$ the function $\mathbf{v}^{r,\rho,0,l}$ solves the equation
\begin{equation}\label{Navlerayinitl*}
\left\lbrace \begin{array}{ll}
\frac{\partial v^{r,\rho,0,l}_i}{\partial \tau}-\rho_l\nu\sum_{j=1}^n \frac{\partial^2 v^{r,\rho,0,l}_i}{\partial x_j^2} 
+\rho_l\sum_{j=1}^n v^{r,\rho,0,l}_j\frac{\partial v^{r,\rho,0,l}_i}{\partial x_j}\\
\\
=L^{\rho ,l}_i\left(\mathbf{r}^{l}(l-1,.),\mathbf{v}^{r,\rho,l-1}(l-1,.)\right)
 +r^{l}_{i,\tau}(l-1,.)+ \\
 \\
 \rho_l\int_{{\mathbb R}^n}\left( \frac{\partial}{\partial x_i}K_n(x-y)\right) \sum_{j,k=1}^n\left( \frac{\partial v^{r,\rho,l-1}_k}{\partial x_j}\frac{\partial v^{r,\rho,l-1}_j}{\partial x_k}\right) (l-1,y)dy,\\
\\
\mathbf{v}^{r,\rho,0,l}(l-1,.)=\mathbf{v}^{r,\rho,l-1}(l-1,.).
\end{array}\right.
\end{equation}
We shall see below in which sense we have time derivatives at the points where $\tau=l-1$ for some $l\geq 1$.
For the latter equation we have the a priori estimates (\ref{apriorimax}), and hence a global regular solution which is bounded by the respective initial data. The initial data $\mathbf{v}^{r,\rho,l-1}(l-1,.)$ are the final data of the previous time step which are the result of an iteration which we have to define next. We define a series of multivariate Burgers equations for $k\geq 1$ where $\mathbf{v}^{r,\rho,k,l}$ solves
\begin{equation}\label{Navleraylk}
\left\lbrace \begin{array}{ll}
\frac{\partial v^{r,\rho,k,l}_i}{\partial \tau}-\rho_l\nu\sum_{j=1}^n \frac{\partial^2 v^{r,\rho,k,l}_i}{\partial x_j^2} 
+\rho_l\sum_{j=1}^n v^{r,\rho,k,l}_j\frac{\partial v^{r,\rho,k,l}_i}{\partial x_j}\\
\\
=L^{\rho ,l}_i\left(\mathbf{r}^l,\mathbf{v}^{r,\rho,k-1,l}\right) +r^{l}_{i,\tau}\\
\\
+ \rho_l\int_{{\mathbb R}^n}\left( \frac{\partial}{\partial x_i}K_n(x-y)\right) \sum_{j,k=1}^n\left( \frac{\partial v^{r,\rho,k-1,l}_k}{\partial x_j}\frac{\partial v^{r,\rho,k-1,l}_j}{\partial x_k}\right) (\tau,y)dy,\\
\\
\mathbf{v}^{r,\rho,k,l}(l-1,.)=\mathbf{v}^{r,\rho,l-1}(l-1,.).
\end{array}\right.
\end{equation}
Then we may prove - with an appropriate choice of $\rho_l, r_l$-that locally this scheme leads to a fixed point.

\item[(ii)] We may linearize for each time step the equation (\ref{Navleraylr}) and define a sequence of function $\mathbf{v}^{r,\rho,k,l}$ which are solutions of integro-partial-differential equations, i.e., we start with the solution $\mathbf{v}^{r,\rho,0,l}$ of
\begin{equation}\label{Navlerayl0int}
\left\lbrace \begin{array}{ll}
\frac{\partial v^{r,\rho,0,l}_i}{\partial \tau}-\rho_l\nu\sum_{j=1}^n \frac{\partial^2 v^{r,\rho,0,l}_i}{\partial x_j^2} 
+\rho_l\sum_{j=1}^n v^{r,\rho,l-1}_j\frac{\partial v^{r,\rho,0,l}_i}{\partial x_j}\\
\\
=L^{\rho ,l}_i\left(\mathbf{r}^l,\mathbf{v}^{r,\rho,0,l}\right)\\
\\
 +r^l_{i,\tau}+ \rho_l\int_{{\mathbb R}^n}\left( \frac{\partial}{\partial x_i}K_n(x-y)\right) \sum_{j,k=1}^n\left( \frac{\partial v^{r,\rho,l-1}_k}{\partial x_j}\frac{\partial v^{r,\rho,l-1}_j}{\partial x_k}\right) (\tau,y)dy,\\
\\
\mathbf{v}^{r,\rho,0,l}(l-1,.)=\mathbf{v}^{r,\rho,l-1}(l-1,.).
\end{array}\right.
\end{equation}
Then for $k\geq 1$ we define recursively functions $\mathbf{v}^{\rho,k,l}$ to be solutions of
\begin{equation}\label{Navlerayl0int}
\left\lbrace \begin{array}{ll}
\frac{\partial v^{r,\rho,k,l}_i}{\partial \tau}-\rho_l\nu\sum_{j=1}^n \frac{\partial^2 v^{r,\rho,k,l}_i}{\partial x_j^2} 
+\rho_l\sum_{j=1}^n v^{r,\rho,k-1,l}_j\frac{\partial v^{r,\rho,k,l}_i}{\partial x_j}\\
\\
=L^{\rho ,l}_i\left(\mathbf{r}^l,\mathbf{v}^{r,\rho,k-1,l}\right) +r^l_{i,\tau}\\
\\
+ \rho_l\int_{{\mathbb R}^n}\left( \frac{\partial}{\partial x_i}K_n(x-y)\right) \sum_{j,k=1}^n\left( \frac{\partial v^{r,\rho,k-1,l}_k}{\partial x_j}\frac{\partial v^{r,\rho,k,l}_j}{\partial x_k}\right) (\tau,y)dy,\\
\\
\mathbf{v}^{r,\rho,k,l}(l-1,.)=\mathbf{v}^{r,\rho,l-1}(l-1,.).
\end{array}\right.
\end{equation}
Note that this is an iteration of global equations, i.e., linear partial integro-differential equations.
\item[(iii)] Alternatively, we may define a sequence of functions $\mathbf{v}^{r,\rho,k,l}$ which solve linear parabolic partial differential equations. We may choose an iteration starting with $\mathbf{v}^{r,\rho,0,l}$ which solves
\begin{equation}\label{Navlerayl0int}
\left\lbrace \begin{array}{ll}
\frac{\partial v^{r,\rho,0,l}_i}{\partial \tau}-\rho_l\nu\sum_{j=1}^n \frac{\partial^2 v^{r,\rho,0,l}_i}{\partial x_j^2} 
+\rho_l\sum_{j=1}^n v^{r,\rho,l-1}_j\frac{\partial v^{r,\rho,0,l}_i}{\partial x_j}\\
\\
=L^{\rho ,l}_i\left(\mathbf{r}^l,\mathbf{v}^{r,\rho,l-1}\right) +r^l_{i,\tau}\\
\\
+ \rho_l\int_{{\mathbb R}^n}\left( \frac{\partial}{\partial x_i}K_n(x-y)\right) \sum_{j,k=1}^n\left( \frac{\partial v^{r,\rho,l-1}_k}{\partial x_j}\frac{\partial v^{r,\rho,l-1}_j}{\partial x_k}\right) (\tau,y)dy\\
\\
\mathbf{v}^{r,\rho,0,l}(l-1,.)=\mathbf{v}^{r,\rho,l-1}(l-1,.).
\end{array}\right.
\end{equation}
Then for $k\geq 1$ we define recursively functions $\mathbf{v}^{r,\rho,k,l}$ to be solutions of
\begin{equation}\label{Navlerayl0int}
\left\lbrace \begin{array}{ll}
\frac{\partial v^{r,\rho,k,l}_i}{\partial \tau}-\rho_l\nu\sum_{j=1}^n \frac{\partial^2 v^{r,\rho,k,l}_i}{\partial x_j^2} 
+\rho_l\sum_{j=1}^n v^{r,\rho,k-1,l}_j\frac{\partial v^{r,\rho,k-1,l}_i}{\partial x_j}\\
\\
=L^{\rho ,l}_i\left(\mathbf{r}^l,\mathbf{v}^{r,\rho,k-1,l}\right)\\
\\ +r^l_{i,\tau}+ \rho_l\int_{{\mathbb R}^n}\left( \frac{\partial}{\partial x_i}K_n(x-y)\right) \sum_{j,k=1}^n\left( \frac{\partial v^{r,\rho,k-1,l}_k}{\partial x_j}\frac{\partial v^{r,\rho,k-1,l}_j}{\partial x_k}\right) (\tau,y)dy,\\
\\
\mathbf{v}^{r,\rho,k,l}(l-1,.)=\mathbf{v}^{r,\rho,l-1}(l-1,.).
\end{array}\right.
\end{equation}
\end{itemize}
 Let us discuss the different advantages/disadvantages of the three approaches for the local solution at each time step $l\geq 1$.
\begin{itemize}
 \item[ad(i)] The solution of the multivariate Burgers equation is closer to the solution of the Navier-Stokes equation than the linear approximations in (ii) and (iii). However, at each substep $k$ of the $l$th time step the function $\mathbf{v}^{\rho,k,l}$ may not be divergence-free (a divergence free vector field is obtained in the limit). We have a priori estimates (essentially a maximum principle) for the multivariate Burgers equation. On the other hand from a numerical point of view linearized equation are preferable in iteration schemes.
 \item[ad(ii)] This iteration has the advantage that we can ensure that at each time step $l\geq 1$ each approximation we can ensure that $\mathbf{v}^{\rho,k,l}$ is divergence free, i.e, $\div \mathbf{v}^{\rho,k,l}=0$ for all $k\geq 0$. However, at each substep $k$ of a time step $l$ we have to solve linear partial integro-differential equations which may be complicated from a numerical point of view.
 \item[ad(iii)] In this case we have to solve local scalar linear parabolic equations in order to determine $\mathbf{v}^{\rho,k,l}$ at each substep $k$ of a time step $l$. From a numerical point of view this is interesting since we have good approximations of solutions for the involved parabolic equations where the second order terms form a Laplacian (cf. (\cite{KKS}), \cite{FrKa}). On the other hand, the local iteration does not take place in a space of divergence-free vector fields in general, i.e. a divergence free vector field is obtained in the limit. 
\end{itemize}

In this paper we choose the third alternative approach for the local solutions and construct for each time step $l$ a fixed point in some appropriate function space of maps on domains $D^{\tau}_l$. This allows us to define the most efficient algorithm among the three alternatives. Compared to standard discretization schemes all three versions of our analytical  scheme have the advantage that there is no spatial discretization. Next let us consider the local iteration in the form of the third item (iii) above in more detail. At each time step $l$ we consider maps of the form
\begin{equation}
\mathbf{f}\rightarrow \mathbf{v}^{r,f,\rho,l}=F^r_l(\mathbf{f}),
\end{equation}
where $\mathbf{f}=\left(f_1,\cdots ,f_n\right)^T$, and where $\mathbf{v}^{r,f,\rho,l}=\left(v^{r,f,\rho,l}_1,\cdots , v^{r,f,\rho,l}_n\right)^T$ satisfies the equation
\begin{equation}\label{Navlerayiterl}
\left\lbrace \begin{array}{ll}
\frac{\partial v^{r,f,\rho,l}_i}{\partial \tau}-\rho_l\nu\sum_{j=1}^n \frac{\partial^2 v^{r,f,\rho,l}_i}{\partial x_j^2} 
+\rho_l\sum_{j=1}^n f_j\frac{\partial v^{r,f,\rho,l}_i}{\partial x_j}=L^{\rho ,l,f}_i\left(\mathbf{r},\mathbf{v} \right)+ \\
\\ \hspace{1cm}\rho_l\int_{{\mathbb R}^n}\left( \frac{\partial}{\partial x_i}K_n(x-y)\right) \sum_{j,k=1}^n\left( \frac{\partial f_k}{\partial x_j}\frac{\partial f_j}{\partial x_k}\right) (\tau,y)dy+r^l_{i,\tau},\\\\
\\
\mathbf{v}^{r,f,\rho,l}(l-1,.)=\mathbf{v}^{r,\rho,l-1}(l-1,.).
\end{array}\right.
\end{equation}
Here we assume that we have solved the Navier-Stokes equation in transformed coordinates for $\tau\leq l-1$ and the initial data at $l-1$ are given by the solution at this time. The domain of the map $F^r_l$ is 
\begin{equation}
D_{F^r_l}:=\left\lbrace \mathbf{f}\in \left[ C^{1,2}_b\left(\left[(l-1),l\right]\times {\mathbb R}^n\right)\right]^n|\mathbf{f}(l-1,.)=\mathbf{v}^{r,\rho,l-1}(l-1,.) \right\rbrace .
\end{equation}
In case $l=1$ we have $\mathbf{v}^{r,f,\rho,1}(0,.)=\mathbf{h}$, of course. 
With the appropriate choice of $\rho_l$ and $r^l_i$ a local scheme for $\mathbf{v}^{r,\rho,l}$ may be defined in terms of a functional series $(\mathbf{v}^{r,\rho,k,l})_k$ with $\lim_{k\uparrow \infty} \mathbf{v}^{r,\rho,k,l}=\mathbf{v}^{r,\rho,l}$. We start the iteration determing for $1\leq i\leq n$ the functions $\phi^l_i$ and then the functions $r^l_i$ (solving a certain equation as scetched in $\beta$ above). Then for each $k\geq 0$ we have 
\begin{equation}\label{Navlerayl0int*}
\left\lbrace \begin{array}{ll}
\frac{\partial v^{r,\rho,k,l}_i}{\partial \tau}-\rho_l\nu\sum_{j=1}^n \frac{\partial^2 v^{r,\rho,k,l}_i}{\partial x_j^2} 
+\rho_l\sum_{j=1}^n v^{r,\rho,k-1,l}_j\frac{\partial v^{r,\rho,k,l}_i}{\partial x_j}=L^{\rho ,l}_i\left(\mathbf{r}^l,\mathbf{v}^{r,\rho,k-1,l} \right)+\\
\\ \rho_l\int_{{\mathbb R}^n}\left( \frac{\partial}{\partial x_i}K_n(x-y)\right) \sum_{j,m=1}^n\left( \frac{\partial v^{r,\rho,k-1,l}_m}{\partial x_j}\frac{\partial v^{r,\rho,k-1,l}_j}{\partial x_m}\right) (\tau,y)dy+r^l_{i,\tau},\\
\\
\mathbf{v}^{r,\rho,l}(l-1,.)=\mathbf{v}^{r,\rho,l-1}(l-1,.),
\end{array}\right.
\end{equation} 
along with $\mathbf{v}^{r,\rho,-1,l}=\mathbf{v}^{r,\rho,l-1}$. 
Let us look at the difference of two successive approximations $\mathbf{v}^{r,\rho,k-1,l}$ and $\mathbf{v}^{r,\rho,k,l}$ where we consider fixed functions $\mathbf{f}$ and $\mathbf{g}$ instead of $\mathbf{v}^{r,\rho,k-1,l}$ etc. as coefficient functions.
Comparing $\mathbf{v}^{f,\rho,l}$ with $\mathbf{v}^{g,\rho,l}$ leads us to an expression for the difference which we denote by
\begin{equation}
\delta \mathbf{v}^{r,f,g,\rho,l}:=\mathbf{v}^{r,f,\rho,l}-\mathbf{v}^{r,g,\rho,l}.
\end{equation}
This function satisfies the equation
\begin{equation}\label{Navlerayiterl**}
\left\lbrace \begin{array}{ll}
\frac{\partial \delta v^{r,f,g,\rho,l}_i}{\partial \tau}-\rho_l\nu\sum_{j=1}^n \frac{\partial^2 \delta v^{r,f,g,\rho,l}_i}{\partial x_j^2} 
+\rho_l\sum_{j=1}^n f_j\frac{\partial v^{r,f,g,\rho,l}_i}{\partial x_j}\\
\\
\hspace{2cm}=L^{\rho ,l,f,g,0}_i\left(\mathbf{r}^l,\delta\mathbf{v}^{r,f,g,\rho,l} \right)+\\
\\
\hspace{2cm}-\rho_l\sum_{j=1}^n\left( f_j-g_j\right)\frac{\partial v^{r,g,\rho,l}_i}{\partial x_j}\\
\\ \hspace{2cm}+\rho_l\int_{{\mathbb R}^n}\left( \frac{\partial}{\partial x_i}K_n(x-y)\right) \sum_{j,k=1}^n\left( \frac{\partial f_k}{\partial x_j}\frac{\partial f_j}{\partial x_k}\right) (\tau,y)dy,\\
\\
\hspace{2cm}-\rho_l\int_{{\mathbb R}^n}\left( \frac{\partial}{\partial x_i}K_n(x-y)\right) \sum_{j,k=1}^n\left( \frac{\partial g_k}{\partial x_j}\frac{\partial g_j}{\partial x_k}\right) (\tau,y)dy,
\\
\\
\delta\mathbf{v}^{r,f,g,\rho,l}(l-1,.)=0.
\end{array}\right.
\end{equation}
Here we have 
\begin{equation}
\begin{array}{ll}
L^{\rho ,l,f,g,0}_i\left(\mathbf{r}^l,\delta\mathbf{v}^{r,f,g,\rho,l} \right):=\\
\\
L^{\rho ,l,f,0}_i\left(\mathbf{r}^l,\mathbf{v}^{r,f,\rho,l} \right)
-L^{\rho ,l,g,0}_i\left(\mathbf{r}^l,\mathbf{v}^{r,g,\rho,l} \right)
\end{array}
\end{equation}

Note that the source term $r^l_{i,\tau}$ does not appear on the right side of this equation for the difference, furthermore all the terms involving only $r^l_i$ and the derivatives of $r^l_i$ (because they cancel out).
Hence the first term $\mathbf{v}^{r,\rho ,0,l}$ of the functional series $\mathbf{v}^{r,\rho,k,l}$ should contain the essential information concerning the growth of the solution at time step $l$ and the higher order terms of the local iteration. In the first iterative time step of an iteration scheme related to (\ref{Navlerayiterl**}) we shall have $f_j=v^{r,\rho,1,l}_j$ and  $g_j=v^{r,\rho,0,l}_j$, such that the first term on the right side related to (\ref{Navlerayiterl**}) becomes  
\begin{equation}
-\rho_l\sum_{j=1}^n\left( f_j-g_j\right)\frac{\partial v^{r,g,\rho,l}_i}{\partial x_j}=-\rho_l\sum_{j=1}^n\left( v^{r,\rho,1,l}_j-v^{r,\rho,0,l}_j\right)\frac{\partial v^{r,\rho,0,l}_i}{\partial x_j}.
\end{equation}
In \cite{K} we considered a bound of $v^{\rho,0,l}_i$ and $v^{\rho,0,1}_{i,j}$ of form $C^*C_l$ where $C_l$ depends linearly on $l$ and $C^*$ is a constant independent of the time step number $l$. Linear dependence of $C_l$ of both terms with respect to $l$ is sufficient in order to make our scheme global as the sum of time step sizes $\sum_{l\geq 1}\rho_l$ is unbounded. However, in this paper we shall construct a uniform bound $C_1C^*$ for $v^{r,\rho,0,l}_{i,j}$ which is independent of the time step number $l$. Let us stick to this second term of right side of (\ref{Navlerayiterl**}) for a moment. 
In the $k$-th iteration step the contribution to
the first term on the right side related to (\ref{Navlerayiterl**}) becomes  
\begin{equation}
\begin{array}{ll}
-\rho_l\sum_{j=1}^n\left( v^{r,\rho,k,l}_j-v^{r,\rho,k-1,l}_j\right)\frac{\partial v^{r,\rho,k-1,l}_i}{\partial x_j}=\\
\\
-\rho_l\sum_{j=1}^n\left( v^{r,\rho,k,l}_j-v^{r,\rho,k-1,l}_j\right)\left( \frac{\partial v^{r,\rho,0,l}_i}{\partial x_j}+\sum_{m=1}^{k-1}\frac{\partial}{\partial x_j}\delta v^{r,\rho,m,l}_i\right),
\end{array}
\end{equation} 
such that we may estimate these terms by using a contraction property of the differences $\delta \mathbf{v}^{r,\rho,k,l}$ which we shall observe for the specific series 
\begin{equation}
\begin{array}{ll}
\mathbf{v}^{r,\rho,k,l}=\mathbf{v}^{r,\rho,0,l}+\sum_{m=1}^k\delta\mathbf{v}^{r,\rho,m,l}=\\
\\
\mathbf{v}^{r,\rho,0,l}
+\sum_{m=1}^k\left( \mathbf{v}^{r,\rho,m,l}-\mathbf{v}^{r,\rho,m-1,l}\right).
\end{array}
\end{equation}
Now consider again the equation (\ref{Navlerayiterl**}). Let us consider $r^l_i=0$ for $1\leq i\leq n$, since we have no source terms $r^l_i$ in the higher correction terms for $\delta \mathbf{v}^{r,\rho,k,l}$ (in the proof of the main theorem below we shall see the additional terms $L^{\rho ,l,f,g,0}_i\left(\mathbf{r},\delta\mathbf{v}^{r,\rho,k,l} \right)$ with $\mathbf{f}=\mathbf{v}^{\rho,k,l}$ and $\mathbf{g}=\mathbf{v}^{\rho,k,l}$ do not alter the reasoning of time local convergence essentially). Note that for $r^l_i=0$ we denote
\begin{equation}
\mathbf{v}^{\rho,k,l}=\mathbf{v}^{0,\rho,k,l}=\mathbf{v}^{r,\rho,k,l},
\end{equation}
and similar for all components $v^{\rho,k,l}_i$.
Since $\mathbf{f}\in D_{F_l}$ we have that for each $1\leq i\leq n$ the fundamental solution $\Gamma^l_f$ of the scalar equation 
\begin{equation}\label{fundf}
\frac{\partial \Gamma^l_f}{\partial \tau}-\rho_l\nu\sum_{j=1}^n \frac{\partial^2 \Gamma^l_f}{\partial x_j^2} 
+\rho_l\sum_{j=1}^n f_j\frac{\partial \Gamma^l_f}{\partial x_j}=0
\end{equation}
exists. Similar, since $\mathbf{g}\in D_{F_l}$ we have that for each $1\leq i\leq n$ the fundamental solution $\Gamma^l_g$ of the scalar equation 
\begin{equation}
\frac{\partial \Gamma^l_g}{\partial \tau}-\rho_l\nu\sum_{j=1}^n \frac{\partial^2 \Gamma^l_g}{\partial x_j^2} 
+\rho_l\sum_{j=1}^n g_j\frac{\partial \Gamma^l_g}{\partial x_j}=0
\end{equation}
exists. Then formally we may represent the solution of the equation (\ref{Navlerayiterl}) in terms of the this fundamental solution (recall that we consider $\mathbf{r}=0$ for simplicity). Since $\mathbf{v}^{f,\rho,l}$ and $\mathbf{v}^{g,\rho,l}$ have the same initial data we have for $1\leq i\leq n$ (recall that ${\mathbf r}\equiv 0$ this time)
\begin{equation}\label{burgfg}
\begin{array}{ll}
v^{f,\rho,l}_i(\tau,x)-v^{g,\rho,l}_i(\tau,x)=\\
\\
-\int_0^{\tau}\int_{{\mathbb R}^n}\sum_{j=1}^n\left( f_j-g_j\right)(s,y)\frac{\partial v^{g,\rho,l}_i}{\partial x_j}(s,y)\Gamma^l_f(\tau,x;s,y)dyds+\\
\\
\int_0^{\tau}\rho_l\int_{{\mathbb R}^n}\int_{{\mathbb R}^n}K_{n,i}(z-y)\times\\
\\
{\Big (} \sum_{j,k=1}^n\left( f_{k,j}f_{j,k}\right) (s,y)- \sum_{j,k=1}^n\left( g_{k,j}g_{j,k}\right) (s,y){\Big)}\Gamma^l_f(\tau,x;s,z)dydzds,\\
\end{array}
\end{equation}
where we denote 
\begin{equation}\label{der}
 f_{k,j}=\frac{\partial f_k}{\partial x_j}~~\mbox{etc.}
\end{equation}
for simplicity. (Note that in (\ref{burgfg}) we may express the first spatial derivatives of the function $\mathbf{v}^g$ in terms of an integral involving the respective first spatial derivatives of the fundamental solution $\Gamma^l_g$).
In order to construct a fixed point we shall use some decay at spatial infinity. 
Indeed we need a certain decay at spatial infinity of the approximating functions $\mathbf{v}^{\rho,k,l}$ of $\mathbf{v}^{f,\rho,l}$ in order to estimate the limit of the functional series $\left( \mathbf{v}^{\rho,k,l}\right)_k$ (with respect to some appropriate norm, for example a Sobolev norm $|.|_{H^s}$ for $s\geq \frac{n}{2}+\alpha$. In order to construct an iteration scheme for the higher order corrections at each time step $l$, i.e., the functions $\delta \mathbf{v}^{r,\rho,k,l}$ for $k\geq 1$, we observe that
\begin{equation}\label{rightsiderel}
\begin{array}{ll}
\sum_{j,k=1}^n\left( f_{k,j}f_{j,k}\right) (s,y)- \sum_{j,k=1}^n\left( g_{k,j}g_{j,k}\right) (s,y)=\\
\\
\sum_{j,k=1}^n\left( f_{k,j}f_{j,k}\right) (s,y)-\sum_{j,k=1}^n\left( f_{k,j}g_{j,k}\right) (s,y)+\\
\\
\sum_{j,k=1}^n\left( g_{k,j}f_{j,k}\right) (s,y)-\sum_{j,k=1}^n\left( g_{k,j}g_{j,k}\right) (s,y)=\\
\\
\left( \sum_{j,k=1}^n\left( f_{k,j}(s,y)+g_{k,j}(s,y)\right) \right) \times\\
\\
\left( \sum_{j,k=1}^n\left( f_{j,k}(s,y)-g_{j,k}(s,y)\right) \right).
\end{array}
\end{equation}
In order to deal with the problem of constructing a fixed point in a function space of infinite domain we use the standard assumption concerning decay of regular initial data at spatial infinity, i.e. we assume that the map $x\rightarrow \mathbf{h}(x)=\left(h_1(x),\cdots ,h_n(x) \right) $ is a given function with components $h_i$ in $C^ {\infty}\left({\mathbb R}^n \right)\cap H^s$ for all $1\leq i\leq n$ and $s$ large enough. Here, we write $H^s=H^s\left({\mathbb R}^n\right)$. 
 Now three different procedures are possible in order to construct a fixed point for each time step $l$ in an appropriate function space, i.e. locally in time. The first method a) below makes use of embedding theorems for convergence of local schemes. It is remarkable that Hilbert space theory suffices in order to deal with the most interesting case of dimension $n=3$. The following method can be generalized in order to include arbitrary dimensions by considering embedding theorems for spaces of H\"{o}lder type (so called Zygmund spaces $C^s_*$). These spaces coincide with classical H\"{o}lder spaces for noninteger values $s$. We have the standard embedding theorem
 \begin{thm} For all $s\in {\mathbb R}^n$ and $p\in (1,\infty)$
  \begin{equation}
  H^{s,p}\left({\mathbb R}^n\right)\subset C^{r}_*\left({\mathbb R}^n\right)
  \end{equation}
for $r=s-\frac{n}{p}$
 \end{thm}
Here the spaces $H^{s,p}$ are defined similarly as in the case $p=2$ and coincide with the spaces $H^{k,p}$ for integers $k$ which have weak derivatives up to order $k$ in $L^p$. We do not repeat the definitions here (which can be looked up in standard textbooks) since this of marginal importance for us, i.e., in order to remark that the following method a) can be realized for arbitrary dimension $n\geq 1$. We shall see later that the integral magnitude in (\ref{vnorm}) is finite.
 \begin{itemize}
  \item[a)]  The following method can applied in the case of arbitrary dimension (cf. preceding remarks), but for simplicity we consider the case $n=3$. This is a construction in $\left[ C^{1,2}_b\left(\left[l-1,l\right] \times {\mathbb R}^n \right) \right]^n\cap \left[H^2_l\right]^n$, where
   \begin{equation}
   H^2_l:=\left\lbrace f\in C_b\left(\left[l-1,l\right] \times {\mathbb R}^n \right)|f(t,.)\in H^2\left({\mathbb R}^n\right),~\forall~t\in [l-1,l] \right\rbrace .
   \end{equation}
This is the space of vector-valued functions ${\mathbf h}=(h_1,\cdots ,h_n)^T$ with $h_i\in C^{1,2}_b$, i.e.,  $h_i$ is in $C^{1,2}$ with bounded derivatives of first order with respect to time and second order with respect to space. 
   For the series $\mathbf{v}^{\rho,k,l}$ defined in $(iii)$ above we can establish  that $\mathbf{v}^{\rho,k,l}\in \left[ C^{1,2}_b\left(\left[l-1,l\right] \times {\mathbb R}^n \right) \right]^n$ such that for some $0<c<1$ we have 
  \begin{equation}\label{vfg12}
   |\mathbf{v}^{r,\rho,k+1,l}-\mathbf{v}^{r,\rho,k,l}|^n_{1,2}\leq c|\mathbf{v}^{r,\rho,k,l}-\mathbf{v}^{r,\rho,k-1,l}|^n_{1,2},
  \end{equation}
where 
\begin{equation}
|\mathbf{f}|^n_{1,2}:=\sum_{i=1}^n\left[ |f_{i}|_0+\sum_{j=1}^n|f_{i,j}|_0+\sum_{j,k=1}^n|f_{i,j,k}|_0\right],
\end{equation}
and where $|.|_0$ denotes the supremums norm. Then we can show that
\begin{equation}
 \mathbf{f}\in \left[H^2_l\right]^n\rightarrow \mathbf{v}^f\in \left[  H^2_l\right]^n.
\end{equation}
Then our scheme leads to a series $\left( \mathbf{v}^{\rho,k,l}\right)_k\in \left[ C^{1,2}_b\left(\left[l-1,l\right] \times {\mathbb R}^n \right) \right]^n\cap \left[  H^2_l\right]^n$ with a limit
\begin{equation}
\mathbf{v}^{r,\rho,l}\in \left[  H^2_l\right]^n.
\end{equation}
Since $n=3$ the functions $v^{\rho,l}_i$ are H\"{o}lder continuous with respect to the spatial variable (uniformly in time $\tau$) and the fundamental solution $\Gamma^{r,l}_v$ of
\begin{equation}
\frac{\partial \Gamma^{r,l}_v}{\partial \tau}-\rho_l\nu\sum_{j=1}^n \frac{\partial^2 \Gamma^{r,l}_v}{\partial x_j^2} 
+\rho_l\sum_{j=1}^n v^{\rho,l}_j\frac{\partial \Gamma^{r,l}_v}{\partial x_j}-L^{\rho,l,0}_i(\mathbf{r},\mathbf{v})=0
\end{equation}
exists. Then we can use the representation
\begin{equation}\label{navvl}
\begin{array}{ll}
v^{r,\rho,l}_i(r,\tau,x)=\int_{{\mathbb R}^n}h_i(y)\Gamma^{r,l}_v(\tau,x;0,y)dy+\\
\\
\int_0^{\tau}\rho_l\int_{{\mathbb R}^n}\int_{{\mathbb R}^n}K_{n,i}(z-y)\times\\
\\
\sum_{j,k=1}^n\left( v^{r,\rho,l}_{k,j}v^{r,\rho ,l}_{j,k}\right) (s,y)\Gamma^{r,l}_v(\tau,x;s,z)dydzds\\
\\
-\int_0^{\tau}\int_{{\mathbb R}^n}r^l_{i,\tau}(s,y)\Gamma^{r,l}_v(\tau,x;s,z)dydzds,\\
\\
+\int_0^{\tau}\int_{{\mathbb R}^n} L^{\rho,l,0}_i(\mathbf{r},\mathbf{v})(s,y) \Gamma^{r,l}_v(\tau,x;s,z)dydzds,
\end{array}
\end{equation}
where the existence of the first spatial derivatives on right side of (\ref{navvl}) may be established directly. Existence of second order spatial derivatives and first order time derivatives may be derived using partial integration.
 \begin{rem} In a subsequent paper concerning Navier-Stokes equations on manifolds we shall show that
 from a closer look at (\ref{vfg12}) we may even conclude that for some $c\in (0,1)$ we have a time-local contraction with respect to H\"{o}lder norms, i.e.,
\begin{equation}\label{vfgalpha}
   |\mathbf{v}^{r,\rho,k+1,l}-\mathbf{v}^{\rho,k,l}|^n_{\alpha/2,2+\alpha}|\leq c|\mathbf{v}^{r,\rho,k,l}-\mathbf{v}^{\rho,k-1,l}|^n_{\alpha/2,2+\alpha}.
  \end{equation} 
 For compact manifolds  a system of localized equations can be derived of where we have a local contraction on a  Banach space of H\"{o}lder type, and this simplifies the analysis a bit. Note, however, that the usual the introduction of another time transformation which introduces a potential term of a specific sign is critical. A naive use of this Schauder type estimates within our scheme does not lead to a global scheme. 
 Here we use standard notation for H\"{o}lder norms, where the first subscript $\alpha/2$ refers to the modulus of H\"{o}lder continuity with respect to the time variable and the second subscript refers to the modulus of H\"{o}lder continuity of the (second derivatives) of the functions $v^{r,\rho,k,l}$ with respect to the spatial variables. More precisely, define the Euclidean distance in ${\mathbb R}^{n+1}$ between the points $y_1=(t_1,x_1), y_2=(t_2,x_2)$ by
\begin{equation}
e(z_1,z_2)=\sqrt{|t_1-t_2|}+|x_1-x_2|.
\end{equation}  
If $w$ is a function in a domain $D\subset{\mathbb R}^{n+1}$ we denote for $\alpha \in (0,1)$
\begin{equation}
[w]_{\alpha/2,\alpha, D}=\sup_{y_1\neq y_2;y_1,y_2\in D}\frac{|w(y_1)-w(y_2)|}{e^{\alpha}(y_1,y_2)}.
\end{equation}
Next define
\begin{equation}
|w|_{\alpha/2,\alpha;D}=|w|_{0 ,D}+[w]_{\alpha/2,\alpha;D}.,
\end{equation}
Similarly, we define 
\begin{equation}
[w]_{\alpha/2,1+\alpha, D}=\sum_{j=1}^n[w_{,j}]_{\alpha/2,\alpha, D},
\end{equation}
and with the notation $w_{t}:=\frac{\partial w}{\partial t}$ we have 
\begin{equation}
[w]_{1+\alpha/2,2+\alpha, D}=\sum_{j=1}^n[w_{t}]_{\alpha/2,\alpha, D}+\sum_{j,k=1}^n[w_{,j,k}]_{\alpha/2,\alpha, D}.
\end{equation}
This leads to the notation of more regular H\"{o}Der spaces
\begin{equation}
|w|_{\alpha/2,1+\alpha;D}:=|w|_{0,1;D}+[w]_{\alpha/2,1+\alpha, D}
\end{equation}
and
\begin{equation}
|w|_{1+\alpha/2,2+\alpha;D}:=|w|_{1,2;D}+[w]_{1+\alpha/2,2+\alpha, D},
\end{equation}
where we use the notation
\begin{equation}\label{classnorm0}
|w|_{0,1;D}:=|w|_{0;D}+\sum_{i=1}^n|w_{x_i}|_{0;D},~\mbox{and}
\end{equation}
\begin{equation}\label{classnorm}
|w|_{0,2;D}:=|w|_{0;D}+\sum_{i=1}^n|w_{x_i}|_{0;D}++\sum_{i,j=1}^n|w_{x_ix_j}|_{0;D},
\end{equation}
and
\begin{equation}\label{classnorm}
|w|_{1,2;D}:=|w|_{0;D}+\sum_{i=1}^n|w_{x_i}|_{0;D}+|w_t|_{0;D}+\sum_{i,j=1}^n|w_{x_ix_j}|_{0;D}.
\end{equation}
Note that the latter norms do not define Banach spaces in general (they do on compact domains). However, below we shall use the fact that a functional series which is uniformly and absolutely bounded with uniformly and absolutely bounded derivatives can be differentiated term by term. For this matter (\ref{classnorm}) is useful. If the domain $D$ is determined from the context we shall suppress it in notation, especially if $D$ is of the form $[S,T]\times {\mathbb R}^n$.
\end{rem}
 \item[b)] Another possibility (a more elementary way) is the following. We use the relation (\ref{vfg12}) and apply it iteratively starting with the function $\mathbf{v}^{\rho,l-1}$ and get a series $\left( \mathbf{v}^{r,\rho,k,l}\right)_k\in C^{1,2}_b$ with
 \begin{equation}
 \mathbf{v}^{\rho,k,l}=\mathbf{v}^{r,\rho,l-1}+\sum_{m=0}^k\delta \mathbf{v}^{r,\rho,m,l}\in \left[ C^{1,2}_b\left(\left[l-1,l\right] \times {\mathbb R}^n \right) \right]^n,
 \end{equation}
where
\begin{equation}
\delta \mathbf{v}^{r,\rho,m,l}=\mathbf{v}^{r,\rho,m,l}-\mathbf{v}^{r,\rho,m-1,l}.
\end{equation}
Next assume that for all $x\in\mathbb{R}^n$ and for multiindices $\alpha$ with $|\alpha|\leq 2$ and $k\leq 5$ we have constants $C_{\alpha k}$ such that 
\begin{equation}\label{growthh}
    \vert \partial^\alpha_x \mathbf{h}(x)\vert\le \frac{C_{\alpha k}}{(1+\vert x\vert)^k}\qquad 
\end{equation}
For each substep $k$ of each timestep $l$ we shall show that for $|\alpha|\leq 2$ the approximations $\mathbf{v}^{r,\rho,k,l}$ satisfy
\begin{equation}\label{growthuburg}
    \vert \partial^\alpha_x \mathbf{v}^{r,\rho,k,l}(t,x)\vert\le \frac{C_{\alpha k}}{(1+\vert x\vert)^5}~\mbox{if}~|\alpha|\leq 2\qquad . 
\end{equation}
Then the decay relation for the solution helps us to control convergence of our scheme by consideration of convergence of a functional series in transformed spatial coordinates. In transformed spatial coordinates we may use
\begin{prop}
Let $K$ be a compact domain and assume that a functional series $(h_k)_k:C_b(K)\rightarrow {\mathbb R}$ is given. 
Assume that for all $x\in K$ the series $\sum_{k\in{\mathbb N}} h_k(x_0)$ converges, and assume that for $1\leq j\leq n$ the series $\sum_{k\in {\mathbb N}}h_{k,j}$ converges uniformly in $K$. Then the series $\sum_{k\in {\mathbb N}}h_{k}$ converges uniformly to a function $h^*$ which is differentiable and such that for $1\leq j\leq n$ we have 
\begin{equation}
h_{,j}(x)=\sum_{k\in {\mathbb N}}h_{k,j}(x).
\end{equation}
\end{prop}
The transformation uses the polynomial decay of order $5$ described above.
Application of this lemma then leads to a classical solution $\mathbf{w}^l$ (locally in time) on a compact space $K_l$ (the image of a coordinate transformation of $D^{\tau}_l$) and then to a classical solution $v^{l}$ on $D^{\tau}_l$.
\item[c)]
The third possibility is to do the contraction estimate in a stronger H\"{o}lder norm, i.e.
\begin{equation}\label{vfgalphaclass}
   |\mathbf{v}^{r,\rho,k+1,l}-\mathbf{v}^{r,\rho,k,l}|^n_{1+\alpha/2,2+\alpha}|\leq c|\mathbf{v}^{r,\rho,k,l}-\mathbf{v}^{r,\rho,k-1,l}|^n_{1+\alpha/2,2+\alpha}.
  \end{equation}
  and then establish decay at spatial infinity as in b). Then in transformed coordinates for the equivalent series $\left( \mathbf{w}^{r,\rho,k,l}\right)
  _k$ we have $\mathbf{w}^{r,\rho,k,l}\in C_{1+\alpha/2,2+\alpha}\left(K_l\right)$, and $\mathbf{w}^{r,\rho,l}\in C_{1+\alpha/2,2+\alpha}\left(K_l\right)$ as $k\uparrow \infty$. Note that the latter space is a Banach space.
 \end{itemize}

We may summarize our construction as follows. A solution of the Navier Stokes equation system which is given by a global fixed point
\begin{equation}
\mathbf{f}\rightarrow \mathbf{v}^{f}=F(\mathbf{f}),
\end{equation}
where $\mathbf{f}=\left(f_1,\cdots ,f_n\right)^T$ is defined on $[0,\infty)\times {\mathbb R}^n$, and $\mathbf{v}^{f}$ satisfies the equation
\begin{equation}\label{Navlerayiterglob}
\left\lbrace \begin{array}{ll}
\frac{\partial v^{f}_i}{\partial t}-\nu\sum_{j=1}^n \frac{\partial^2 v^{f}_i}{\partial x_j^2} 
+\sum_{j=1}^n f_j\frac{\partial v^{f}_i}{\partial x_j}=\\
\\ \hspace{3cm}\int_{{\mathbb R}^n}\left( \frac{\partial}{\partial x_i}K_n(x-y)\right) \sum_{j,k=1}^n\left( \frac{\partial f_k}{\partial x_j}\frac{\partial f_j}{\partial x_k}\right) (\tau,y)dy,\\\\
\\
\mathbf{v}^{f}(0,.)=\mathbf{h}(0,.).
\end{array}\right.
\end{equation}
This fixed point may be denoted by
\begin{equation}
\mathbf{v}^{*}:=F(\mathbf{v}^{*})=\mathbf{v}^{\mathbf{v}^{*}}.
\end{equation}
It seems difficult to construct this fixed point directly. Therefore, we construct
a fixed point of an equivalent problem which has the solution $\mathbf{v}^r=\mathbf{v}+\mathbf{r}$ in successive time steps $l$ on the domains 
$[l-1,l]\times {\mathbb R}^n$, where $\mathbf{r}=(r_1,\cdots ,r_n)$ is a bounded function.
We solve for $\mathbf{v}^r$ in time steps with functions $\mathbf{v}^{r,\rho,l}=\left(v^{r,\rho,l}_1,\cdots , v^{r,\rho,l}_n\right)^T$, 
via a functional series
\begin{equation}
 \mathbf{v}^{r,\rho,l}=\mathbf{v}^{r,\rho,0,l}+\sum_{k\geq 1}\delta \mathbf{v}^{r,\rho,k,l}.
\end{equation}
We  establish a time-local contraction property for the specific functional series $\mathbf{v}^{r,\rho,k,l}$ and prove convergence to $\mathbf{v}^{r,\rho,l}$ for each time step $l>0$. The function $\mathbf{r}$ equals $\mathbf{r}^l$ on each domain $(l-1,l]\times {\mathbb R}^n$, and is defined recursively in such a way that the solution $\mathbf{v}^r$ and the control function $\mathbf{r}$ and theire first spatial derivaitves are bounded  globally in time. Here, $\mathbf{r}^l$ solves at each time step a linearized equation of Navier-Stokes type with 'consumption' source term.  Note that this is crucial for global convergence. The local contraction property for the higher order corrections follow from a priori estimates of fundamental solutions of approximative subproblems. We shall use the Levy expansion of scalar linear fundamental solutions. 
\begin{rem}
Note that the application of standard Schauder estimates (as will be used in the second part of this article) requires a certain sign of the coefficient of the potential term. This sign is not given for the successive problems for $\mathbf{v}^{r,\rho,l}$, $l\geq 1$. However, since the functions $r^l_i$ are bounded,  we can always consider equivalent problems related to
\begin{equation}\label{pot}
u^{r,\rho,l,\lambda}_i(\tau,x)=e^{-\lambda \tau}v_i^{r,\rho,l}(\tau,x)
\end{equation}
with $\lambda>0$ large enough have coefficients for potential terms with the same sign as the time derivative. Once we have established global H\"{o}lder-estimates we can use the Schauder estimates for the transformed functions in order to improve the regularity estimates locally using our local scheme of a a first linear parabolic approximation and iterative scalar linear parabolic corrections. 
\end{rem}
 Note the difference to our construction in \cite{K} for the multidimensional Burgers equation. If the transformation (\ref{pot}) and standard Schauder estimates are used in order to establish a local contraction property of our scheme, then it is not sufficient to prove linear growth of the solution globally. However, since we prove he boundedness of the solution we can even use standard Schauder estimates at each time step in order to prove local convergence. We emphasize that the global H\"{o}lder continuity of the $\mathbf{v}^{r,\rho}$ and $\mathbf{r}$ are essential. Let us point out this from the point of view of Schauder type estmates. This view is not equally constructive as the point of view of the Levy expansion. Nevertheless, it may be useful. For the sake of completeness let us refer to this standard results which may be found in \cite{Kr}.

 Let $L$ be a scalar parabolic partial differential operator defined by
\begin{equation}
\begin{array}{ll}
Lf(t,x)\equiv \frac{\partial}{\partial t}f(t,x)- \sum_{i,j=1}^n a_{ij}(t,x)\frac{\partial^2}{\partial x_i\partial x_j}f(t,x)\\
\\
+\sum_{i=1}^n b_i(t,x)\frac{\partial }{\partial x_i}f(t,x)+c(t,x)f(t,x)
\end{array}
\end{equation}
for all $f\in C^{1,2}\left(Q_T\right)$ for arbitrary $T\in (0,\infty)$.
\begin{rem}\label{rempot}
 Note that any potential term $cf$ has the same (positive ) sign as the time derivative.
\end{rem}

Let us assume that the spatial part of the operator $L$ is uniformly elliptic with ellipticity constant $\lambda\in (0,\infty)$ (ellipticity from below) and $\Lambda\in (\lambda ,\infty)$ (ellipticity from above), i.e.
for all 
\begin{equation}\label{ell}
\forall~\xi \in {\mathbb R}^n\setminus\left\lbrace 0\right\rbrace,~~\xi \in {\mathbb R}^n:~\lambda |\xi|^2\leq \sum_{ij}a_{ij}(t,x)\xi_i\xi_j\leq \Lambda |\xi|^2
\end{equation}
Furthermore we assume that the coefficients have bounded H\"{o}lder norms, i.e.
\begin{equation}\label{coeff}
|a_{ij}|_{\delta/2,\delta,Q_T}\leq K,~|b_{i}|_{\delta/2,\delta,Q_T} \leq K,~|c|_{\delta/2,\delta,Q_T} \leq K
\end{equation}
for some constant $K>0$.
\begin{thm}\label{cauchylin}
Assume that (\ref{ell}) and (\ref{coeff}) hold and that $c\leq r$ for some constant $r>0$ (note that minus sign in front of the second order coefficients in the definition of the parabolic operator $L$, cf. remark (\ref{rempot})). Let $g\in C^{2+\delta}\left({\mathbb R}^n\right) $ and $f\in C^{\delta/2,\delta}(Q_T)$. Then there exists a unique function $u\in C^{1+\delta/2,2+\delta}\left(Q_T\right)$ which solves the Cauchy problem
\begin{equation}
\left\lbrace \begin{array}{ll}
Lu=f~\mbox{on}~Q_T,\\
\\
u(0,.)=g(.)~\mbox{on}~{\mathbb R}^n. 
\end{array}\right.
\end{equation}
Moreover the estimate
\begin{equation}\label{schauderapriori}
|u|_{1+\delta/2,2+\delta,Q_T}\leq C\left(|g|_{2+\delta,{\mathbb R}^n}+|f|_{\delta/2,\delta,Q_T} \right) 
\end{equation}
holds for some constant $C>0$ depending only on $n,\lambda,\Lambda, \delta ,K$.
\end{thm}
Note that in the estimate (\ref{schauderapriori}) the norm $|g|_{2+\delta,{\mathbb R}^n}$ is with respecto the intial data $g$ which are zero for the linear problems $\delta {\bf v}^{r,\rho,l,k}$. This means that the Schauder estimates can be applied for a local fixed point iteration, i.e. a contraction
\begin{equation}
|\delta v^{r,\rho,l,k+1}_i|_{1+\delta/2,2+\delta,Q_T}\leq c |\delta v_i^{r,\rho,l,k}|_{1+\delta/2,2+\delta,Q_T}
\end{equation}
can be derived for some $c\in (0,1)$ if $\rho_l$ is small enough. However it is more difficult in this case to combine the local Schauder estmates to a global estimate since the Schauder estimate requires a negative coefficient of the potential term. Standard transformations would lead to a time step size which is exponentially decreasing and this is not sufficient in order to ensure that the scheme is global.    
 Therefore, in this paper we shall show constructively that we have a classical bounded solution with an integral magnitude (\ref{vnorm}) which has linear growth with respect to time. This is achieved by a dynamical choice of the functions $r^l_i$, i.e., a recursive construction with respect to the time step number $l$, and where we use classical expansions of fundamental solutions.
In case of the Navier-Stokes equation on compact manifolds (the flat $n$-torus is a simple example of a flat one) we do not have to deal with the integral norm seperately. However, some different difficulties arise if the coefficient of second order depend on the spatial variables. We shall treat this case in a subsequent paper.  

In order to prove the existence of a bounded solution $\mathbf{v}=\left(v_1,\cdots ,v_n \right)$ of the incompressible Navier Stokes equation in its Leray projection form we construct a bounded function $\mathbf{r}=\left(r_1,\cdots ,r_n \right)$ which is smooth with respect to the spatial variables and differentiable with respect to time except at a discrete set of points. Moreover the functions $\tau\rightarrow r_i(\tau,x)$ are H\"{o}lder continuous across the time step points $l>0$, where classical time derivatives exist elsewhere. The values where the time derivatives of $\tau\rightarrow r_i(\tau,x)$ do not exist in a classical sense coincide with the time discretizaion of our scheme, i.e. with the values of natural numbers ($\tau\in {\mathbb N}$). The bounded function $\mathbf{r}$ is constructed recursively at each time step $l$. Hence $\mathbf{r}$ is determined by a sequence of functions $\mathbf{r}^l=(r^l_1,\cdots ,r^l_n)$, where each $\mathbf{r}^l$ is defined on the domain $[l-1,l]\times {\mathbb R}^n$. Instead of solving the incompressible Navier-Stokes equation we solve a corresponding equation for the function
\begin{equation}
(\tau,x)\rightarrow \mathbf{v}^{r,\rho}:=\mathbf{v}^{\rho}+\mathbf{r},
\end{equation}
time step by time step, i.e., in the form $\mathbf{v}^{r,\rho,l}=\mathbf{v}^{\rho,l}+\mathbf{r}^l$ for $l\geq 1$ where each $\mathbf{v}^{r,\rho,l}$
is defined on the domain $[l-1,l]\times {\mathbb R}^n$. We first determine $\mathbf{r}^l$ from data of the previous time step, i.e., from  $\mathbf{v}^{r,\rho,l-1}$ (or from the initial data $\mathbf{h}$ if $l=1$. This is done in such a way that the locally constructed solution for $\mathbf{v}^{r,\rho,l}$ stays bounded independently of the time step number $l$. 

The outline of this paper is as follows. In the next Section 2 we prove the main theorem, i.e., the existence of a bounded regular function $\mathbf{r}$  such that $\mathbf{v}^r:=\mathbf{v}+\mathbf{r}$ satisfies an equation equivalent to the incompressible Navier-Stokes equation in an (almost) classical sense. This leads to a classical solution $\mathbf{v}$ for the Navier-Stokes equation itself.   Then in Section 3  we draw conclusions concerning regularity, uniqueness, and extensions to equations with external forces. Finally, in Section 4 we sketch an algorithm for the solution of the Navier-Stokes equation by iterative use of local analytic expansions of scalar parabolic equations. This algorithm simplifies the construction of the global solution via the function $\mathbf{v}^r$ since we can use the information that the solution is bounded. Nevertheless the control function $\mathbf{r}$ may be useful also in numerical practice, since it can make the computation of the solution scheme more stable.

\section{Main theorem}
We shall assume that the initial data function
 $x\rightarrow \mathbf{h}(x)=\left(h_1(x),\cdots ,h_n(x) \right)^T$ with components $h_i$ in $C^ {\infty}\left({\mathbb R}^n \right)$ for all $1\leq i\leq n$ satisfies
\begin{equation}\label{growthh}
    \vert \partial^\alpha_x \mathbf{h}(x)\vert\le \frac{C_{\alpha k}}{(1+\vert x\vert)^k}\qquad 
\end{equation} 
for all $x\in\mathbb{R}^n$, and for all multiindices $\alpha$ and and integers $k$ with some constants $C_{\alpha k}$. 
\begin{rem}
If there is an external force term  $\mathbf{f}_{ex}$ included on the right then we may assume that
\begin{equation}\label{growthf}
    \vert \partial^m_t\partial^\alpha_x \mathbf{f}_{ex}(t, x)\vert\le \frac{C_{\alpha m k}}{(1+\vert x\vert + t)^k}\qquad ~~\mbox{for all}~~ \qquad (t,x)\in [0,\infty)\times \mathbb{R}^n,
\end{equation}
for any multiindices $\alpha$, integers $k$, and nonnegative integers $m$, and with some constants $C_{\alpha m k}$. 
\end{rem}
Essentially, neglecting regularity requirements with respect to time for the external force function $\mathbf{f}_{ex}$ for a moment, the assumptions on the external forces ${\mathbf f}_{ex}$ and the initial data ${\mathbf h}$ mean that these functions are located in Sobolev spaces $H^s$ of arbitrary order $s\in {\mathbb R}$, i.e., for all $s\in {\mathbb R}$ we have
\begin{equation}\label{hHs}
{\mathbf h}\in \left[ H^s\left( {\mathbf R}^n\right)\right]^n,
\end{equation}
and for all $t\in \left[0,\infty \right)$ 
\begin{equation}\label{fHs}
{\mathbf f}_{ex}(t,.)\in 
\left[ H^s\left({\mathbf R}^n\right)\right]^n.
\end{equation}
Here, for $s\in {\mathbb R}$ $H^s$ is an ordinary Hilbert space with
\begin{equation}
H^s\equiv H^s\left({\mathbb R}^n\right)=\left\lbrace f\in {\cal S}'\left({\mathbb R}^n\right)|\int\hat{f}(\xi)^2\left( 1+|\xi|^2 \right)^s d\xi <\infty\right\rbrace,
\end{equation}
where $\hat{f}$ denotes the Fourier transform and ${\cal S}'\left({\mathbb R}^n\right)$ denotes the space of tempered distributions on ${\mathbb R}^n$. The time dependence in the estimate (\ref{growthf}) implies a similar Sobolev regularity with respect to time of course. In order to keep notation simple we consider the case $f_i=0$ in this proof of a classical global solution (however, this is not essential and it will be clear that our construction can be generalized to the case involving source terms satisfying (\ref{fHs}) or (\ref{growthf})).
As indicated in the introduction our proof of a bounded regular solution of the incompressible Navier-Stokes equation consists of three main ideas: a) we introduce a time discretization and a series of linear time transformations $t=\rho_l \tau$ such that time step size $1$ in $\tau$-coordinates is related to a small time step size in original coordinates and small coefficients of spatial derivatives in transformed time coordinates such that an iteration procedure leads to a local solution in time. Furthermore , the time step size has to be large enough such that the scheme becomes global. This means that we should have $\sum_{l=1}^N\rho_l\uparrow \infty$ as $N\uparrow \infty$. We mentioned already that we can even construct a lower bound, i.e., $\rho_l\geq \rho$ for some $\rho>0$ independent of the time step number $l$. This is a difference to \cite{KB2} where we had $\rho_l\sim\frac{1}{l}$ only; 
 b) in order to control the growth of the solution we introduce a locally regular function $\mathbf{r}$ and solve an equivalent problem for the function $\mathbf{v}^r:=\mathbf{v}+\mathbf{r}$ at each time step where $\mathbf{r}$ is itself a bounded function with bounded spatial derivatives and which is continuous across the time points $l>0$. The function $\mathbf{r}$ solves a linearized  Navier-Stokes type equation, but with a 'consumption' source term $\mathbf{\phi}^l$ which has been explained in the introduction to some extent.  
This source term consists of $n$ components $\mathbf{\phi}^l=\left(\phi^l_1,\cdots ,\phi^l_n\right)$ 
such that the functions $\mathbf{v}^{r,\rho,l}$ and $\mathbf{r}^l$ are bounded independent of  the time step number $l$. Here, the 'consumption' function $\mathbf{\phi}^l$ and the function $\mathbf{r}^l$ are defined recursively on the domains $D^{\tau}_l$ of each time step number $l$ and depend on the solution $\mathbf{v}^{r,l-1}$ and $\mathbf{r}^{l-1}$ defined at the previous time step or dependent on the data $\mathbf{h}$ if $l=1$; c) At each time step we ensure that the solution has some decay at spatial infinity which is necessary in order to prove the existence of a time local fixed point. In the case $n=3$ it is sufficient to have $v^{r,\rho,l}_i(\tau,.)\in H^2$ for all $l\geq 1$ and $\tau\in [l-1,l]$.

\begin{rem}
In the introduction of this paper we mentioned that global regular solutions of the multivariate Burgers equation can be obtained by a priori estimates of the form (\ref{aprioriest}). If we consider the method (i) described above and start at each time step with the solution of the associated Burgers equation with a certain source term then we need to prove that solutions of the associated Burgers equation satisfy a certain decay at spatial infinity (in case the initial condition at each time step satisfies a certain decay condition at spatial infinity). This is true in any case a), b), and c) of local fixed point construction above. For example, if we consider the method c) of the fixed point construction, then a decay of polynomial order $5$ at spatial infinity is sufficient in order to prove the global convergence of our scheme to a classical solution. Hence, in order to apply method i) for the local iteration part it is sufficient to prove the following:
\begin{prop}
The Cauchy problem for the multivariate Burgers equation with initial data $\mathbf{h}\in H^s$ 
for $s\in {\mathbb R}$ has a classical solution $\mathbf{u}\in \left[ C_{1+\alpha/2,2+\alpha}\left( \left[0,\infty \right)\times {\mathbb R}^n\right)\right) ^n$
  such that $\mathbf{u}(t,.)$ satisfies (\ref{growthh}) for $k\leq 5$, i.e. we have for all $t\in [0,\infty)$
 \begin{equation}\label{growthu}
    \vert \partial^\alpha_x \mathbf{u}(t,.)\vert\le \frac{C_{\alpha k}}{(1+\vert x\vert)^k}~\mbox{for}~k\leq 5~\mbox{and}~|\alpha|\leq 2.\qquad 
\end{equation}  
More over the solution is bounded by the initial data, i.e., (\ref{apriorimax}) holds.   
\end{prop}
The proof of this proposition essentially follows from our considerations in \cite{K} and from estimates of type (\ref{aprioriest}).
The result in Proposition 1 can be sharpened in order to construct more regular solutions. In this Section we deal with the construction of classical solutions (which is essential). However, since we use the alternative (iii) in this paper we do not use this proposition in the present proof.
\end{rem}

Next let us consider the construction of the function $\mathbf{r}$ more closely. Assume that we have solved for $\mathbf{v}^{r,\rho,m}:=\mathbf{v}^{\rho,m}+\mathbf{r}^m$ for $1\leq m\leq l-1$ such that the solution for $\mathbf{v}^{r,\rho}$ is constructed on the domain $[0,l-1]\times{\mathbb R}^n$. Especially the functions $\mathbf{r}^m$ for $1\leq m\leq l-1$ have been constructed. Consider the restriction $\mathbf{v}^{r,\rho,l-1}$ of $\mathbf{v}^{r,\rho}$ to the domain $[l-1,l]\times {\mathbb R}^n$. Equation (\ref{Navlerayrrhol}) may be written explicitly in the form
\begin{equation}\label{Navlerayrrhol*++}
\left\lbrace \begin{array}{ll}
\frac{\partial v^{r,\rho,l}_i}{\partial \tau}-\rho_l\nu\sum_{j=1}^n \frac{\partial^2 v^{r,\rho,l}_i}{\partial x_j^2} 
+\rho_l\sum_{j=1}^n v^{r,\rho,l}_j\frac{\partial v^{r,\rho,l}_i}{\partial x_j}=\psi^l_i,\\
\\
\mathbf{v}^{r,\rho,l}(l-1,.)=\mathbf{v}^{r,\rho,l-1}(l-1,.),
\end{array}\right.
\end{equation}
where
\begin{equation}\label{Navlerayrrhol*+++}
\begin{array}{ll}
\psi^l_i=L^{\rho,l}_i\left(\mathbf{r}^l,\mathbf{v}^{r,\rho,l}\right)\\
\\
+\rho_l\int_{{\mathbb R}^n}\left( \frac{\partial}{\partial x_i}K_n(x-y)\right) \sum_{j,k=1}^n\left( \frac{\partial v^{r,\rho,l}_k}{\partial x_j}\frac{\partial v^{r,\rho,l}_j}{\partial x_k}\right) (\tau,y)dy+r^{l}_{i,\tau}
\end{array}
\end{equation}
\begin{equation}
\begin{array}{ll}
=r^l_{i,\tau}-\rho_l\nu \Delta r^l_i+\rho_l\sum_{j=1}^n r^{l}_j\frac{\partial r^{l}_i}{\partial x_j}\\
\\
-\rho_l\int_{{\mathbb R}^n}\left( \frac{\partial}{\partial x_i}K_n(x-y)\right) \sum_{j,k=1}^n\left( \frac{\partial r^{l}_k}{\partial x_j}\frac{\partial r^{l}_j}{\partial x_k}\right) (\tau,y)dy\\
\\
+\rho_l\sum_{j=1}^n r^{\rho,l-1}_j\frac{\partial v^{r,\rho,l-1}_i}{\partial x_j}+\rho_l\sum_{j=1}^n v^{r,\rho,l-1}_j\frac{\partial r^{l-1}_i}{\partial x_j}\\
\\ -2\rho_l\int_{{\mathbb R}^n}\left( \frac{\partial}{\partial x_i}K_n(x-y)\right) \sum_{j,k=1}^n\left( \frac{\partial r^{l-1}_k}{\partial x_j}\frac{\partial v^{r,\rho,l-1}_j}{\partial x_k}\right) (\tau,y)dy\\
\\
 +\rho_l\int_{{\mathbb R}^n}\left( \frac{\partial}{\partial x_i}K_n(x-y)\right) \sum_{j,k=1}^n\left( \frac{\partial v^{r,\rho,l}_k}{\partial x_j}\frac{\partial v^{r,\rho,l}_j}{\partial x_k}\right) (\tau,y)dy.
\end{array}
\end{equation}
In order to control the growth of the functions $\mathbf{v}^{r,\rho,l}$ at time step $l$ we control 1) the growth of $\mathbf{v}^{r,\rho,0,l}$ which satisfies a linearized equation, and 2) ensure that the correction $\mathbf{v}^{r\rho,l}-\mathbf{v}^{r,\rho,0,l}=\sum_{k=1}^{\infty}\delta \mathbf{v}^{r,\rho,k,l}$ is small enough by choosing $\rho_l$ appropriately. Let us look at the linear approximation $\mathbf{v}^{r,\rho,0,l}$ of the function $\mathbf{v}^{r,\rho,l}$ first. The equation for $\mathbf{v}^{r,\rho,0,l}$ is
\begin{equation}\label{Navlerayrrhol*++0-}
\left\lbrace \begin{array}{ll}
\frac{\partial v^{r,\rho,0,l}_i}{\partial \tau}-\rho_l\nu\sum_{j=1}^n \frac{\partial^2 v^{r,\rho,0,l}_i}{\partial x_j^2} 
+\rho_l\sum_{j=1}^n v^{r,\rho,l-1}_j(l-1,.)\frac{\partial v^{r,\rho,0,l}_i}{\partial x_j}=\\
\\
L^{\rho,l}_i\left(\mathbf{r}^l,\mathbf{v}^{r,\rho,l-1}(l-1,.)\right)+r^l_{i,\tau}\\
\\
+\rho_l\int_{{\mathbb R}^n}\left( \frac{\partial}{\partial x_i}K_n(x-y)\right) \sum_{j,k=1}^n\left( \frac{\partial v^{r,\rho,l-1}_k}{\partial x_j}\frac{\partial v^{r,\rho,l-1}_j}{\partial x_k}\right) (l-1,y)dy,\\
\\
\mathbf{v}^{r,\rho,0,l}(l-1,.)=\mathbf{v}^{r,\rho,l-1}(l-1,.).
\end{array}\right.
\end{equation}
Note that all functions with index $l-1$ are defined on the domain $[l-2,l-1]\times {\mathbb R}^n$ of the previous time step. Therefore these functions are alway evaluated at time $l-1$ if they occur in the $l$th time step.
Considering the linearized equation for $\mathbf{v}^{r,\rho,0,l}$ has the advantage that we may represent the solution of (\ref{Navlerayrrhol*++0-}) in terms of the fundamental solution of the equation
\begin{equation}\label{fundvr0}
\frac{\partial v^{r,\rho,0,l}_i}{\partial \tau}-\rho_l\nu\sum_{j=1}^n \frac{\partial^2 v^{r,\rho,0,l}_i}{\partial x_j^2} 
+\rho_l\sum_{j=1}^n v^{r,\rho,l-1}_j(l-1,.)\frac{\partial v^{r,\rho,0,l}_i}{\partial x_j}=0.
\end{equation}
This representation involves a spatial integral with the initial data and a source term integral with respect to space and time involving the right side of the equation (\ref{Navlerayrrhol*++0-}). The classical representation involves an integral with initial data which is solution of an equation where we can apply the maximum principle. The second term can be controlled by choice of $\mathbf{r}^l$ as indicated in the introduction. Next determining $\mathbf{r}^l$ in a straightforward way such that the right side of (\ref{Navlerayrrhol*++0-}) becomes a function $\psi^{l,0}_i$ approximating the function $\psi^l_i$ (\ref{Navlerayrrhol*++}) would lead us to a Navier-Stokes type equation for $\mathbf{r}^l$. The equation is
\begin{equation}\label{Navlerayrrhol*++00}
\left\lbrace \begin{array}{ll}
r^l_{i,\tau}-\rho_l\nu \Delta r^l_i+\rho_l\sum_{j=1}^n r^{l}_j\frac{\partial r^{l}_i}{\partial x_j}\\
\\
-\rho_l\int_{{\mathbb R}^n}\left( \frac{\partial}{\partial x_i}K_n(x-y)\right) \sum_{j,k=1}^n\left( \frac{\partial r^{l}_k}{\partial x_j}\frac{\partial r^{l}_j}{\partial x_k}\right) (\tau,y)dy\\
\\
+L^{\rho,l,0}_i\left(\mathbf{r}^l,\mathbf{v}^{r,\rho,l-1}\right)\\
\\
+\rho_l\int_{{\mathbb R}^n}\left( \frac{\partial}{\partial x_i}K_n(x-y)\right) \sum_{j,k=1}^n\left( \frac{\partial v^{r,\rho,l-1}_k}{\partial x_j}\frac{\partial v^{r,\rho,l-1}_j}{\partial x_k}\right) (l-1,y)dy\\
\\
=\psi^{l,0}_i\\
\\
\mathbf{r}^{l}(l-1,.)=\mathbf{r}^{l-1}(l-1,.).
\end{array}\right.
\end{equation}
However, solving for $\mathbf{r}^l$ in this form would imply that we have transferred the original difficulties to control the growth of the solution to $\mathbf{r}^l$. Instead, we   
consider a different right side $\phi^l_i$ which approximates $\psi^{l,0}_i$ for $\rho_l$ small. Here, 'approximation' is in the sense that $\psi^{l,0}_i=\phi^l_i+\left(\psi^{l,0}_i-\phi^l_i \right) $, and the difference becomes small if $\rho_l$ becomes small (cf. proof of main theorem below). Furthermore $\psi^{l,0}_i$ approximates  $\psi^l_i$. Here, we let  $r^{l}_i$ solve a linearization of (\ref{Navlerayrrhol*++}), i.e. the function $\mathbf{r}^l$ solves
\begin{equation}\label{Navlerayrrhol+++++}
\left\lbrace \begin{array}{ll}
r^l_{i,\tau}-\rho_l\nu \Delta r^l_i+\rho_l\sum_{j=1}^n r^{l-1}_j(l-1,.)\frac{\partial r^{l}_i}{\partial x_j}=\\
\\
-\rho_l\int_{{\mathbb R}^n}\left( \frac{\partial}{\partial x_i}K_n(x-y)\right) \sum_{j,k=1}^n\left( \frac{\partial r^{l-1}_k}{\partial x_j}\frac{\partial r^{l-1}_j}{\partial x_k}\right) (l-1,y)dy\\
\\
+L^{\rho ,l,0}_i(\mathbf{r}^{l-1}(l-1,.);\mathbf{v}^{r,\rho,l-1}(l-1,.))\\
\\-\rho_l\int_{{\mathbb R}^n}\left( \frac{\partial}{\partial x_i}K_n(x-y)\right) \sum_{j,m=1}^n\left( \frac{\partial v^{r,\rho,l-1}_m}{\partial x_j}\frac{\partial v^{r,\rho,l-1}_j}{\partial x_m}\right) (l-1,y)dy+\phi^l_i,\\
\\
\mathbf{r}^{l}(l-1,.)=\mathbf{r}^{l-1}(l-1,.).
\end{array}\right.
\end{equation}
Note that it is the same $\mathbf{r}^l$ which solves this linearized equation which is the equation with a different source term $\phi^l_i$. Hence the difference of $\psi^{l,0}_i$ and $\phi^l_i$ may be obtained by subtracting equation (\ref{Navlerayrrhol+++++}) from equation (\ref{Navlerayrrhol*++00}). We get
\begin{equation}\label{Navlerayrrhol*++000}
\begin{array}{ll}
\rho_l\sum_{j=1}^n (r^{l}_j-r^{l-1}_j)\frac{\partial r^{l}_i}{\partial x_j}\\
\\
-\rho_l\int_{{\mathbb R}^n}\left( \frac{\partial}{\partial x_i}K_n(x-y)\right) \sum_{j,k=1}^n\left( \frac{\partial r^{l}_k}{\partial x_j}\frac{\partial r^{l}_j}{\partial x_k}\right) (\tau,y)dy\\
\\
+\rho_l\int_{{\mathbb R}^n}\left( \frac{\partial}{\partial x_i}K_n(x-y)\right) \sum_{j,k=1}^n\left( \frac{\partial r^{l-1}_k}{\partial x_j}\frac{\partial r^{l-1}_j}{\partial x_k}\right) (l-1,y)dy\\
\\
+L^{\rho,l,0}_i\left(\mathbf{r}^l-\mathbf{r}^{l-1},\mathbf{v}^{r,\rho,l-1}\right)\\
\\
=\psi^{l,0}_i-\phi^l_i.
\end{array}
\end{equation}
We shall indeed see that we can make the right side of (\ref{Navlerayrrhol*++000}) small independently of $l$ with the right choice of $\rho_l$ and $\mathbf{r}^l$. How small do we want to get it? Well, we want to have $\phi^l_i$ close to $\psi^{l,0}_i$ and $\psi^{l,0}_i$ close to $\psi^{l}_i$ where we want to choose $\phi^l_i$ such that it controls both, $\mathbf{v}^{r,\rho,l}$ and $\mathbf{r}^l$. This is done as follows: first we choose $\phi^l_i$ based on the dynamic information of $\mathbf{v}^{r,\rho,l-1}$ and $\mathbf{r}^{l-1}$. We do this as indicated in the introduction, and as we explain in a more elaborative way now. Further details are  provided below in the proof of the main theorem.
 For each $1\leq i\leq n$ and at each time step $l$ the function $\phi^l_i$ is constructed as a sum $\phi^l_i=\phi^{v,l}_i+\phi^{r,l}_i$ as described in the introduction.

Having constructed $\phi^{l}_i=\phi^{v,l}_i+\phi^{r,l}_i$ for all $1\leq i\leq n$  we solve
the linear equation (\ref{Navlerayrrhol+++++}) for $\mathbf{r}^l$. Note hat all terms on the right side of (\ref{Navlerayrrhol+++++}) have the factor $\rho_l$ except for $\phi^l_i$. Choosing $\rho_l$ small enough we can ensure that $\phi^l_i$ dominates the source terms in the critical regions of arguments where the moduli of the function $r^{l-1}_i$ exceed a certain level. Moreover, the control function is designed in such a way that the spatial first order derivatives are also controlled. Indeed the functions $\mathbf{r}^l$ are classical solutions of linear parabolic problems. Then we plug in this function $\mathbf{r}^l$ into the right side of (\ref{Navlerayrrhol*++0-}), and using (\ref{Navlerayrrhol*++000}) we get
\begin{equation}\label{Navlerayrrhol*++0}
\left\lbrace \begin{array}{ll}
\frac{\partial v^{r,\rho,0,l}_i}{\partial \tau}-\rho_l\nu\sum_{j=1}^n \frac{\partial^2 v^{r,\rho,0,l}_i}{\partial x_j^2} 
+\rho_l\sum_{j=1}^n v^{r,\rho,l-1}_j(l-1,.)\frac{\partial v^{r,\rho,0,l}_i}{\partial x_j}=\\
\\
\psi^{l,0}_i=\phi^l_i+\left( \psi^{l,0}_i-\phi^l_i\right),\\
\\
v^{r,\rho,0,l}_i(l-1,.)=v^{r,\rho,l-1}_i(l-1,.).
\end{array}\right.
\end{equation}
Then we may represent the solution of (\ref{Navlerayrrhol*++0}) in terms of the fundamental solution $\Gamma^l$ of the equation (\ref{fundvr0}) and get
\begin{equation}
\begin{array}{ll}
 v^{r,\rho,0,l}_i(\tau,x)=\int_{{\mathbb R}^n}v^{r,\rho,l-1}_i(l-1,y)\Gamma^l(\tau,x;0,y)dy\\
 \\
 +\int_{l-1}^{\tau}\int_{{\mathbb R}^n}\left(\phi^l_i+\left( \psi^{l,0}_i-\phi^l_i\right) \right)(s,y) \Gamma^l(\tau,x;s,y)dsdy.
 \end{array}
\end{equation}
The first term of this representation can be estimated using the maximum principle for the corresponding equation without source term. The second term can be estimated (and becomes small) if $\rho_lC_r$ becomes small (independently of the number $l$). Finally we have to ensure that the 'series of higher order corrections' of the functions $ v^{r,\rho,0,l}_i,~1\leq i\leq n$, i.e., the functional series $\sum_{k=1}^{\infty}\delta\mathbf{v}^{r,\rho,k,l}$ which we have to add in order to compute the function $\mathbf{v}^{r,\rho,l}$ from its linear approximation $\mathbf{v}^{r,\rho,0,l}$, becomes small if $\rho_l$ is small enough.

We shall prove
\begin{thm}\label{mainthm}
Given any dimension $n$ and a viscosity constant $\nu >0$ let $\mathbf{h}$ satisfy (\ref{hHs}) for any $s\in {\mathbb R}$. Then there is a global classical solution
\begin{equation}
\mathbf{v}\in \left[ C^{1,2}\left( \left[0,\infty \right)\times {\mathbb R}^n\right)\right] ^n
\end{equation}
to the Navier-Stokes equation system (\ref{nav}), (\ref{navdiv}), (\ref{navinit}) which satisfies
\begin{equation}
v_i(t,.)\in H^2.
\end{equation}
\end{thm}
\begin{rem}
The indicated more elementary methods of the proof below also show that
for all $t\geq 0$ and all $1\leq i\leq n$.
\begin{equation}\label{growthv}
    \vert \partial^\alpha_x \mathbf{v}(t,.)\vert\le \frac{C_{\alpha k}}{(1+\vert x\vert)^k}~\mbox{for}~k\leq 5~\mbox{ and }~|\alpha|\leq 2\qquad 
\end{equation}
for some constants $C_{\alpha k}>0$ which depend on dimension $n$ the viscosity constant $\nu$ and the initial data $\mathbf{h}$.
Moreover, below in Section 4 we shall use the proof of theorem  \ref{mainthm} in order to improve the scheme above for numerical purpose and in order to adapt it to initial-boundary value problems. Especially, the function $\mathbf{r}$ is not needed for the algorithm since the boundedness of the solution $\mathbf{v}$ is shown. The main purpose of the function $\mathbf{r}$ is to show boundedness of the solution $\mathbf{v}$. Once this is done it is not needed in order to define an algorithm. However, it may be useful in order to stabilize computations, especially in case of initial-boundary-value problems which occur in practice. 
\end{rem}

\begin{proof}
We do the proof in four steps.
\begin{itemize}
 \item[1)] In a first step we prove existence of a bounded local classical solution $\mathbf{v}^{\rho,l}$ assuming existence of a regular function $\mathbf{v}^{\rho,l-1}(l-1,.)$, i.e., we prove existence of the local solution $\mathbf{v}^{\rho,l}$ on $[l-1,l]\times{\mathbb R}^n$ for $l-1\geq 0$, where bounded data $v^{\rho,l-1}_i(l-1,.)\in C^2_b\cap H^2$ are given. Here $C^2_b\equiv C^2_b({\mathbb R}^n)$ denotes the space of  functions which have bounded derivatives up to second order. In this first step we may set $\mathbf{r}=0$ in order to simplify the local part of the argument. Recall that $\mathbf{v}^{\rho,l}=\mathbf{v}^{r,\rho,l}-\mathbf{r}$, and that $\mathbf{v}^{\rho,l}=\mathbf{v}^{0,\rho,l}=\mathbf{v}^{r,\rho,l}$ in case $\mathbf{r}=0$. Moreover, we show that the local solution has certain decay at spatial infinity (as indicated in the statement of theorem \ref{mainthm} above) if the data have this property. Especially, we show that for each component function $v^{r,\rho,l}_i$ of the function $\mathbf{v}^{\rho,l}$ we have
 \begin{equation}\label{step02}
 v^{r,\rho,l}_i(\tau,.)\in H^2
 \end{equation}
for all $\tau\in [l-1,l]$. For higher dimension $n>3$ the estimate
 \begin{equation}
 v^{r,\rho,l}_i(\tau,.)\in H^{s,p}
 \end{equation}
along with $\alpha =s-\frac{n}{p}$ for some $\alpha \in (0,1)$ also leads to the conclusion that the velocity functions $v^{r,\rho,l}_i(\tau,.)$ are H\"{o}lder continuous.
 \item[2)] In a second step we show that the local iteration of the first step of this proof can be extended in order to show  existence of a local solution $\mathbf{v}^{r,\rho,l}$  on $[l-1,l]\times{\mathbb R}^n$ for a certain class of functions $\mathbf{r}^l$ assuming existence of a regular function $\mathbf{v}^{r,\rho,l-1}(l-1,.)$ and a bounded regular function $\mathbf{r}^l$ with a linearly bounded integral magnitude for $r^l_i$ (as in (\ref{vnorm})), i.e., we show existence of a bounded  local solution $\mathbf{v}^{r,\rho,l}$ on $[l-1,l]\times{\mathbb R}^n$ for $l-1\geq 0$ where bounded data $v^{r,\rho,l-1}_i(l-1,.)\in C^2_b$ and $\mathbf{r}^{l-1}(l-1,.)\in C^2_b$ are given, and where
 \begin{equation}\label{rmagnitude}
 \int_{{\mathbb R}^n}\sum_{j,k=1}^n{\Big |}\left( \frac{\partial r^l_k}{\partial x_j}\frac{\partial r^l_j}{\partial x_k}\right) (t,y){\Big |}dy<\infty.
 \end{equation}
 Moreover, we show that the local solution has polynomial decay stated in theorem \ref{mainthm} if the initial data of the respective time step have this property.
 \item[3)] In a third step we show that there exists a recursively defined bounded function ${\mathbf r}$ and a global iterative scheme which solves for $\mathbf{v}^{r,\rho}=\mathbf{v}+\mathbf{r}$. Especially, for each time step number $l$ the bounded function $\mathbf{r}^l$ is defined via a linearized Navier-Stokes-type equation which has certain source term $\phi^l=(\phi^l_1,\cdots ,\phi^l_n)$ (among other source terms). These source terms $\phi^l_i$ are defined in terms of the functions $\mathbf{v}^{r,\rho,l-1}(l-1,.)$ and $\mathbf{r}^{l-1}(l-1,.)$ which  are  known from the previous time step. They are defined in order to control the growth of the function $\mathbf{v}^{r,\rho}$ and of the control function $\mathbf{r}^l$ itself. Here we use the fact that we have some freedom in order to define the functions $r^l_i$. Indeed these functions are defined in terms of the 'consumption terms' $\phi^l$. We shall conclude that the functions $\mathbf{v}^{r,\rho,l}$ and $\mathbf{r}$ are globally bounded, i.e., there exists a constant $C>0$ independent of $l$ such that for all $1\leq i\leq n$ and $l\geq 1$ we have
\begin{equation}
|v^{r,\rho,l}_i(\tau,.)|_ {H^2}\leq C,
\end{equation}
and uniformly with respect to $\tau$.
Moreover at each time step $l\geq 1$ we shall ensure that
\begin{equation}
|r^l_i(\tau,.)|_{H^2}\leq C,~|v^{\rho,l}_i(\tau,.)|_{H^2}=|v^{r,\rho,l}_i(\tau,.)-r^l_i(\tau,.)|_{H^2}\leq C.
\end{equation}
for some generic constant $C$. For $n=3$ this implies that we have a global bound in terms of a H\"{o}lder norm. Since we show that 
\begin{equation}
|v^{r,\rho,l}_i(\tau,.)|_ {C^1}\leq C,
\end{equation}
and uniformly with respect to $\tau$, and
\begin{equation}
|r^l_i(\tau,.)|_{C^1}\leq C,~|v^{\rho,l}_i(\tau,.)|_{C^1}=|v^{r,\rho,l}_i(\tau,.)-r^l_i(\tau,.)|_{C^1}\leq C.
\end{equation}
we have global boundedness to H\"{o}lder norms also in higher dimensions $n$, i.e., there exists a constant $C>0$ independent of $l$ and theres exists a H\"{o}lder constant $\alpha$ such that for all $1\leq i\leq n$ and $l\geq 1$ we have
\begin{equation}
|v^{r,\rho,l}_i|_{\alpha}\leq C,
\end{equation}
and
\begin{equation}
|r^l_i|_{\alpha}\leq C,~|v^{\rho,l}_i|_{\alpha}=|v^{r,\rho,l}_i-r^l_i|_{\alpha}\leq C,
\end{equation}
where $C>0$ is generic.
The construction of the function ${\mathbf r}$ which equals $\mathbf{r}^l$ at each time step $l$ is such that a) step 2 above can be applied with $\rho_l$ and ${\mathbf r}^l$ as they are specified in this step, and b) such that the functions ${\mathbf r}^l$ and their spatial derivatives up to second order on $(l-1,l]\times {\mathbb R}^n$ are bounded and have a finite integral magnitude (\ref{rmagnitude}). In \cite{BK2} the integral magnitude increases linearly in time. Hence the choice $\rho_l\sim \frac{1}{l}$. Applying an appropriate control function we can use a uniform step size $\rho_l\geq \rho$ for some $\rho>0$ independent of the tie step number $l$. Issues of computational efficiency make increasing $\rho_l$ with time step number $l\geq 1$ a matter of interest. For this reason we keep the index $l$ for the time stepsize $\rho_l$.
\item[4)] In a fourth step we show the existence of a globally bounded classical solution of the Navier Stokes equation, i.e., we show that $\mathbf{v}^{\rho}=\mathbf{v}^{r,\rho}+\mathbf{r}\in \left[ C^{1,2}\left([0,\infty)\times {\mathbb R} \right)\right]^n $, and that for all $l\geq 1$ we have $\mathbf{v}^{\rho,l}=\mathbf{v}^{r,\rho,l}+\mathbf{r}^l\in \left[ C^{1,2}_b\left([l-1,l]\times {\mathbb R} \right)\right]^n $ accordingly. This is not an immediate consequence of step 3) because for all $x\in {\mathbb R}^n$ the function $\tau \rightarrow \mathbf{r}(\tau,x)$ is only weakly differentiable at the integer values $l\in {\mathbb N}$ in general. Since the sequence $(\rho_l)$ is chosen such that $\sum_{l=1}^{N}\rho_l\uparrow \infty$ as $N\uparrow \infty$ we conclude that the global solution $\mathbf{v}$ in original time coordinates exists.
\end{itemize}

\subsection{step 1: proof of local existence of solutions at each time step} 

 The proof of step i) is in two substeps. First we prove a contraction property on the domain $[l-1,l]\times {\mathbf R}^n$ for the series $(v^{\rho,k,l}_i)_k,~1\leq i\leq n$ with
 \begin{equation}\label{vrhokl*}
 v^{\rho,k,l}_i=v^{\rho,0,l}_i+\sum_{m=1}^k\delta v^{\rho,m,l}_i,
 \end{equation}
 for a fixed time step $l$ where we assume that $\mathbf{v}^{\rho,l-1}(l-1,.)\in C^{2}_b\cap H^2$ at the $l$th time step. In a first substep we construct an iteration in $C^{1,2}_b\left([l-1,l]\times {\mathbb R}^n\right)$ which is not an iteration in a Banach space. However, it is useful to have classical representations of approximating solutions in order to do some estimates. The limit function is in $H^2$ (resp. $H^{2,p}$ for higher dimensions $n$) and therefore H\"{o}lder continuous. Therefore, in a second substep of this first step we show that we have a certain decay at spatial infinity  such that we can use one of the method a) (in case $n=3$ with Hilbert space theory, i.e., showing convergence of the series $v^{r,\rho,k,l}_i$ in the Hilbert space $H^2$, or showing convergence of the series $v^{r,\rho,k,l}_i$ in with respect to $H^{s,p}$ Banach spaces  in the case of general dimension $n$ ) or the even more elementary method  b) (where $n$ is an arbitrary dimension) in order to prove the existence of a local fixed point. The method c) of the introduction is considered in a subsequent paper- it becomes important if one considers the Navier-Stokes equation on manifolds.
  \begin{rem}
 Note that the argument of this first step is local. We establish some property for the function $\mathbf{v}^{\rho,l}$ assuming that some property holds for the function $\mathbf{v}^{\rho,l-1}(l-1,.)$ of the previous time step which serves as initial data for the $l$th time step, i.e., $\mathbf{v}^{\rho,l}(l-1,.)=\mathbf{v}^{\rho,l-1}(l-1,.)$. This means that the argument here is about local existence at an arbitrary time step $l$. The global existence (induction over $l$) is then considered with the introduction of the function ${\mathbf r}$ in the third step of this proof.
 \end{rem} 
For $l\geq 1$ assume that $\mathbf{v}^{\rho,l-1}(l-1,.)$ has been computed, i.e., the initial data of equation (\ref{Navlerayl}) have been computed, and assume that for all $1\leq i\leq n$
\begin{equation}
 |v^{\rho,l-1}_i|_{0}\leq C^{l-1}_{0},~|v^{\rho,l-1}_i|_{0,1}\leq C^{l-1}_{1}|,~|v^{\rho,l-1}_i|_{0,2}\leq C^{l-1}_{0,2},~|v^{\rho,l-1}_i|_{1,2}\leq C^{l-1}_{1,2}
\end{equation}
 for some constant $C^{l-1}_{0},~C^{l-1}_{1},~C^{l-1}_{0,2},~C^{l-1}_{1,2}$ which are given from the previous time step $l-1$. In order to get estimates on Hilbert spaces $H^{2}$  we use an inductive assumption for the $H^{2}$-norm, i.e. we assume that
 \begin{equation}
\sum_{i=1}^n\sum_{|\alpha|\leq 2} \int_{{\mathbb R}^n}|v^{\rho,l-1}_{i,\alpha}(l-1,x)|^2dx\leq C^{l-1}_{1,2}
 \end{equation}
 for some generic $C^{l-1}_{1,2}$
The reason for the difference to \cite{KB2} -where we have a bound which is linear in time- will become apparent in step 3 of this proof.
The superscript $l-1$ of the constants above indicates that we do local estimates with respect to time here, i.e. the constants may depend on the time step number $l$. Recall that at this stage of the argument we have not introduced any control function $r^l$ yet.

 Furthermore, in order to apply the alternative method b) described in the introduction we assume that  initial data function of the previous time step
 $x\rightarrow \mathbf{v}^{\rho,l-1}(l-1,x)=\left(v^{\rho,l-1}_1(l-1,x),\cdots ,v^{\rho,l-1}_n(l-1,x) \right)^T$ with components $v^{\rho,l-1}_i$ which  for all $1\leq i\leq n$ satisfy
\begin{equation}\label{growthh}
    \vert \partial^\alpha_x \mathbf{v}^{\rho,l-1}(l-1,x)\vert\le \frac{C_{\alpha k}}{(1+\vert x\vert)^k}\qquad 
\end{equation} 
for all $x\in\mathbb{R}^n$, and for all multiindices $|\alpha|\leq 2$ and and integers $1\leq k\leq 5$ with some constants $C_{\alpha k}$.

  In case $l=1$ (the first time step) we have $\mathbf{v}^{\rho,l-1}(l-1,.)=\mathbf{h}$ and the constants $C^{0}_{0},~C^{0}_{1}, C^{0}_{0,2}$ are given in terms of upper bounds of the components of the functions $h_i$ and derivatives of the functions $h_i$. The series (\ref{vrhokl*}) can be generated by an iterative application of the map
\begin{equation}\label{contract}
 F_l:\mathbf{f}\rightarrow \mathbf{v}^{f,\rho,l},
\end{equation}
starting with the value of $F_l$ applied to the initial data of step $l$, i.e., starting with the function $\mathbf{v}^{\rho,0,l}=F_l\left( \mathbf{v}^{\rho,l-1}(l-1,.)\right)$. Here, the function $\mathbf{v}^{\rho,l-1}(l-1,.)$  represents the final data of the previous step or the data $\mathbf{h}$ in case $l=1$. The function $\mathbf{v}^{f,\rho,l}$ is determined by the solution of the equation
\begin{equation}\label{Navlerayiterlfff}
\left\lbrace \begin{array}{ll}
\frac{\partial v^{f,\rho,l}_i}{\partial \tau}-\rho_l\nu\sum_{j=1}^n \frac{\partial^2 v^{f,\rho,l}_i}{\partial x_j^2} 
+\rho_l\sum_{j=1}^n f_j\frac{\partial v^{f,\rho,l}_i}{\partial x_j}= \\
\\ \hspace{1cm}\rho_l\int_{{\mathbb R}^n}\left( \frac{\partial}{\partial x_i}K_n(x-y)\right) \sum_{j,k=1}^n\left( \frac{\partial f_k}{\partial x_j}\frac{\partial f_j}{\partial x_k}\right) (\tau,y)dy,\\\\
\\
\mathbf{v}^{f,\rho,l}(l-1,.)=\mathbf{v}^{\rho,l-1}(l-1,.).
\end{array}\right.
\end{equation}
We want to show that the map $F_l$ is a contraction on this series with respect to some appropriate norm, i.e., we want to show that
\begin{equation}\label{contractspec}
|\delta \mathbf{v}^{\rho,k,l}|\leq \frac{1}{4}|\delta \mathbf{v}^{\rho,k-1,l}|
\end{equation}
for some suitable norms $|.|$ and $k\geq 1$, and where we start the series with $
 \mathbf{v}^{\rho,0,l}=F_l\left(\mathbf{v}^{\rho,l-1}\right)$. The norms are the $|.|_{1,2}$ norm in order to have classical representations of approximating solutions in terms of certain fundamental solutions, and the $|.|_{H^2}$- and $|.|_{H^{2,p}}$-norms in order finish the estimate for the limit according to method a). Furthermore, we shall show that we have some decay at infinity such that the alternative method $b)$ can be applied. This implies that we have an upper bound of the integral magnitude for $\mathbf{v}^{\rho,l}$, i.e.,
 we have a  upper bound of the magnitude
 \begin{equation}\label{vrhonorm}
 \int_{{\mathbb R}^n}\sum_{j,k=1}^n{\Big |}\left( \frac{\partial v^{\rho,l}_k}{\partial x_j}\frac{\partial v^{\rho,l}_j}{\partial x_k}\right){\Big |} (t,y)dy.
 \end{equation}
Since products of functions may be pointwise estimated by half the sum of the squares of the factor functions, the upper bound of (\ref{vrhonorm}) follows from upper bound of the $H^2$-norm of the solution function. Here we note that a $H^1$-norm estimate seems sufficient (in case $n=3$), but we shall need first spatial dervatives of the integral term for our estimates. Note that the term (\ref{vrhonorm}) is related to the integral term of the Leray projection form of the Navier-Stokes equation. Since first order derivatives of (\ref{vrhonorm}) can be estimated quite naturally in terms of $H^2$ norms, we use such an  estimate for our solution scheme of the Navier-Stokes equation. Interestingly, in case of $n=3$ this is also sufficient in order to have an embedding in H\"{o}lder continuous spaces.    
Recall that in case $l=1$ we have $\mathbf{v}^{\rho,l-1}=\mathbf{h}$.
 In order to prove a contraction property of type (\ref{contractspec}) we consider the map $F_l(\mathbf{f})=\mathbf{v}^{f,\rho,l}$ for general $\mathbf{f}\in \left[ C^{1,2}_b\right]^n$ first, 
where $\mathbf{v}^{f,\rho,l}$ satisfies the equation (\ref{Navlerayiterlfff}). We assume that the data $v^{\rho,l-1}_i(l-1,.)\in C^{1,2}_b$ are bounded (this is certainly true in the case $l=1$, where $\mathbf{v}^{\rho,l-1}_i(l-1,.)=\mathbf{h}$). Note again that in this first step we let $\mathbf{r}=0$. We consider the methods a) and b) of the introduction and consider first the normed space $C^{1,2}_b\left(D^{\tau}_l\right)$ along with the norm $|.|_{1,2}$ of the introduction. 
 Note that this normed space, i.e., $C^{1,2}_b$ equipped with the norm $|.|_{1,2}$ is not a Banach space (especially, because the domain is not compact), but we can proceed according to the methods a) or b) of the introduction. We shall determine $\mathbf{v}^{\rho,l}=\left(v^{\rho,l}_1,\cdots ,v^{\rho,l}_n\right)^T$ as a limit of a functional series with members in $C^{1,2}_b$, and with an additional decay at infinity, i.e., for all $1\leq i\leq n$ we construct $v^{\rho,l}_i$ as a limit of the approximations $v^{\rho,k,l}_i$
\begin{equation}
C^{1,2}_b\left(D^{\tau}_l\right)\ni v^{\rho,k,l}_i=v^{\rho,0,l}_i+\sum_{m=1}^{k} \delta v^{\rho,m,l}_i,
\end{equation}
where we prove in a second substep that all functions $v^{\rho,k,l}_i$ have a certain decay at infinity.
Here, the function 
$\mathbf{v}^{\rho,0,l}=\left(v^{\rho,0,l}_1,\cdots ,v^{\rho,0,l}_n\right)^T$ solves
\begin{equation}\label{Navlerayl0intr0}
\left\lbrace \begin{array}{ll}
\frac{\partial v^{\rho,0,l}_i}{\partial \tau}-\rho_l\nu\sum_{j=1}^n \frac{\partial^2 v^{\rho,0,l}_i}{\partial x_j^2} 
+\rho_l\sum_{j=1}^n v^{\rho,l-1}_j\frac{\partial v^{\rho,0,l}_i}{\partial x_j}\\
\\
=\rho_l\int_{{\mathbb R}^n}\left( \frac{\partial}{\partial x_i}K_n(x-y)\right) \sum_{j,k=1}^n\left( \frac{\partial v^{\rho,l-1}_k}{\partial x_j}\frac{\partial v^{\rho,l-1}_j}{\partial x_k}\right) (\tau,y)dy,\\
\\
\mathbf{v}^{\rho,0,l}(l-1,.)=\mathbf{v}^{\rho,l-1}(l-1,.).
\end{array}\right.
\end{equation}
Furthermore, for $k\geq 1$ we consider the difference functions $\delta \mathbf{v}^{\rho,k,l}=\mathbf{v}^{\rho,k,l}-\mathbf{v}^{\rho,k-1,l}$   where for all $k$ we have $\mathbf{v}^{\rho,k,l}=\mathbf{v}^{r,\rho,k,l}|_{\mathbf{r}=0}$. Note that we have $\delta \mathbf{v}^{\rho,k,l}$ recursively defined to be solutions of the equations
\begin{equation}\label{Navlerayl0intdelta}
\left\lbrace \begin{array}{ll}
\frac{\partial \delta v^{\rho,k,l}_i}{\partial \tau}-\rho_l\nu\sum_{j=1}^n \frac{\partial^2 \delta v^{\rho,k,l}_i}{\partial x_j^2} 
+\rho_l\sum_{j=1}^n v^{\rho,k-1,l}_j\frac{\partial \delta v^{\rho,k,l}_i}{\partial x_j}=\\
\\
-\rho_l\sum_{j=1}^n \delta v^{\rho,k-1,l}_j\frac{\partial  v^{\rho,k,l}_i}{\partial x_j}\\
\\ + \rho_l\int_{{\mathbb R}^n}\left( \frac{\partial}{\partial x_i}K_n(x-y)\right) \sum_{j,k=1}^n\left( \frac{\partial v^{\rho,k-1,l}_k}{\partial x_j}\frac{\partial v^{\rho,k-1,l}_j}{\partial x_k}\right) (\tau,y)dy\\
\\
- \rho_l\int_{{\mathbb R}^n}\left( \frac{\partial}{\partial x_i}K_n(x-y)\right) \sum_{j,k=1}^n\left( \frac{\partial v^{\rho,k-2,l}_k}{\partial x_j}\frac{\partial v^{\rho,k-2,l}_j}{\partial x_k}\right) (\tau,y)dy,\\
\\
\delta\mathbf{v}^{\rho,k,l}(l-1,.)=0.
\end{array}\right.
\end{equation}
Here, for $k=1$ we denote $v^{\rho,k-2,l}_i=v^{\rho,-1,l}_i:=v^{\rho,l-1}_i$.
Solving these equations recursively leads us successively to the functions 
\begin{equation}\label{seriesk}
\mathbf{v}^{\rho,k,l}=\mathbf{v}^{\rho,0,l}+\sum_{m=1}^k\mathbf{v}^{\rho,m,l}
\end{equation}
 which approximate the local solution $\mathbf{v}^{\rho,l}$. In order to apply the methods a) and b) of the introduction we show first that for $\rho_l$ small enough we have
 \begin{equation}\label{deltacontract}
|\delta \mathbf{v}^{\rho,k,l}|_{1,2}\leq \frac{1}{4}|\delta \mathbf{v}^{\rho,k-1,l}|_{1,2}.
\end{equation}
Then in order to apply method a) of the introduction  we show that for all $\tau\in [l-1,l]$ we have
\begin{equation}\label{estH2*}
|\delta \mathbf{v}^{\rho,k,l}(\tau,.)|_{H^2}\leq \frac{1}{4}|\delta \mathbf{v}^{\rho,k-1,l}(\tau,.)|_{H^2}.
\end{equation} 
uniformly in $\tau$. We shall also consider how polynomial decay of some degree at infinity is preserved by the local scheme if the initial data from the previous time step have this property. 
In this work on the Navier-Stokes equation system we use methods a) and  b) of the introduction in order to show the existence of time-local solutions for general dimension by elementary means. Recall that in case of method b) it is shown that the members of the functional series $(\mathbf{v}^{\rho,k,l})_k$ satisfy  a specific polynomial decay at infinity such that a spatial transformation leads us to a functional series in a Banach space, and this method applies for any dimension $n$. Hence we have two different proofs in the case $n=3$ corresponding to the methods a) and b) of the introduction. 
 \begin{rem}
 Locally a sharper estimate  with respect to
the H\"{o}lder space $C_{1+\alpha/2,2+\alpha}\left(D^{\tau}_l\right)$ is possible. However, it is easier to control the growth with respect to the $|.|_{1,2}$ norm. This control of the growth is essential for the global existence proof. Classical Schauder estimates introduce a potential term with a certain sign, and this would lead to a exponentially decreasing series of time step size numbers $\rho_l$ if applied naively. In this case the sum of the time step size numbers $\rho_l$ is finite and the scheme is not global in time. Hence, method c) of the introduction may be applied locally, but a globally some additional work is needed.
\end{rem}
Note that in case $n=3$ it is essential to show in the second substep that $v^{\rho,k,l}_i(\tau,.)\in H^2\left( {\mathbb R}^n\right) $ for all $1\leq i\leq n$, because the limit $v^{\rho,l}_i(\tau,.)\in H^2\left({\mathbb R}^n\right)$ for all $1\leq i\leq n$ (uniformly in $\tau$) implies that 
\begin{equation}
v^{\rho,l}(\tau,.)_i\in C^{\alpha}\subset H^2\left( {\mathbb R}^3\right)
\end{equation}
for some $\alpha>0$ uniformly in $\tau$ for all $1\leq i\leq n$. 
\begin{rem} In order to extend the proof to general dimension $n$ we may use an estimate based on Zygmund spaces.
For all $s\in {\mathbb R}^n$ and $p\in (1,\infty)$
  \begin{equation}
  H^{s,p}\left({\mathbb R}^n\right)\subset C^{r}_*\left({\mathbb R}^n\right)
  \end{equation}
for $r=s-\frac{n}{p}$. Hilbert space estimates for $n>3$ are harder to obtain. The relation $v^{\rho,k,l}_i(\tau,.)\in H^s\left( {\mathbb R}^n\right)$ for some $s\geq m+\frac{n}{2}+\alpha$ and some $m\geq 0$, and some $\alpha \in (0,1)$ for all $1\leq i\leq n$ 
 implies
 \begin{equation}
v^{\rho,l}_i(\tau,.)\in C^{\alpha}\subset H^s\left( {\mathbb R}^n\right).
\end{equation}
\end{rem}
Concerning method b) and c) of the introduction we note that the spatial transformation on a finite domain does not preserve the uniform ellipticity of the operator (represented by the positive constant $\nu$ in the original equation). In this respect the method a) is a little bit more elementary than methods b) and c) of the introduction, because we do not need a spatial coordinate transformation in this case.
In order to check the contraction property of the functional series $\left( \mathbf{v}^{\rho,k,l}\right)_{k\geq 0}$ we  check the contraction property of the functional (\ref{contract}) on an appropriate functional subset of the function space $\left[ C^{1,2}_b\left({\mathbb R}^n\right)\right]^n$.
First let us stick with the original spatial coordinates and consider the equation for $v^{f,\rho,l}_i$ more closely. Using relation (\ref{rightsiderel}) from equation (\ref{burgfg}) we get 

\begin{equation}\label{burgfg*}
\begin{array}{ll}
v^{f,\rho,l}_i(\tau,x)-v^{g,\rho,l}_i(\tau,x)=\\
\\
=-\rho_l\int_{(l-1)}^{\tau}\int_{{\mathbb R}^n}\sum_{j=1}^n\left( f_j-g_j\right)(s,y)\frac{\partial v^{g,\rho,l}_i}{\partial x_j}(s,y)\Gamma^l_f(\tau,x;s,y)dyds+\\
\\
\int_{l-1}^{\tau}\rho_l\int_{{\mathbb R}^n}\int_{{\mathbb R}^n}K_{n,i}(z-y)\times\\
\\
{\Big (} \left( \sum_{j,k=1}^n\left( f_{k,j}+g_{k,j}\right)(s,y) \right) \times\\
\\
\left( \left( f_{j,k}-g_{j,k}\right)(s,y) \right) {\Big)}\Gamma^l_f(\tau,x;s,z)dydzds,\\
\end{array}
\end{equation}
where  $\Gamma^l_f$ is the fundamental solution of the equation
\begin{equation}
\frac{\partial \Gamma^l_f}{\partial \tau}-\rho_l\nu \Delta \Gamma^l_f+\rho_l\sum_{j=1}^n f_j\frac{\partial \Gamma^l_f}{\partial x_j}=0
\end{equation}
on the domain $D^{\tau}_l$.

Here, and in the following we write $\left( f_j-g_j\right)(s,y):=f_j(s,y)-g_j(s,y)$, $\left( f_{j,k}-g_{j,k}\right)(s,y)= f_{j,k}(s,y)-g_{j,k}(s,y)$ etc. for simplicity of notation.
We show that $\rho_l$ can be chosen such that the iterations
\begin{equation}
F_l^m(\mathbf{v}^{\rho,l-1})\in S_0,
\end{equation}
where $S_0$ is a subset of $C^{1,2}_b$, i.e.,
\begin{equation}\label{S0}
\mathbf{f},\mathbf{g}\in \left\lbrace \mathbf{f}||f_i|_{1,2}\leq 2C^{l-1}_{1,2}\& |f_i(\tau,.)|_{H^{2}}\leq 2 C^{l-1}_{1,2}~\mbox{unif.}~\forall \tau\in [l-1,l]\right\rbrace =:S_{0}.
\end{equation}
Recall that we assumed that
\begin{equation}
|v^{\rho,l-1}_i(l-1,.)|_{1,2}\leq C^{l-1}_{1,2} \mbox{ and }|v^{\rho,l-1}_i(l-1,.)|_{H^2}\leq C^{l-1}_{1,2}
\end{equation}
The choice of the space $S_0$ is related to the fact that we use classical representations of approximating solutions which may be estimated in terms of Gaussian estimates. These Gaussian estimates are estimates of convolutions which are estimated by  applying a generalized Young inequality (cf. below).
\begin{rem}
You may ask why we consider the contraction in the space $C^{1,2}_b$ with norm $|.|_{1,2}$ at all, since this is not a Banach space. The reason is that classical representations are very useful when we establish a local scheme. This will become apparent in step 3 of this proof.
\end{rem}

The factor $2$ in the definition of $S_0$ is due to the fact that the first order coefficients vary with the the local iteration number $k$ (see below). The requirement that $h_i\in H^{2}$ in (\ref{S0}) is due to the fact that we need to establish a contraction with respect to the $H^2$ norm in order to apply method a) of the introduction. 
In order to get a uniform bound of this first order coefficients which are independent of the iteration number $k$ we establish an estimate of the form 
\begin{equation}\label{burgfgestclass}
\begin{array}{ll}
\sum_{i=1}^n|v^{f,\rho,l}_i-v^{g,\rho,l}_i|_{1,2}\leq 
\frac{1}{4}
 |\mathbf{f}-\mathbf{g}|^n_{1,2}.
\end{array}
\end{equation}
Recall from above that we equip the vector valued functions of the classical function space $C^{1,2}_b$ with the norm
\begin{equation}
|\mathbf{f}|^{n}_{1,2}=\sum_{i=1}^n|f_i|_{1,2}
\end{equation}
(we may also consider the maximum etc. instead of the sum, of course, but such considerations do matter only if we consider implementations of associated algorithms; the reference to the domain $D^{\tau}_l$ is not denoted if it is clear from the context for sake of simplicity of notation).

\begin{rem} 
Note that the estimate (\ref{burgfgestclass}) is one step of our construction of a solution in order to apply the methods of item a) and b) of the introduction.
Locally in time there is a refined estimate for H\"{o}lder spaces, i.e., we can get estimates of the form 
\begin{equation}\label{burgfgestalpha}
\begin{array}{ll}
\sum_{i=1}^n|v^{f,\rho,l}_i-v^{g,\rho,l}_i|_{1+\alpha/2,2+\alpha}\leq 
c
 |\mathbf{f}-\mathbf{g}|^n_{1+\alpha/2,2+\alpha}.
\end{array}
\end{equation}
Here, the 
H\"{o}lder norm of the vector-valued function $\mathbf{f}$ is
\begin{equation}
|\mathbf{f}|^{n}_{1+\alpha/2,2+\alpha}=\sum_{i=1}^n|f_i|_{1+\alpha/2,2+\alpha}.
\end{equation}
Then the method c) of the introduction can be used in order to construct a solution at each time step $l$. The application of the methods a) and b) is described below.
\end{rem}
In order to prove the contraction property (\ref{burgfgestclass}) it is essential to estimate the second order derivatives of the difference $v^{f,\rho,l}_{i}-v^{g,\rho,l}_{i}$, i.e., it is essential to estimate 
\begin{equation}
 v^{f,\rho,l}_{i,k,m}-v^{g,\rho,l}_{i,k,m}=\frac{\partial^2}{\partial x_k\partial x_m}\left( v^{f,\rho,l}_{i}-v^{g,\rho,l}_{i}\right). 
\end{equation}
Note that we consider the case $r^l_i=0$ for all $1\leq i\leq n$ in this first step of the proof.

We consider the two summands on the right side of (\ref{burgfg*}) separately. Accordingly, let us define
\begin{equation}\label{firstsum}
\begin{array}{ll}
v^{f,\rho,l,1}_i(\tau,x)-v^{g,\rho,l,1}_i(\tau,x):=\\
\\
=-\rho_l\int_{(l-1)}^{\tau}\int_{{\mathbb R}^n}\sum_{j=1}^n\left( f_j-g_j\right)(s,y)\frac{\partial v^{g,\rho,l}_i}{\partial x_j}(s,y)\Gamma^l_f(\tau,x;s,y)dyds,
\end{array}
\end{equation}
and
\begin{equation}\label{secondsum}
\begin{array}{ll}
v^{f,\rho,l,2}_i(\tau,x)-v^{g,\rho,l,2}_i(\tau,x):=\\
\\
=\int_{(l-1)}^{\tau}\rho_l\int_{{\mathbb R}^n}\int_{{\mathbb R}^n}K_{n,i}(z-y)\times\\
\\
{\Big (} \left( \sum_{j,k=1}^n\left( f_{k,j}(s,y)+g_{k,j}(s,y)\right) \right) \times\\
\\
\left( \sum_{j,k=1}^n\left( f_{j,k}(s,y)-g_{j,k}(s,y)\right) \right) {\Big)}\Gamma^l_f(\tau,x;s,z)dydzds,
\end{array}
\end{equation}

Since $\mathbf{f},\mathbf{g}\in \left[ C^{1,2}_b\left(D^{\tau}_l\right) \right]^n$ we know that the classical fundamental solution $\Gamma^l_f$ exists in its classical Levy expansion form. Since the initial data are 
$C^{1,2}_b$ we know that $v^g_i$ is in $C^{1,2}_b\left(D^{\tau}_l\right)$ for all $1\leq i\leq n$.
 
 Especially, $(\tau,x)\rightarrow \sum_{j=1}^n\left( f_j-g_j\right)(\tau,x) v^{g,\rho,l}_{i,j}(\tau,x)$ is H\"{o}lder continuous in the spatial variables uniformly in $\tau \in [l-1,l]$. 
 
 Hence, classical Levy expansion theory of fundamental solutions (cf. \cite{F2}) shows that the function $v^{f,\rho,l,1}_i(\tau,x)-v^{g,\rho,l,1}_i(\tau,x)$ has spatial derivatives up to second order and that for spatial derivative of second order of the first summand we have the representation
\begin{equation}\label{v1}
\begin{array}{ll}
v^{f,\rho,l,1}_{i,j,k}(\tau,x)-v^{g,\rho,l,1}_{i,j,k}(\tau,x):=\\
\\
-\rho_l\int_{(l-1)}^{\tau}\int_{{\mathbb R}^n}\sum_{j=1}^n\left( f_j-g_j\right)(s,y)\frac{\partial v^{g,\rho,l}_i}{\partial x_j}(s,y)\Gamma^l_{f,j,k}(\tau,x;s,z)dydzds.
\end{array}
\end{equation}
We shall see below how to rewrite this with only one spatial derivative of the adjoint of the fundamental solution $\Gamma^l_f$. 
Next concerning the function  $v^{f,\rho,l,2}_i-v^{g,\rho,l,2}_i$ for fixed time parameter $s$ the integral function part
\begin{equation}\label{poiss1}
\begin{array}{ll}
x\rightarrow \int_{{\mathbb R}^n}K_{n,i}(x-y)\times\\
\\
{\Big (} \left( \sum_{j,k=1}^n\left( f_{k,j}(s,y)+g_{k,j}(s,y)\right) \right) \times\\
\\
\left( \sum_{j,k=1}^n\left( f_{j,k}(s,y)-g_{j,k}(s,y)\right) \right) {\Big)}dy 
\end{array}
\end{equation}
looks formally like the solution of a Poisson equation
\begin{equation}\label{poiss2}
-\Delta p=\sum_{j,k=1}^n\left( f_{k,j}f_{j,k} - g_{k,j}g_{j,k}\right).
\end{equation}
It is indeed a solution if the right side is in $L^1$, if pointwise products of functions of type $x\rightarrow f_k(s,x)$ and $x\rightarrow g_k(s,x)$ are in $H^{1,1}$. Again, since
\begin{equation}
  |f_{k,j}f_{j,k}|(\tau,x)\leq  \frac{1}{2}\left( |f_{k,j}|^2(\tau,x)+|f_{j,k}|^2(\tau,x)\right) 
\end{equation}
it is sufficient that we have $f_k(\tau,.),g_k(\tau,.)\in H^2$ (even $f_k(\tau,.),g_k(\tau,.)\in H^1$ is sufficient at this stage).
We shall see that the functions which are produced by iterations of the functional $F_l$ starting with $v^{\rho,l}_i$ satisfy this requirement. Therefore in the following we may assume that this requirement is satisfied such that the function in (\ref{poiss1}) exists. We shall consider this in more detail later in this first step of the proof, and see that it follows from our polynomial decay condition on the initial data $v^{\rho,l-1}_i$ and the properties of the functional $F_l$. The following fact is rather classical.

\begin{lem}
Let $n\geq 3$, and let $f\in L^1\left({\mathbb R}^n\right)$. Let $K$ be the fundamental solution for the Laplacian. Then 
\begin{equation}
(t,x)\rightarrow u(t,x):=\int_{{\mathbb R}^n}f(y)K(x-y)dy
\end{equation}
is a well-defined locally integrable function which solves the Poisson equation on ${\mathbb R}^n$ with data $f$, i.e.,
\begin{equation}
\Delta u=f.
\end{equation}
\end{lem}
\begin{rem}
If $n=2$ then we need the stronger assumption $\int f(z)\log |z|dz<\infty$ and the theorem is still true.
\end{rem}
Hence, if $f_j,g_j\in S_0$, then the function (\ref{poiss1}) is indeed the solution of (\ref{poiss2}). Concerning the regularity of the solution $u$ of the Poisson equation we have another classical result due to H\"{o}lder himself.
\begin{lem}
Let $n,k$ be some positive integers. Let $\Omega\subset {\mathbb R}^n$ be an open set, and let $f\in C^{k+\alpha}\left(\Omega\right)$. If $u$ is a distributive solution of
\begin{equation}
\Delta u=f,
\end{equation}
then $u\in C^{2+\alpha+k}\left(\Omega\right)$.
\end{lem}
If $f_j,g_k\in C^{1,2}_b$ for $1\leq j,k\leq n$, then the first order derivatives $f_{j,l}, g_{k,m}$  are H\"{o}lder continuous. This implies that the function of equation (\ref{poiss1}) is in $C^{1+\alpha}$ (note the derivative of the kernel $K$). 
Hence, the integral  
\begin{equation}\label{v2*}
\begin{array}{ll}
v^{f,\rho,l,2}_{i,j,k}(\tau,x)-v^{g,\rho,l,2}_{i,j,k}(\tau,x):=\\
\\
=\int_{(l-1)}^{\tau}\rho_l\int_{{\mathbb R}^n}\int_{{\mathbb R}^n}K_{n,i}(z-y)\times\\
\\
{\Big (} \left( \sum_{j,k=1}^n\left( f_{k,j}(s,y)+g_{k,j}(s,y)\right) \right) \times\\
\\
\left( \sum_{j,k=1}^n\left( f_{j,k}(s,y)-g_{j,k}(s,y)\right) \right) {\Big)}\Gamma^l_{f,j,k}(\tau,x;s,z)dydzds
\end{array}
\end{equation}
exists and represents the second order derivatives of the second term (\ref{secondsum}).
The equations (\ref{v1}) and (\ref{v2*}) contain second order derivatives of the fundamental solution $\Gamma^l_{f}$. Note that the function 
\begin{equation}\label{source}
\begin{array}{ll}
z\rightarrow \int_{{\mathbb R}^n}K_{n,i}(z-y)\times\\
\\
{\Big (} \left( \sum_{j,k=1}^n\left( f_{k,j}+g_{k,j}\right)(s,y) \right) \times\\
\\
\left( \sum_{j,k=1}^n\left( f_{j,k}-g_{j,k}\right)(s,y) \right) {\Big)}dy
\end{array}
\end{equation}
is H\"{o}lder continuous, hence the function $v^{f,\rho,l,2}_{i,j,k}-v^{g,\rho,l,2}_{i,j,k}$  is continuous. In order to do the contraction estimate it is useful to rewrite equation (\ref{v2*}) and equation (\ref{v1}).

Note that the function (\ref{source}) is solution of a Poisson equation with a right side which is H\"{o}lder uniformly in $\tau$. Hence, since the function  (\ref{source}) is in $C^{1+\alpha}$ with respect to the spatial variables and uniformly in $\tau$, we may rewrite (\ref{v2*}) shifting differentiation from the fundamental solution to the function (\ref{source}) (shifting one differentiation is enough). This way we have only first order derivatives of the fundamental solution in the representation of $v^{f,\rho,l,2}_{i,j,k}-v^{g,\rho,l,2}_{i,j,k}$, and first order derivatives of the fundamental solution are integrable and can be handled more easily. Note that the derivative of the fundamental solution is with respect to the $x$-variables and the integrals in (\ref{v1}) and (\ref{v2*}) are with respect to the conjugate $y$ variables. In order to shift the differentiation one may use
a relation of the fundamental solution and its adjoint, i.e., the relation
\begin{equation}\label{adjoint}
\Gamma^l_f(\tau,x;s,y)=\Gamma^{l,*}_f(s,y;\tau,x).
\end{equation}
where $\Gamma^{l,*}_f$ is the adjoint of $\Gamma^{l}_f$. Application of derivatives of this relation (\ref{adjoint}) is exactly what we need. However, let us show this in detail in terms of the Levy expansion of the fundamental solution (as a by-product, this is also an independent proof of (\ref{adjoint})).
First we shift the derivative $K_{,i}$ of the kernel $K$ in (\ref{v2*}). This is no problem since the inner integral with respect to $z$ represents a convolution. Hence, we have  

\begin{equation}\label{v2}
\begin{array}{ll}
v^{f,\rho,l,2}_{i,m,p}(\tau,x)-v^{g,\rho,l,2}_{i,m,p}(\tau,x):=\\
\\
=-\int_{(l-1)}^{\tau}\rho_l\int_{{\mathbb R}^n}\int_{{\mathbb R}^n}K_{n}(z-y)\times\\
\\
{\Big (} 2\left( \sum_{j,k=1}^n\left( f_{k,j,i}+g_{k,j,i}\right)(s,y) \right) \times\\
\\
\left( \sum_{j,k=1}^n\left( f_{j,k}-g_{j,k}\right)(s,y)\right) {\Big)}\Gamma^l_{f,m,p}(\tau,x;s,z)dydzds.
\end{array}
\end{equation}
\begin{rem}
Note that in this representation we have even only first derivatives of the difference $\left( f_{j,k}(s,y)-g_{j,k}(s,y)\right)$. This may be convenient in some circumstances (if we want a sharper estimate) but it is not necessary in case of our modest goal, i.e. the contraction on a subset of $C^{1,2}_b$ with respect to $|.|_{1,2}$ norm.
\end{rem}

Now let us show explicitly how we shift one derivative of $\Gamma^l_{f,m,p}$  
First, recall from classical theory that the fundamental solution $\Gamma^l_f$ in the Levy expansion form is given by
\begin{equation}
\Gamma^l_f(\tau,x;s,y):=N^l(\tau,x;s,y)+\int_s^{\tau}\int_{{\mathbb R}^n}N^l(\tau,x;\sigma,\xi)\phi(\sigma,\xi;s,y)d\sigma d\xi,
\end{equation}
where
\begin{equation}
N^l(\tau,x;s,y)=\frac{1}{\sqrt{4\pi \rho_l\nu (\tau-s)}^n}\exp\left(-\frac{|x-y|^2}{4\rho_l\nu (\tau-s)} \right),
\end{equation}
and $\phi$ is a recursively defined function which is H\"{o}lder continuous in $x$, i.e.,
\begin{equation}
\phi(\tau,x;s,y)=\sum_{m=1}^{\infty}(L_lN^l)_m(\tau,x;s,y),
\end{equation}
along with the recursion
\begin{equation}
\begin{array}{ll}
(L_lN^l)_1(\tau,x;s,y)=L_lN^l(\tau,x;s,y)\\
\\
=\frac{\partial N^l}{\partial \tau}-\rho_l\nu \Delta N^l+\rho_l\sum_{j=1}^n f_j\frac{\partial N^l}{\partial x_j}\\
\\
=\rho_l\sum_{j=1}^n f_j\frac{\partial N^l}{\partial x_j},\\
\\
(LN^l)_{m+1}(\tau ,x):=\int_s^t\int_{\Omega}\left( LN^l(\tau,x;\sigma,\xi)\right)_m LN^l(\sigma,\xi;s,y)d\sigma d\xi.
\end{array}
\end{equation}
First we look at (\ref{v2}). We may write 
\begin{equation}\label{v2a}
\begin{array}{ll}
v^{f,\rho,l,2}_{i,m,p}(\tau,x)-v^{g,\rho,l,2}_{i,m,p}(\tau,x):=\\
\\
=-\int_0^{\tau}\rho_l\int_{{\mathbb R}^n}S(s,y)\Gamma^l_{f,m,p}(\tau,x;s,z)dydzds
\end{array}
\end{equation}
along with a H\"{o}lder continuous function $S$ 
\begin{equation}\label{Ssy}
\begin{array}{ll}
S(s,y):=\int_{{\mathbb R}^n}K_{n}(z-y)\times\\
\\
{\Big (} 2\left( \sum_{j,k=1}^n\left( f_{k,j,i}+g_{k,j,i}\right)(s,y) \right) \times\\
\\
\left( \sum_{j,k=1}^n\left( f_{j,k}-g_{j,k}\right)(s,y) \right) {\Big)}dz,
\end{array}
\end{equation}
which is uniformly H\"{o}lder continuous with respect to the time parameter $s$. 
Now we may write
\begin{equation}\label{v2a}
\begin{array}{ll}
v^{f,\rho,l,2}_{i,m,p}(\tau,x)-v^{g,\rho,l,2}_{i,m,p}(\tau,x):=\\
\\
=-\int_0^{\tau}\rho_l\int_{{\mathbb R}^n}S(s,y)N^l_{,m,p}(\tau,x;s,y)ds dy\\
\\
+\int_s^{\tau}\int_{{\mathbb R}^n}
S(s,y)N^l_{,m,p}(\tau,x;\sigma,\xi)
\left( \sum_{q=1}^{\infty}(LN^l)_q(\sigma,\xi;s,y)\right)S(s,y)
d\sigma d\xi dydzds.
\end{array}
\end{equation}
We may do the partial integral term by term and get exactly the (second order differential of the) Levy expansion for the adjoint. Well it does not matter so much that we get exactly the adjoint. What matters is that we get an absolutely convergent sum where only first order derivatives of the kernel $N_l$ are involved, since such first order derivatives are locally integrable and can be estimated easily. We may denote the result of this partial integration for the Levy expansion sum by $\Gamma^{l,*}_f$ (which happens to be the adjoint).
Hence, from equation (\ref{v1}) we get
\begin{equation}\label{burgfg*2}
\begin{array}{ll}
v^{f,\rho,l,1}_{i,k,m}(\tau,x)-v^{g,\rho,l,1}_{i,k,m}(\tau,x)=\\
\\
 =-\rho_l\int_0^{\tau}\int_{{\mathbb R}^n}\sum_{j=1}^n\left( f_j-g_j\right)(s,y) v^{g,\rho,l}_{i,j}(s,y)\Gamma^l_{f,k,m}(\tau,x;s,y)dyds\\
 \\
=\rho_l\int_{(l-1)}^{\tau}\int_{{\mathbb R}^n}\sum_{j=1}^n\left( f_j-g_j\right)(s,y)v^{g,\rho,l}_{i,j,m}(s,y)\Gamma^{l,*}_{f,k}(s,y;\tau,x)dyds\\
\\
+\rho_l\int_{(l-1)}^{\tau}\int_{{\mathbb R}^n}\sum_{j=1}^n\left( f_{j,m}-g_{j,m}\right)(s,y)v^{g,\rho,l}_{i,j}(s,y)\Gamma^{l,*}_{f,k}(s,y;\tau,x)dyds.
\end{array}
\end{equation}
Next, we estimate the integral with integrand $\Gamma^{l,*}_{f,k}$ by a constant $C_{\Gamma_f}$ using estimates for local integrability of the first order derivatives of the fundamental solution in Levy expansion form. Here, we may use standard estimates of the form
\begin{equation}\label{gammafl}
|\Gamma^{l,*}_{f,k}|\leq \frac{C_f}{\sqrt{4\pi (\tau-s)}^{n+1}}\exp\left(-\lambda_f\frac{|x-y|^2}{4\rho_l\nu(\tau-s)} \right),
\end{equation}
for some constants $C,\lambda$ and for $|x-y|\geq 1$ in order to have integrability at infinity. Furthermore we may use 
\begin{equation}\label{Nlest}
 |N^l(\tau-s,x,y)|\leq \frac{C}{(\tau-s)^{\alpha}|x-y|^{n+1-2\alpha}}
\end{equation}
for $\alpha\in (0.5,1)$ on a domain where $|x-y|\leq 1$ in order to have local integrability. We can formulate have both standard estimates independently of  $\mathbf{f}\in S_0$. We get
\begin{equation}\label{gammafl}
\sup_{\mathbf{f}\in S_0}|\Gamma^{l,*}_{f,k}|\leq \frac{C}{\sqrt{4\pi (\tau-s)}^{n+1}}\exp\left(-\lambda\frac{|x-y|^2}{4\rho_l\nu(\tau-s)} \right),
\end{equation}
where
\begin{equation}
C=\sup_{\mathbf{f}\in S_0}C_f,
\end{equation}
and
\begin{equation}
\lambda=\inf_{\mathbf{f}\in S_0}\lambda_f.
\end{equation}
Furthermore, we get
\begin{equation}
|\Gamma^{l,*}_{f,k}(\tau,x;s,y)|\leq \frac{C^+}{(\tau-s)^{\mu}|x-y|^{n+1-2\mu}},
\end{equation}
for some $\mu\in (0.5,1)$ and with some constant $C^+$, hence, locally integrability with respect to time and space.
This implies that there exists $C_{\Gamma}>0$ such that
 \begin{equation}
\sup_{(\tau,x)\in D^{\tau}_l}\int_{l-1}^{\tau}\int_{{\mathbb R}^n}\sup_{\mathbf{f}\in S_0}|\Gamma^{l,*}_f|dyds \leq  C_{\Gamma},
 \end{equation}
 and such that
\begin{equation}
\sup_{(\tau,x)\in D^{\tau}_l}\int_{l-1}^{\tau}\int_{{\mathbb R}^n}\sup_{\mathbf{f}\in S_0}|\Gamma^{l,*}_{f,i}|dyds \leq  C_{\Gamma},
 \end{equation} 
 where $1\leq i\leq n$.
Hence, since $\mathbf{g}\in S_0$, for the second order derivatives of the first summand (\ref{firstsum}) we get 
\begin{equation}\label{xxx}
\begin{array}{ll}
|v^{f,\rho,l,1}_{i,k,m}-v^{g,\rho,l,1}_{i,k,m}|_0\leq \\
\\
\rho_l\int_{(l-1)}^{\tau}\int_{{\mathbb R}^n}\sum_{j=1}^n| f_j-g_j|_0|\mathbf{v}^{g,\rho,l}|_{0,2}C_{\Gamma}\\
\\
+\rho_l\int_{(l-1)}^{\tau}\int_{{\mathbb R}^n}\sum_{j=1}^n| f_{j,m}-g_{j,m}|_0|\mathbf{v}^{g,\rho,l}|_{0,1}C_{\Gamma}\\
\\
\leq \rho_l\int_{(l-1)}^{\tau}\int_{{\mathbb R}^n}| \mathbf{f}-\mathbf{g}|^n_{0,1}|\mathbf{v}^{g,\rho,l}|^n_{0,2}C_{\Gamma}\\
\\
\leq \rho_lC^{l-1}_{1,2}C_{\Gamma}^2| \mathbf{f}-\mathbf{g}|^n_{0,1}.
\end{array}
\end{equation}
Note that in this estimate the second $C_{\Gamma}$ comes from the estimate of $|\mathbf{v}^{g,\rho,l}|^n_{0,2}$. 
From equation (\ref{v2}) we get
\begin{equation}\label{yyy}
\begin{array}{ll}
|v^{f,\rho,l,2}_{i,l,m}-v^{g,\rho,l,2}_{i,l,m}|_0:=\\
\\
\leq\int_{(l-1)}^{\tau}\rho_l\int_{{\mathbb R}^n}\int_{{\mathbb R}^n}K_{n,m}(z-y)\times\\
\\
{\Big (} \left( \sum_{j,k=1}^n| f_{k,j,i}+g_{k,j,i}|_0 \right) \times\\
\\
\left( \sum_{j,k=1}^n|f_{j,k,i}-g_{j,k,i}|_0\right)  {\Big)}C_{\Gamma}\\
\\
\leq \rho_lC_Kn^24C^{l-1}_{1,2}C_{\Gamma}| \mathbf{f}-\mathbf{g}|^n_{0,2}.
\end{array}
\end{equation}
Here the constant $C_K$ is from the spatial integral with integrand $K_{n,m}$. We may assume without loss of generality that $C_{\Gamma}\geq 1$. The constant $C_K$ depends on dimension only and may be subsumed by a constant $C^*_n$. In the following we shall use the latter constant generically if we want to subsume all constants depending only on dimension.
We summarize that equation (\ref{xxx}) and (\ref{yyy}) gives 
\begin{equation}\label{xx0}
\sum_{j,m=1}^n|v^{f,\rho,l}_{i,j,m}-v^{g,\rho,l}_{i,j,m}|_{0}\leq \rho_lC^*_n C^2_{\Gamma}C^{l-1}_{1,2}| \mathbf{f}-\mathbf{g}|^n_{0,2}.
\end{equation}
Similarly we get
\begin{equation}\label{xxx1}
\begin{array}{ll}
\sum_{m=1}^n|v^{f,\rho,l}_{i,m}-v^{g,\rho,l}_{i,m}|_0\leq 
\rho_lC^*_nC_{\Gamma}^2C^{l-1}_{1,2}| \mathbf{f}-\mathbf{g}|^n_{0,2}.
\end{array}
\end{equation}
and
\begin{equation}\label{xxx2*}
\begin{array}{ll}
|v^{f,\rho,l}_{i}-v^{g,\rho,l}_{i}|_0\leq 
\rho_lC^*_nC^2_{\Gamma}C^{l-1}_{1,2}| \mathbf{f}-\mathbf{g}|^n_{0,2}.
\end{array}
\end{equation}
An estimate for the time derivatives of $v^{f,\rho,l}_{i}$ and $v^{g,\rho,l}_{i}$ can be reduced to the estimates (\ref{xx0}), (\ref{xxx1}), and (\ref{xxx2*}) via the defining equation for $\Gamma^l_f$
\begin{equation}\label{gammaft}
\begin{array}{ll}
\frac{\partial \Gamma^l_f}{\partial \tau}=\rho_l\nu\sum_{j=1}^n \frac{\partial^2 \Gamma^l_f}{\partial x_j^2} 
-\rho_l\sum_{j=1}^n f_j\frac{\partial \Gamma^l_f}{\partial x_j}=0.
\end{array}
\end{equation}
\begin{rem}
Alternatively we could use defining equations 
\begin{equation}\label{Navburgft}
\begin{array}{ll}
\frac{\partial v^{f,\rho,l}_i}{\partial \tau}=\rho_l\nu\sum_{j=1}^n \frac{\partial^2 v^{f,\rho,l}_i}{\partial x_j^2} 
-\rho_l\sum_{j=1}^n f_j\frac{\partial v^{f,\rho,l}_i}{\partial x_j}+\\
\\
\int_{{\mathbb R}^n}\left( \frac{\partial}{\partial x_i}K_n(x-y)\right) \sum_{j,k=1}^n\left( \frac{\partial f_k}{\partial x_j}\frac{\partial f_j}{\partial x_k}\right) (\tau,y)dy
\end{array}
\end{equation}
for the functions $ v^{f,\rho,l}_i$ and analogous defining equations for the functions $ v^{g,\rho,l}_i$. However this would enlarge the estimate constant. In this case we get another factor $C_K$ and another factor $C^l_{1,2}$. The method of growth control described in step 3 still applies. 
 \end{rem}
 
 Summarizing we get
\begin{equation}\label{xxx2}
\begin{array}{ll}
|\mathbf{v}^{f,\rho,l}-\mathbf{v}^{g,\rho,l}|^n_{1,2}=\sum_{i=1}^n|v^{f,\rho,l}_{i}-v^{g,\rho,l}_{i}|_{1,2}\\
\\
\leq 
\rho_lC^*_n C_{\Gamma}^2C^{l-1}_{1,2}| \mathbf{f}-\mathbf{g}|^n_{0,2}\\
\\
\leq \rho_lC^*_nC^2_{\Gamma}C^{l-1}_{1,2}| \mathbf{f}-\mathbf{g}|^n_{1,2}.
\end{array}
\end{equation}
for some generic constant $C^*_n>0$ depending only on dimension $n$. We may choose
\begin{equation}
\rho_l\leq \frac{1}{4C^*_n C_{\Gamma}^2C^{l-1}_{1,2}}.
\end{equation}
Then for $\mathbf{f}, \mathbf{g}\in S_0$ we have 
\begin{equation}\label{xxx2}
\begin{array}{ll}
|\mathbf{v}^{f,\rho,l}-\mathbf{v}^{g,\rho,l}|^n_{1,2}
\leq \frac{1}{4}| \mathbf{f}-\mathbf{g}|^n_{1,2}.
\end{array}
\end{equation}
Note that 
\begin{equation}
\mathbf{f}, \mathbf{g}\in S_0\rightarrow \mathbf{v}^{f,\rho,l}\in S_0
\end{equation}
for this choice of $\rho_l$.
Furthermore, note that the constant $C_{\Gamma}$ depends on the class $S_0$.  Hence, we do not have a contraction on the whole space but on a subspace $S_0$ of $C^{1,2}_b$. Since the definition of $S_0$ contains a factor $2$ and we start with $|v^{\rho,l-1}|\leq C^{l-1}_{1,2}$ the iterations are $F_l^k(\mathbf{v}^{\rho,l-1}),~k\geq 1$ belong to $S_0$ indeed by choice of $\rho_l$.
Hence, the contraction condition leads to an iteration in a subspace $S_0$ of $C^{1,2}_b$. However, note that this is not a Banach space. In order to have the functional series in a Banach space or in order to have the limit of the functional series $\mathbf{v}^{\rho,l}=\lim_{k\uparrow \infty}{\mathbf v}^{\rho,k,l}$ known to be in some regular space, we need some decay at spatial infinity. Let us summarize first what we have achieved so far on the way of the construction of a local solution. Iteration of the map $F_l$ starting with the function $\mathbf{v}^{\rho,l-1}$ leads to a series of functions $\mathbf{v}^{\rho,k,l}=F_l^{k+1}\left(\mathbf{v}^{\rho,l-1}\right) $
where the components of the functions  $\mathbf{v}^{\rho,k,l}$ satisfy
\begin{equation}
v^{\rho,k,l}_i=v^{\rho,0,l}_i+\sum_{m=1}^k\delta v^{\rho,m,l}_i\in C^{1,2}_b,
\end{equation}
and where the increments $\delta v^{\rho,m,l}_i$ satisfy a contraction property 
\begin{equation}\label{xxx3}
\begin{array}{ll}
|\delta \mathbf{v}^{\rho,k,l}|^n_{1,2}
\leq \rho_lC^*_nC_{\Gamma}^2C^{l-1}_{1,2}| \delta \mathbf{v}^{\rho,k-1,l}|^n_{1,2}\\
\\
\leq \frac{1}{4}| \delta \mathbf{v}^{\rho,k-1,l}|^n_{1,2}.
\end{array}
\end{equation}
The approximations $v^{\rho,k,l}_i$ can be represented in terms of  
 fundamental solution $\Gamma^l_k$, where for $k\geq 1$ $\Gamma^l_k$ is the fundamental solution of the equation
\begin{equation}\label{gammak}
\frac{\partial \Gamma^l_k}{\partial \tau}-\rho_l\nu \Delta \Gamma^l_k+\rho_l\sum_{j=1}^n v^{\rho,k-1,l}_j\frac{\partial \Gamma^l_k}{\partial x_j}=0,
\end{equation}
and for $k=0$ the function $\Gamma^l_0$ is the fundamental solution of the equation
\begin{equation}\label{gamma0}
\frac{\partial \Gamma^l_0}{\partial \tau}-\rho_l\nu \Delta \Gamma^l_0+\rho_l\sum_{j=1}^n v^{\rho,l-1}_j\frac{\partial \Gamma^l_0}{\partial x_j}=0.
\end{equation}
The choice of $\rho_l$ ensures that the first order coefficients  in (\ref{gammak}) and (\ref{gamma0}) have a uniform bound.
This leads to the estimate of integrals involving integrands $\Gamma_k$ by a uniform constant $C_{\Gamma}$ such that
\begin{equation}
|\mathbf{v}^{\rho,k,l}|^n_{1,2}=|\mathbf{v}^{\rho,l-1}+\sum_{m=1}^k\delta \mathbf{v}^{\rho,m,l}|^n_{1,2}\leq 2C^{l-1}_{1,2},
\end{equation}
and, hence, to a functional series in $S_0$. This leads to the contraction property of the series $F^k_l\left(\mathbf{v}^{\rho,l-1}\right)$ with respect to the norm $|.|_{1,2}$, and we have (\ref{deltacontract}) especially. 

Next in order to apply method a) of the introduction we show that the elements of the functional series $v^{\rho,k,l}_i(\tau,.)$ are in $H^2$ uniformly with respect to $\tau \in [l-1,l]$ and for all $1\leq i\leq n$. This means that we extend the estimate (\ref{xxx3}) above to the space of functions $\mathbf{f}$ with components 
\begin{equation}
f_i\in H^{2}_l:=\left\lbrace h\in C^{1,2}_b(D^{\tau}_l)| h(\tau,.)\in H^2~~\mbox{unif. for all}~\tau\in [l-1,l] \right\rbrace .
\end{equation}
Here 'unif.' is an abbreviation for uniformly.
This is a good space in order to have a finite magnitude (\ref{vrhonorm}), since we have 
\begin{equation}\label{vrhonormH2}
\begin{array}{ll}
 \int_{{\mathbb R}^n}\sum_{j,k=1}^n{\Bigg |}\left( \frac{\partial v^{\rho,l}_k}{\partial x_j}\frac{\partial v^{\rho,l}_j}{\partial x_k}\right) (t,y){\Bigg |}dy\\
 \\
\int_{{\mathbb R}^n}\sum_{j,k=1}^n{\Bigg |}\left( \frac{1}{2}\left( \frac{\partial v^{\rho,l}_k}{\partial x_j}\right) ^ 2+\frac{1}{2}\left( \frac{\partial v^{\rho,l}_j}{\partial x_k}\right)\right)^2 (t,y){\Bigg |}dy 
 <\infty .
 \end{array}
 \end{equation}
if $v^{\rho,l}_i\in H^1$. We choose the space $H^2$ since we have first derivatives of  (\ref{vrhonorm}) involved in our estimation.

Recall that method a) of the introduction applies in the case $n=3$.
Note that a converging series
\begin{equation}\label{vkH}
\mathbf{v}^{\rho,k,l}(\tau,.)\in \left[ H^s\left({\mathbb R}^n\right)\right]^n~~\mbox{ for some }~s\geq \alpha+\frac{1}{2}n
\end{equation}
(uniformly in $\tau$) implies that
\begin{equation}\label{vH}
\mathbf{v}^{\rho,l}(\tau,.)\in \left[ H^s\left({\mathbb R}^n\right)\right]^n~~\mbox{ for some }~s\geq \alpha+\frac{3}{2}.
\end{equation}
Hence, in case $n=3$ we have
\begin{equation}
\mathbf{v}^{\rho,l}(\tau,.)\in 
\left[ C^{\alpha}\left( {\mathbb R}^n\right)\right]^n
\end{equation}
for $\alpha\in (0,0.5)$ uniformly in $\tau$. Hence, if we can show that (\ref{vkH}) holds (which implies (\ref{vH}), then  we have bounded H\"{o}lder continuous first order coefficients in the equation
\begin{equation}\label{Navlerayl*}
 \begin{array}{ll}
\frac{\partial \Gamma^{l}_v}{\partial \tau}-\rho_l\nu\sum_{j=1}^n \frac{\partial^2 \Gamma^{l}_v}{\partial x_j^2} 
+\rho_l\sum_{j=1}^n v^{\rho,l}_j\frac{\partial \Gamma^{l}_v}{\partial x_j}=0,
\end{array}
\end{equation}
and this means that a fundamental solution $\Gamma^l_v$ of (\ref{Navlerayl*}) exists. Then we have the representation
\begin{equation}
\begin{array}{ll}
v^{\rho,l}_i(\tau,x)=\int_{{\mathbb R}^n}v^{\rho,l-1}_i(l-1,y)\Gamma^{l}_v(\tau,x;l-1,y)dy+\\
\\
\int_{l-1}^{\tau}\int_{{\mathbb R}^n}\rho_l\int_{{\mathbb R}^n}\left( \frac{\partial}{\partial x_i}K_n(y-z)\right) \left( \sum_{j,k=1}^n\left( \frac{\partial v^{\rho,l}_k}{\partial x_j}\frac{\partial v^{\rho,l}_j}{\partial x_k}\right) (\tau,z)\right) \times\\
\\
\times\Gamma^l_v(\tau,x,s,y)dzdyds
\end{array}
\end{equation}
and from this representation we immediately get
\begin{equation}
\mathbf{v}^{\rho,l}\in \left[  C^{1,2}_b\left(D^{\tau}_l\right)\right]^n.   
\end{equation}
Now let us finish the proof in the case $n=3$ using method a) of the introduction. We have to show that
\begin{equation}
v^{\rho,k,l}_i\left(\tau,.\right)\in H^2\left({\mathbb R}^n \right),
\end{equation}
where we show that the latter relation holds uniformly in $\tau$. First, inductively with respect to $l$ we have 
\begin{equation}
v^{\rho,l-1}_i(l-1,.)\in C^{2}_b~\mbox{ and }~v^{\rho,l-1}_i\left(l-1,.\right)\in H^2\left({\mathbb R}^n \right). 
\end{equation}
First we show that
\begin{equation}
v^{\rho,k,l}_i(\tau,.)\in L^{2}
\end{equation}
uniformly in $\tau$. The reasoning for the first and second derivatives below requires only a little more work. We start with $k=0$.
Note that the equation which defines $v^{\rho,0,l}_i$ has first order coefficient functions $v^{\rho,l-1}_i(l-1,.)$ which are independent of time. Hence the fundamental solution $\Gamma^l_0$ of the equation
\begin{equation}\label{Navleraylgamma0}
 \begin{array}{ll}
\frac{\partial \Gamma^{l}_0}{\partial \tau}-\rho_l\nu\sum_{j=1}^n \frac{\partial^2 \Gamma^{l}_0}{\partial x_j^2} 
+\rho_l\sum_{j=1}^n v^{\rho,l-1}_j\frac{\partial \Gamma^{l}_0}{\partial x_j}=0
\end{array}
\end{equation}
exists. For this fundamental solution  we have a priori estimates. Especially, we have
\begin{equation}
|\Gamma^l_0(\tau,x;s,y)|\leq \frac{C}{(\tau-s)^{n/2}}\exp\left(-\lambda\frac{|x-y|^2}{\tau-s}\right)  
\end{equation}
for some constants $C$ and $\lambda$ (the constants $\lambda$ and $C$ are used generically here).
Next we have the representation
\begin{equation}
\begin{array}{ll}
v^{\rho,0,l}_i(\tau,x)=\int_{{\mathbb R}^n}v^{\rho,l-1}_i(l-1,y)\Gamma^{l}_0(\tau,x;l-1,y)dy+\\
\\
\int_{(l-1)}^{\tau}\int_{{\mathbb R}^n}\rho_l\int_{{\mathbb R}^n}\left( K_{n,i}(y-z)\right) \left( \sum_{j,k=1}^n\left( \frac{\partial v^{\rho,l-1}_k}{\partial x_j}\frac{\partial v^{\rho,l-1}_j}{\partial x_k}\right) (\tau,z)\right) \times\\
\\
\times\Gamma^l_0(\tau,x,s,y)dzdyds
\end{array}
\end{equation}
Note that for $\tau =l-1$ we have $v^{\rho,l-1}(l-1,.)\in H^2$ inductively.
We get
\begin{equation}
\begin{array}{ll}
|v^{\rho,0,l}_i(\tau,x)|\leq \int_{{\mathbb R}^n}
|v^{\rho,l-1}_i(l-1,y)|\frac{C}{(\tau-(l-1))^{n/2}}\exp\left(-\lambda\frac{|x-y|^2}{\tau-(l-1)}\right)dy+\\
\\
\int_{(l-1)}^{\tau}\int_{{\mathbb R}^n}\rho_l{\Big |}\int_{{\mathbb R}^n}\left( K_{n,i}(y-z)\right) \left( \sum_{j,k=1}^n\left( \frac{\partial v^{\rho,l-1}_k}{\partial x_j}\frac{\partial v^{\rho,l-1}_j}{\partial x_k}\right) (s,z)\right){\Big |} \times\\
\\
\times\frac{C}{(\tau -s)^{n/2}}\exp\left(-\lambda\frac{|x-y|^2}{\tau -s}\right) dzdyds\\
\\
=:(I)+(II).
\end{array}
\end{equation}
Now (I) and (II) are convolutions. We may apply Young's inequality, i.e.,
\begin{equation}
|f\star g|_r\leq |f|_q|g|_p,
\end{equation}
where
\begin{equation}
1+\frac{1}{r}=\frac{1}{p}+\frac{1}{q},
\end{equation}
and $1\leq p,q,r\leq \infty$.
Note that the family of Gaussian functions
\begin{equation}\label{gaussian}
x\rightarrow \frac{C}{(\tau -s)^{n/2}}\exp\left(-\lambda\frac{|x|^2}{\tau -s}\right)
\end{equation}
(family with respect to time parameters) is in $L^p$ for all $1\leq p<\infty$ (although not in $L^{\infty}$ as $s\uparrow \tau$). 
We know that the first term (I) is in $L^1$ because $v^{\rho,l-1}_i\in H^2$ inductively and the Gaussian is in $L^2$. The idea is that locally the convolution is in $L^2$ for fixed $\tau$ (since it is in $C^1$) and globally you may shift a derivative to the kernel. Second derivatives of the kernel $K$ are square integrable outside a ball, since 
\begin{equation}
K_{,i,j}\sim\frac{x_ix_j}{|x|^{n+2}}
\end{equation}
 for $n\geq 3$ and $i\neq j$ is $L^2$ outside a ball around zero. Similarly for $i=j$.
The details of this estimate are described in \cite{KB2}. Note that we may set the function in (\ref{gaussian}) to zero if $\tau=s$ in the time integral rom $s$ to $\tau$. Furthermore, applying Young's inequality for fixed $t>s$ first we see that the integrand for the second term is in $L^2$ where $v^{\rho,l-1}\in H^2$ inductively. Since the $L^2$-bound of the integrand is uniform the time integral $(II)$ is also $L^2$. Here again, for each $x\in {\mathbb R}^n$ you may use decompose the integral in two parts $[0,\tau]\times B_x$ and $[0,\tau]\times {\mathbb R}^n\setminus \left( [0,\tau]\times B_x\right) $, where $B_x$ is a n-dimensional ball around $x$ of finite radius.
Then on the ball we may use the estimate
\begin{equation}
\frac{1}{(\tau -s)^{n/2}}\exp\left(-\lambda\frac{|x-y|^2}{\tau -s}\right)\leq \frac{C}{(t-s)^{\mu}|x-y|^{n-2\mu}}
\end{equation}
for $\mu \in (0,1)$. Outside the ball we may use the estimate
\begin{equation}
{\big |}\Gamma^l_0(\tau,x;s,y){\big |}\leq \frac{C}{(\tau -s)^{n/2}}\exp\left(-\lambda\frac{|x-y|^2}{\tau -s}\right).
\end{equation}
Hence we have
\begin{equation}
v^{\rho,0,l}_i\left(\tau,.\right)\in L^2\left({\mathbb R}^n \right) 
\end{equation}
uniformly in $\tau$. 
Next note that we have representations of the first and second spatial derivatives of $v^{\rho,0,l}_{i}$ which involve only the first spatial derivative of the fundamental solution (for the second spatial derivative we use the adjoint and partial integration). Then an analogous argument as above which uses the first derivative estimates 
\begin{equation}\label{firstder}
\frac{\partial}{\partial x_i}\frac{1}{(\tau -s)^{n/2}}\exp\left(-\lambda\frac{|x-y|^2}{\tau -s}\right)\leq \frac{C}{(t-s)^{\mu}|x-y|^{n+1-2\mu}}
\end{equation}
for $\mu\in (0.5,1)$,
and 
\begin{equation}
{\big |}\Gamma^l_0(\tau,x;s,y){\big |}\leq \frac{C}{(\tau -s)^{(n+1)/2}}\exp\left(-\lambda\frac{|x-y|^2}{\tau -s}\right).
\end{equation}
leads to the conclusion that
\begin{equation}
v^{\rho,0,l}_i\left(\tau,.\right)\in H^2\left({\mathbb R}^n \right) 
\end{equation}
uniformly in $\tau$. Next assuming that
\begin{equation}
v^{\rho,k-1,l}_i\in C^{1,2}_b~\mbox{ and }~v^{\rho,k-1,l}_i\left(\tau,.\right)\in H^2\left({\mathbb R}^n \right)  
\end{equation}
uniformly in $\tau\in [l-1,l]$ we know that $v^{\rho,k,l}_i\in C^{1,2}_b$ and we want to show that
\begin{equation}
~v^{\rho,k,l}_i\left(\tau,.\right)\in H^2\left({\mathbb R}^n \right)  
\end{equation}
uniformly in $\tau$. It is convenient to use the series $\left( \mathbf{v}^{\rho,k,l}\right)_k$ along with
\begin{equation}\label{vrhokl}
\mathbf{v}^{\rho,k,l}=\mathbf{v}^{\rho,0,l}+\sum_{m=1}^k\delta \mathbf{v}^{\rho,m,l}.
\end{equation} 
Then it suffices to show a
a contraction property with respect to the $H^2$-norm in the sense that
\begin{equation}\label{contractdeltavk}
|\delta v^{\rho,k,l}_i(\tau,.)|_{H^2}\leq \frac{1}{4} |\delta v^{\rho ,k-1,l}_i(\tau,.)|_{H^2}
\end{equation}
for all $k\geq 1$.
This holds for some generic constant $C^*_n$ if we use a constant $\rho_l$ with 
\begin{equation}
\rho_l\leq \frac{1}{4C^*_n C_{\Gamma}^2C^{l-1}_{1,2}}.
\end{equation}
We consider the main steps in order ro prove this relation.
Note that the equation which defines $v^{\rho,k,l}_i$ has first order coefficient functions $v^{\rho,k-1,l}_i\in C^{1,2}_b$.  
Hence the fundamental solution $\Gamma^l_k$ of the equation
\begin{equation}\label{Navleraylgamma0}
 \begin{array}{ll}
\frac{\partial \Gamma^{l}_k}{\partial \tau}-\rho_l\nu\sum_{j=1}^n \frac{\partial^2 \Gamma^{l}_k}{\partial x_j^2} 
+\rho_l\sum_{j=1}^n v^{\rho,k-1,l}_j\frac{\partial \Gamma^{l}_k}{\partial x_j}=0
\end{array}
\end{equation}
exists. Since we have  $|v^{\rho,k-1,l}_j|_{1,2}\leq 2C^l_{1,2}$  independent of the number $k$. For this fundamental solution  we have a priori estimates. Especially, we have
\begin{equation}\label{gammakest}
|\Gamma^l_k(\tau,x;s,y)|\leq \frac{C}{(\tau-s)^{n/2}}\exp\left(-\lambda\frac{|x-y|^2}{\tau-s}\right)  
\end{equation}
for some $\lambda,C>0$ independent of the number $k$.
Next we have the representation
\begin{equation}
\begin{array}{ll}
\delta v^{\rho,k,l}_i(\tau,x)=-\rho_l\int_{(l-1)}^{\tau}\int_{{\mathbb R}^n}\left( \sum_{j=1}^n\delta v^{\rho,k-1,l}_j\frac{\partial v^{\rho ,k-1,l}}{\partial x_j}\right) (s,y)\Gamma^{l}_k(\tau,x;s,y)ds dy+\\
\\
\int_{(l-1)}^{\tau}\int_{{\mathbb R}^n}\rho_l\int_{{\mathbb R}^n}\left( \frac{\partial}{\partial x_i}K_n(y-z)\right) \left( \sum_{j,m=1}^n\left( \frac{\partial v^{\rho,k-1,l}_m}{\partial x_j}\frac{\partial v^{\rho,k-1,l}_j}{\partial x_m}\right) (\tau,z)\right) \times\\
\\
\times\Gamma^l_k(\tau,x,s,y)dzdyds\\
\\
-\int_{(l-1)}^{\tau}\int_{{\mathbb R}^n}\rho_l\int_{{\mathbb R}^n}\left( \frac{\partial}{\partial x_i}K_n(y-z)\right) \left( \sum_{j,m=1}^n\left( \frac{\partial v^{\rho,k-2,l}_m}{\partial x_j}\frac{\partial v^{\rho,k-2,l}_j}{\partial x_m}\right) (\tau,z)\right) \times\\
\\
\times\Gamma^l_k(\tau,x,s,y)dzdyds
\end{array}
\end{equation}
This leads to
\begin{equation}\label{delta v}
\begin{array}{ll}
\delta v^{\rho,k,l}_i(\tau,x)=-\rho_l\int_{(l-1)}^{\tau}\int_{{\mathbb R}^n}\left( \sum_{j=1}^n\delta v^{\rho,k-1,l}_j\frac{\partial v^{\rho ,k-1,l}}{\partial x_j}\right) (s,y)\Gamma^{l}_k(\tau,x;s,y)ds dy+\\
\\
\int_{(l-1)}^{\tau}\rho_l\int_{{\mathbb R}^n}\int_{{\mathbb R}^n}K_{n,i}(z-y)\times\\
\\
{\Big (} \left( \sum_{j,m=1}^n\left( v^{\rho ,k-1,l}_{m,j}(s,y)+v^{\rho,k-2,l}_{m,j}(s,y)\right) \right) \times\\
\\
\left( \sum_{j,m=1}^n\left( v^{\rho,k-1,l}_{j,m}(s,y)-v^{\rho,k-2,l}_{j,m}(s,y)\right) \right) {\Big)}\Gamma^l_k(\tau,x;s,z)dydzds,\\
\end{array}
\end{equation}
Hence, using a priori estimates analogously as above in case $k=0$ we get (\ref{contractdeltavk}). This finishes the proof in the case $n=3$.
The method a) as it is presented above applies in the case $n=3$. In case $n>3$ we have to ensure that the convergence of the series $v^{\rho,k,l}_i(\tau,.)$ is in $H^{s,p}$ for $p$ large enough. This can be done using the same a priori estimates. Method b) provides an alternative which may be even more elementary.  
This direct (more elementary) way shows that the limit of the functional series exists in a Banach space by transforming on a compact domain.
Such a transformation can be applied if for each $k\geq 0$ we have
\begin{equation}\label{growthuburg*}
    \vert \partial^\alpha_x \mathbf{v}^{\rho,k,l}(t,x)\vert\le \frac{C_{\alpha k}}{(1+\vert x\vert)^5}~\mbox{if}~|\alpha|\leq 2\qquad . 
\end{equation}
for $|\alpha|\leq 2$ and $k\leq 5$. Then we may consider  the transformation of spatial variables
\begin{equation}
x_i\rightarrow y_i:=\arctan\left(x_i\right) 
\end{equation}
from ${\mathbb R}^n$ to $\left] -\frac{\pi}{2},\frac{\pi}{2}\right[ ^n$ and consider the transformed function
\begin{equation}
w^{f,\rho,l}_i(\tau,y):=v^{f,\rho,l}_i(\tau,x)
\end{equation}
for $1\leq i\leq n$. We have
\begin{equation}
\frac{\partial v^{f,\rho,l}_i}{\partial x_m}=\frac{\partial w^{f,\rho,l}_i}{\partial y_m}\frac{\partial y_m}{\partial x_m}=\frac{\partial w^{f,\rho,l}_i}{\partial y_m}\frac{1}{1+x_m^2}
=\frac{\partial w^{f,\rho,l}_i}{\partial y_m}\frac{1}{1+\tan^2(y_m)},
\end{equation}
and
\begin{equation}
\begin{array}{ll}
\frac{\partial^2 v^{f,\rho,l}_i}{\partial x_m^2}=
\frac{\partial^2 w^{f,\rho,l}_i}{\partial y_m^2}\left( \frac{\partial y_m}{\partial x_m}\right)^2+
\frac{\partial w^{f,\rho,l}_i}{\partial y_m}\frac{\partial^2 y_m}{\partial x_m^2}=\\
\\
\frac{\partial^2 w^{f,\rho,l}_i}{\partial y_m^2}\left( \frac{1}{1+\tan^2(y_m)}\right)^2
+\frac{\partial w^{f,\rho,l}_i}{\partial y_m}\frac{-2\tan(y_m)}{\left( 1+\tan^2(y_m)\right)^4}.
\end{array}
\end{equation}
Since 
\begin{equation}
\left(1+|x|^4\right)\frac{\partial^2 v^{f,\rho,l}_i}{\partial x_m^2}\downarrow 0~~\mbox{as}~~|x|\uparrow \infty
\end{equation}
Hence we get a series of functions $\left( \mathbf{w}^{\rho,k,l}\right)_k$ defined on transformed domain $D^{\tau,\pi}_l:=[l-1,l]\times \left]-\pi/2,\pi/2 \right[^n$ along with
\begin{equation}
\mathbf{w}^{\rho,l,k}(\tau,y)=\mathbf{v}^{\rho,l,k}(\tau,x)
\end{equation}
and the limit is in the Banach space 
$\left[ C_{1+\alpha/2,2+\alpha}\left(D^{\tau,\pi}_l\right) \right]^n$
 with the norm $|.|^n_l$ restricted to the domain  $D^{\tau,\pi}_l$. 
 Note that we have  $w^{\rho,k,l}_i(\tau,\mathbf{s})=0$ for all $1\leq i\leq n$ 
if for some $s_j$ in $\mathbf{s}=(s_1,\cdots s_n)$ we have $s_j\in \left\lbrace-\pi/2,\pi/2\right\rbrace $ according to the polynomial decay of $\mathbf{v}^{\rho,k,l}$ at infinity. Hence we get a limit $\mathbf{w}^{\rho,l}$ with $\mathbf{w}^{\rho,l}(\tau,y)=\mathbf{v}^{\rho,l}(\tau,x)$ for all $(\tau,x)\in D^{\tau}_l$ and $(\tau,y)\in D^{\tau,\pi}_l$ where satisfies the transformed Navier-Stokes equation (in $\tau,y$-coordinates on $D^{\tau,\pi}_l$).

\subsection{step 2: Extension of local existence to extended equations for $\mathbf{v}^{r,\rho,l}$}
In step 1 we did not consider the growth of the solution. We just proved that the iterations $F_l^m(\mathbf{v}^{\rho,l-1}),~l\geq 1$ of the functional $F_l$ are in a function set $S_0$ such that first order terms of the linear subproblems are uniformly bounded.  
In order to get a globally bounded scheme we introduce a control function $\mathbf{r}=(\mathbf{r}^l)_l$ and set up a globally bounded scheme for $\mathbf{v}^{r,\rho}=\mathbf{v}^{\rho}+\mathbf{r}$.
In this second step we observe that the extension of the local iteration of step 1 to equations with bounded regular functions ${\mathbf r}\neq 0$ does not change the argument of step 1 essentially. Note that the Leray projection from the equation for $\mathbf{v}^{r,\rho}$ includes integral terms for $r^l_i$ as described in the introduction. This means that we have to estimate magnitudes of the from
\begin{equation}\label{rnormstep2}
\int_{{\mathbb R}^n}\sum_{j,k,m,p=1}^n{\Big |}\frac{\partial r^l_k}{\partial x_j}{\Big |}{\Big |}\frac{\partial r^l_m}{\partial x_p}{\Big |}(\tau,y)dy.
\end{equation}
Again this reduces to $H^2$ estimates.
The similarity of the convergence argument of the local scheme for $\mathbf{v}^{r,\rho,l}$ to the argument for convergence of the local scheme of step 1 is due to the fact that the additional terms symbolized by the term $L^{\rho,l}_i\left(\mathbf{r},\mathbf{v}\right)$ in the equation for $\mathbf{v}^{r,\rho,l}$ consist of a source term dependent on the function $\mathbf{r}$ (which we called $S^l(\mathbf{r})$), and a linear operator dependent on the function  $\mathbf{v}^{r,\rho,l}$ which we called $L^{\rho,l,0}_i\left(\mathbf{r},\mathbf{v}\right)$ (resp. $L^{\rho,l,0}_i\left(\mathbf{r},\mathbf{v}^{r,\rho,k,l}\right)$  for each step of the local iteration). Note that both additional terms have a multiplier $\rho_l$ with the exception of the time derivative of $r^l_{i,\tau}$. If $\mathbf{r}^l$ is regular for given time step number $l$, i.e., if the first time derivative and the second time derivatives exist and are integrable on the domain $(l-1,l]\times {\mathbb R}^n$, then we get a regular solution $\mathbf{v}^{r,\rho,0,l}$ of the linearized Navier-Stokes equation at the first step of the local time iteration. 
Furthermore, the additional source term $S^l(\mathbf{r})$ depending on the function $\mathbf{r}$ appears only in the equation for $\mathbf{v}^{r,\rho,0,l}$, because these source terms cancel out in the subtraction $\delta \mathbf{v}^{r,\rho,k,l}=\mathbf{v}^{r,\rho,k,l}-\mathbf{v}^{r,\rho,k-1,l}$ defining the higher order corrections of the local iterative scheme. It is sufficient to find a lower bound $\rho_l\succsim \frac{1}{l}$ for the number $\rho_l$ in order to make the scheme global, but we shall see that an appropriate choice of the control functions $r^l_i$ leads to a global scheme with $\rho_l\geq \rho$ for some $\rho>0$ independent of the time step number $l\geq 1$. This will be done in the next step of the proof.
\begin{rem}
For numerical purposes we shall consider a different solution scheme below. If we know by analytical means that $\mathbf{v}^r$ and $\mathbf{r}$ are bounded then we know that $\mathbf{v}$ is bounded, and this implies that we can set up a scheme where we do not use the function $\mathbf{r}$ at all. The control  function $\mathbf{r}$ may be reintroduced in order to stabilize the computation, especially if viscosity is low and/or if there are boundary conditions. A second feature of a numerical scheme is that diffusion may imply that increasing time step sizes $\rho_l$ may be chosen as the time step number $l$ increases. This makes the scheme more efficient. 
\end{rem}
Next we construct
\begin{equation}\label{vrlseries}
 \mathbf{v}^{r,\rho,l}=\mathbf{v}^{r,\rho,0,l}+\sum_{k=1}^{\infty}\delta \mathbf{v}^{r,\rho,k,l}.
\end{equation}
In this section our goal is make the extension of the argument in step 1 to the case where $\mathbf{r}\neq 0$ as easy as possible. Therefore we consider an iterative scheme for (\ref{vrlseries}) which we would not choose for numerical purposes. More precisely, for each $k\geq 0$ in the computation of the series (\ref{vrlseries}) we have the additional term $L^{r,\rho,l}_i$ with arguments of the previous time step. This means that in step $k=0$ we add a term $L^{r,\rho,l}_i\left(\mathbf{r},\mathbf{v}^{\rho,l-1}\right)$ on the right side in order to compute the approximation $v^{r,\rho,0,l}_i$ (note that the argument $\mathbf{r}$ is considered to be externally given in this second step of the proof; we shall determine the function $\mathbf{r}$ in the third step if this proof). Similarly, in step $k>0$ we add a term $L^{r,\rho,l}_i\left(\mathbf{r},\mathbf{v}^{\rho,k-1,l}\right)$ on the right side in order to compute the approximation $v^{r,\rho,k,l}_i$. 
The first term of this functional series $\mathbf{v}^{r,\rho,0,l}=\left(v^{r,\rho,0,l}_1,\cdots ,v^{r,\rho,0,l}_n\right)^T$ is solution of the equation (for $1\leq i\leq n$)
\begin{equation}\label{Navlerayrrhol*}
\left\lbrace \begin{array}{ll}
\frac{\partial v^{r,\rho,0,l}_i}{\partial \tau}-\rho_l\nu\sum_{j=1}^n \frac{\partial^2 v^{r,\rho,l}_i}{\partial x_j^2} 
+\rho_l\sum_{j=1}^n v^{r,\rho,l-1}_j\frac{\partial v^{r,\rho,0,l}_i}{\partial x_j}=L^{r,\rho,l}_i\left(\mathbf{r},\mathbf{v}^{r,\rho,l-1}\right) \\
\\
+\rho_l\int_{{\mathbb R}^n}\left( \frac{\partial}{\partial x_i}K_n(x-y)\right) \sum_{j,k=1}^n\left( \frac{\partial v^{r,\rho ,l-1}_k}{\partial x_j}\frac{\partial v^{r,\rho,l-1}_j}{\partial x_k}\right) (\tau,y)dy+r^l_{i,\tau},\\
\\
\mathbf{v}^{r,\rho,0,l}(l-1,.)=\mathbf{v}^{r,\rho,l-1}(l-1,.),
\end{array}\right.
\end{equation}
and the other terms $\delta \mathbf{v}^{r,\rho,k,l}$ of this functional series are 
the respective solutions  of the equations (for $1\leq i\leq n$)
\begin{equation}\label{Navlerayrrholdelta}
\left\lbrace \begin{array}{ll}
\frac{\partial \delta v^{r,\rho,k,l}_i}{\partial \tau}-\rho_l\nu\sum_{j=1}^n \frac{\partial^2 \delta v^{r,\rho,k,l}_i}{\partial x_j^2} 
+\rho_l\sum_{j=1}^n v^{r,\rho,k-1,l}_j\frac{\partial \delta v^{r,\rho,k,l}_i}{\partial x_j}=L^{\rho,l,0}_i\left(\mathbf{r},\delta\mathbf{v}^{r,\rho,k-1,l}\right) \\
\\
-\rho_l\sum_{j=1}^n \delta v^{r,\rho,k-1,l}_j\frac{\partial  v^{r,\rho,k-1,l}_i}{\partial x_j}+\\
\\
\rho_l\int_{{\mathbb R}^n}\left( \frac{\partial}{\partial x_i}K_n(x-y)\right) \sum_{j,m=1}^n\left( \frac{\partial v^{r,\rho ,k-1,l}_m}{\partial x_j}\frac{\partial v^{r,\rho,k-1,l}_j}{\partial x_m}\right) (\tau,y)dy\\
\\
-\rho_l\int_{{\mathbb R}^n}\left( \frac{\partial}{\partial x_i}K_n(x-y)\right) \sum_{j,k=1}^n\left( \frac{\partial v^{r,\rho ,k-2,l}_k}{\partial x_j}\frac{\partial v^{r,\rho,k-2,l}_j}{\partial x_k}\right) (\tau,y)dy,\\
\\
\delta \mathbf{v}^{r,\rho,k,l}(l-1,.)=0.
\end{array}\right.
\end{equation}
Here in case $k=1$ the expression $v^{r,\rho,k-2,l}_j$ is defined by
\begin{equation}
v^{r,\rho,-1,l}_j:=v^{r,\rho,l-1}_j.
\end{equation}
We consider local estimates in this step.
We assume that from the previous time step $l-1$ we have the estimates
\begin{equation}
|r^{l-1}_i|_{0}\leq C^{0,l-1}_r,~|r^{l-1}_i|_{0,1}\leq C^{1,l-1}_r,~|r^{l-1}_i|_{1,2}\leq C^{l-1}_r,
\end{equation}
and that
\begin{equation}
|v^{l-1}_i|_{0}\leq C^{l-1}_{0},~|v^{l-1}_i|_{0,1}\leq C^{l-1}_{1},~\sum_{i=1}^n|v^{l-1}_i|_{1,2}\leq C^{l-1}_{1,2}.
\end{equation}
Furthermore we assume that
\begin{equation}\label{rnormstep2est*}
{\Big |} v^{r,\rho,l-1}_k(\tau,.){\Big |}_{H^2}\leq C^*_n\left( C^{l-1}_{1,2}\right) , 
\end{equation}
and that
\begin{equation}\label{rnormstep2est*}
{\Big |}  r^{l-1}_k{\Big |}_{H^2}\leq C^*_n\left( C^{l-1}_r\right)
\end{equation}
for all $1\leq k\leq n$ and generic $C^*_n$ and $C^{l-1}_r$.
Note that this implies
\begin{equation}\label{rnormstep2est}
\int_{{\mathbb R}^n}\sum_{j,k,m,p=1}^n{\Big |} \frac{\partial v^{r,\rho,l-1}_k}{\partial x_j}{\Big |}{\Big |}\frac{\partial v^{r,\rho,l-1}_m}{\partial x_p}{\Big |}(\tau,y)dy\leq C^*_n\left( C^{l-1}_{1,2}\right) , 
\end{equation}
and that
\begin{equation}\label{rnormstep2est}
\int_{{\mathbb R}^n}\sum_{j,k,m,p=1}^n{\Big |} \frac{\partial r^{l-1}_k}{\partial x_j}{\Big |}{\Big |}\frac{\partial r^{l-1}_m}{\partial x_p}{\Big |}(\tau,y)dy\leq C^*_n\left( C_r\right).
\end{equation}
If we establish the latter estimates for all $l$ then we have a bound of the integral terms of Leray projection in the equations for $\mathbf{v}^{r,\rho,l}$.   
\begin{rem}
 In the next subsection we shall define $\mathbf{r}^l$ and determine a global constant $C_r\geq C^{1,l}_r,C^{0,l-1}_r$ which is independent of the time step number $l$, and which is an upper bound for $\mathbf{r}$. More precisely, we shall show that there exists a constant $C_r>0$ (independent of $l$) such that for all $1\leq i\leq n$ and for all $l\geq 1$ the functions $r^{l}_i$ satisfy
\begin{equation}\label{rconst}
|r^{l}_i|_{0,1}\leq C_r.
\end{equation}
 Here the norm is defined on the local domain $D^{\tau}_l=(l-1,l]\times {\mathbb R}^n$, i.e.  we have
\begin{equation}
|r^{l}_i|_{0,1}:=\sum_{|\alpha|\leq 1}\sup_{(\tau,x)\in D^{\tau}_l}|r^l_{,\alpha}(\tau,x)|
\end{equation}
for multiindices $\alpha$ related to the spatial variables.
\end{rem}
In step 3 of this proof below we construct $r^l_i$ such that the relation (\ref{rconst}) and the relation (\ref{rnormstep2est}) are satisfied for all $l\geq 1$. As outlined in the introduction, in our global scheme we first compute the function $\mathbf{r}^{l}$ from the data $\mathbf{r}^{l-1}(l-1,.)$ and $\mathbf{v}^{r,l-1}(l-1,.)$. In this step we need an additional assumption on $C^l_r$. We shall assume that
\begin{equation}
|r^{l}_i|_{1,2}\leq C^{l-1}_r+1.
\end{equation}
Next recall that
\begin{equation}
\begin{array}{ll}
L^{\rho ,l}_i(\mathbf{r};\mathbf{v}^{r,\rho,l-1})\equiv 
-\rho_l\nu \Delta r^l_i+\rho_l\sum_{j=1}^n r^l_j\frac{\partial r^l_i}{\partial x_j}\\
\\
-\rho_l\int_{{\mathbb R}^n}\left( \frac{\partial}{\partial x_i}K_n(x-y)\right) \sum_{j,k=1}^n\left( \frac{\partial r^l_k}{\partial x_j}\frac{\partial r^l_j}{\partial x_k}\right) (\tau,y)dy+L^{\rho ,l,0}_i(\mathbf{r};\mathbf{v}^{r,\rho,l-1}),
\end{array}
\end{equation}
where
\begin{equation}\label{L00}
\begin{array}{ll}
L^{\rho ,l,0}_i(\mathbf{r};\mathbf{v})\equiv 
+\rho_l\sum_{j=1}^n r^l_j\frac{\partial v^{r,\rho,l}_i}{\partial x_j}+\rho_l\sum_{j=1}^n v^{r,\rho,l}_j\frac{\partial r^l_i}{\partial x_j}\\
\\ -2\rho_l\int_{{\mathbb R}^n}\left( \frac{\partial}{\partial x_i}K_n(x-y)\right) \sum_{j,k=1}^n\left( \frac{\partial r^l_k}{\partial x_j}\frac{\partial v^{r,\rho,l}_j}{\partial x_k}\right) (\tau,y)dy.
\end{array}
\end{equation}
We may use the rough estimate
\begin{equation}
\begin{array}{ll}
|L^{\rho ,l}_i(\mathbf{r};\mathbf{v}^{r,\rho,l-1})|\leq 
\rho_l\nu (C^{l-1}_r+1)+\rho_ln
(C^{l-1}_r+1)^2\\
\\
+\rho_lC_K(C^{l-1}_r+1)^2
+|L^{\rho ,l,0}_i(\mathbf{r};\mathbf{v}^{r,\rho,l-1})|.
\end{array}
\end{equation}
which we get from (\ref{rconst}). This is not needed essentially for the proof of local convergence of this step. For this purpose we use (\ref{rconst}) again, and we get from (\ref{L00}) the estimate (generic use of $C^*_n$)
\begin{equation}
\begin{array}{ll}
|L^{\rho ,l,0}_i(\mathbf{r};\mathbf{v}^{r,\rho,l-1})|\leq \rho_lC^*_n|\mathbf{r}^l|_{0,1}|\mathbf{v}^{r,\rho,l-1}|_{0,1} 
\end{array}
\end{equation}

For $l\geq 1$ given the functional series $(v^{r,\rho,k,l}_i)_k$  is determined by the representations
\begin{equation}\label{v0r}
\begin{array}{ll}
v^{r,\rho,0,l}_i(\tau,x)=\int_{{\mathbb R}^n}v^{r,\rho,l-1}_i(l-1,y)\Gamma^{l}_0(\tau,x;l-1,y)dy+\\
\\
\int_{(l-1)}^{\tau}\int_{{\mathbb R}^n}\rho_l\int_{{\mathbb R}^n}\left( \frac{\partial}{\partial x_i}K_n(y-z)\right) \left( \sum_{j,k=1}^n\left( \frac{\partial v^{r,\rho,l-1}_k}{\partial x_j}\frac{\partial v^{r,\rho,l-1}_j}{\partial x_k}\right) (\tau,z)\right) \times\\
\\
\times\Gamma^{r,l}_0(\tau,x,s,y)dzdyds\\
\\
+\int_{(l-1)}^{\tau}\int_{{\mathbb R}^n}L^{\rho,l}_i(\mathbf{r},\mathbf{v}^{r,l-1})(s,y)\Gamma^{r,l}_0(\tau,x,s,y)dzdyds\\
\\
+\int_{(l-1)}^{\tau}\int_{{\mathbb R}^n}r^l_{i,t}(s,y)\Gamma^{r,l}_0(\tau,x,s,y)dzdyds
\end{array}
\end{equation}
and (for $k\geq 1$) 
\begin{equation}\label{deltavr}
\begin{array}{ll}
\delta v^{r,\rho,k,l}_i(\tau,x)=-\rho_l\int_{(l-1)}^{\tau}\int_{{\mathbb R}^n}\left( \sum_{j=1}^n\delta v^{r,\rho,k-1,l}_j\frac{\partial v^{r,\rho ,k-1,l}}{\partial x_j}\right) (s,y)\Gamma^{r,l}_k(\tau,x;s,y)ds dy+\\
\\
\int_{(l-1)}^{\tau}\rho_l\int_{{\mathbb R}^n}\int_{{\mathbb R}^n}K_{n,i}(z-y)\times\\
\\
{\Big (} \left( \sum_{j,m=1}^n\left( v^{r,\rho ,k-1,l}_{m,j}(s,y)+v^{r,\rho,k-2,l}_{m,j}(s,y)\right) \right) \times\\
\\
\left( \sum_{j,m=1}^n\left( v^{r,\rho,k-1,l}_{j,m}(s,y)-v^{r,\rho,k-2,l}_{j,m}(s,y)\right) \right) {\Big)}\Gamma^{r,l}_k(\tau,x;s,z)dydzds,\\
\\
+\int_{(l-1)}^{\tau}\int_{{\mathbb R}^n}L^{\rho,l,0}_i\left(\mathbf{r},\delta \mathbf{v}^{r,\rho,k-1,l} \right) (s,y)\Gamma^{r,l}_k(\tau,x,s,y)dzdyds
\end{array}
\end{equation}
where the functions $\Gamma^{r,l}_k,~k\geq 0$ are fundamental solutions of \begin{equation}\label{Navleraylgamma0r}
 \begin{array}{ll}
\frac{\partial \Gamma^{r,l}_k}{\partial \tau}-\rho_l\nu\sum_{j=1}^n \frac{\partial^2 \Gamma^{r,l}_k}{\partial x_j^2} 
+\rho_l\sum_{j=1}^n v^{r,\rho,k-1,l}_j\frac{\partial \Gamma^{r,l}_k}{\partial x_j}=0
\end{array}
\end{equation}
(recall that $v^{r,\rho,-1,l}_i=v^{r,\rho,l-1}$ in case $k=0$).
Analogously as in the first step of this proof for some time step $l$ let $\mathbf{v}^{r,\rho,l-1}$ with $\mathbf{v}^{r,\rho,l-1}$ for some let $F_{r,l}:C^{1,2}_b(D^{\tau}_l)$ such that $\mathbf{v}^{r,\rho,k,l}=F_{r,l}^k(\mathbf{v}^{r,\rho,l-1})$. We assume that
\begin{equation}
|\mathbf{v}^{r,\rho,l-1}|^n_{1,2}\leq C^{l-1}_{1,2} 
\end{equation}
and start the iteration with
\begin{equation}
\mathbf{v}^{r,\rho,l-1}\in S^r_0:=\left\lbrace \mathbf{f}||\mathbf{f}|^n_{1,2}\leq 4C^{l-1}_{1,2} \right\rbrace .
\end{equation}

Hence, similar as in the first step we get for some generic constant $C^*_n$ dependent on the dimension $n$ 
\begin{equation}\label{xxx3r}
\begin{array}{ll}
|\delta \mathbf{v}^{r,\rho,k,l}|^n_{1,2}
\leq \rho_lC^*_nC_\Gamma\left( C^{l-1}_{1,2}+C^{l-1}_r\right) C_{\Gamma}| \delta \mathbf{v}^{\rho,k-1,l}|^n_{1,2}\\
\\
\leq \frac{1}{4}| \delta \mathbf{v}^{r,\rho,k-1,l}|^n_{1,2}
\end{array}
\end{equation}
for all $k\geq 0$ if we choose
\begin{equation}
 \rho_l\leq \frac{1}{4C^*_nC_{\Gamma}\left(C^{l-1}_{1,2}+C^{l-1}_r\right)C_{\Gamma}}
\end{equation}
In order to have similar estimates with respect to the $|.|_{H^2}$-norm we use constants $\rho_l$ with 
\begin{equation}
 \rho_l\leq \frac{1}{4C^*_nC_{\Gamma}\left( C^{l-1}_{1,2}+C^{l-1}_r\right)C_{\Gamma}}
\end{equation}
Moreover with analogous arguments as in the first step for all $\tau\in [l-1,l]$ we have
\begin{equation}\label{estH2}
|\delta v^{r,\rho,k,l}(\tau,.)|_{H^2}\leq \frac{1}{4}|\delta v^{r,\rho,k-1,l}(\tau,.)|_{H^2}.
\end{equation} 
uniformly in $\tau$. Then we use either method a) or b) completely analogously as in the first step of this proof.  
Hence locally and for $\rho_l$ as above we have got a function $\mathbf{v}^{\rho,l}\in C^{1,2}_b$ which is constructed from a limit of the functional series
\begin{equation}
v^{r,\rho,l}_i=v^{r,\rho,0,l}_i+\sum_{k=1}^{\infty}\delta v^{r,\rho,k,l}_i
\end{equation}
where $\mathbf{v}^{r,\rho,l}$ is a classical solution of

\begin{equation}\label{Navlerayrrhollimit}
\left\lbrace \begin{array}{ll}
\frac{\partial v^{r,\rho,l}_i}{\partial \tau}-\rho_l\nu\sum_{j=1}^n \frac{\partial^2 v^{r,\rho,l}_i}{\partial x_j^2} 
+\rho_l\sum_{j=1}^n v^{r,\rho,l}_j\frac{\partial v^{r,\rho,l}_i}{\partial x_j}=\\
\\
L^{\rho,l}_i(\mathbf{r},\mathbf{v}^{r,\rho,l})+\\
\\
\rho_l\int_{{\mathbb R}^n}\left( \frac{\partial}{\partial x_i}K_n(x-y)\right) \sum_{j,k=1}^n\left( \frac{\partial v^{r,\rho ,l}_k}{\partial x_j}\frac{\partial v^{r,\rho,l}_j}{\partial x_k}\right) (\tau,y)dy+r^l_{i,\tau},\\
\\
\mathbf{v}^{r,\rho,l}(l-1,.)=\mathbf{v}^{r,\rho,l-1}(l-1,.).
\end{array}\right.
\end{equation}
Finally, we note that by adjusting $C^*_n$ in our choice of $\rho_l$ we can ensure that
\begin{equation}
\sum_{k=1}^{\infty}|\delta \mathbf{v}^{r,\rho,k,l}|^n_{1,2}\leq \frac{1}{4}. 
\end{equation}

\subsection{step 3: Construction of the function $\mathbf{r}$ and of a globally convergent iterative scheme for $\mathbf{v}^{r,\rho,l}$}

We have proved the existence of local solutions for $\mathbf{v}^{r,\rho,l}=\mathbf{v}^{r,l}+\mathbf{r}^l$ for some fixed class of functions $\mathbf{r}^l$ and $\mathbf{v}^{r,\rho,l}$ with components in $C^{1,2}_b$ and such that the components are in the Hilbert space $H^2$ or in some Banach spaces $H^{s,p}$ for $s>2+\frac{p}{n}$ in order to have convergence with some suitable regularity of the limit function. A 'suitable regularity' of the limit function is H\"{o}lder continuity which is uniform in time. This implies another simple reason for choosing the classical space: it allows us to have classical representations in terms of fundamental solutions not only for some approximating linear equations but also for the limit function once it is known. The classical representation in terms of fundamental solutions then allows us to get more regularity. Note that we used regularity of coefficients when we applied the adjoint equation in the first and second step of this proof.  Furthermore the existence of classical solutions is useful because some forms of the maximum principle require the existence of classical solutions. These observations are useful for our construction of the function $\mathbf{r}$ which controls the global growth of the function $\mathbf{v}^{r,\rho}$ and is itself bounded (and has an integral magnitude which has linear growth with respect to time).  The time step size at time step number $l$ (measured in original time coordinates $t$) is $\rho_l$ and we concluded that the local scheme converges with any time step size of which satisfies
\begin{equation}\label{rholsec}
 \rho_l\leq \frac{1}{C^*_n\left( C^{l-1}_{1,2}+C^{l-1}_r\right) 4C_{\Gamma}^2},
\end{equation}
is sufficient, where $C^{l-1}_{1,2}$ is an upper bound for $\mathbf{v}^{r,\rho,l-1}$, and where $C^{l-1}_r$ a certain upper bound bound for $\mathbf{r}^l$ with respect to the $|.|_{1,2}$ norm. The next step is to define the function $\mathbf{r}$ such that there is is a global constant $C_r\geq C^{1,l}_r$, i.e., an upper bound for the control function with respect to the $|.|_{0,1}$ norm which is independent of the time step number $l$  (we shall see that a global upper boung for the modulus of the component functions $r_i$ and their spatial derivatives is sufficient). Indeed, we determine the constant $C_r$ in terms of the initial data function $h_i$ in this section. Furthermore, at the end of this section we shall show that this choice of $C_r$ is an uniform upper bound if $\rho_l$ is chosen as in (\ref{rholsec}) with $C_r$ instead of $C^{l-1}_r$. Concerning this constant $C^*_{n}$ we note that it depends only on dimension (and the time step size $\rho$) and may be computed explicitly. The constant $C_{\Gamma}$ in (\ref{rholsec}) is a uniform bound of constants related to the fundamental solutions which appear at the $k$th substep of time step $l$ (cf. step 1 of this proof).  Step 1 and step 2 of this proof show: if the constants $C^{l-1}_{1,2}$ and $C_r$ are given then $C_{\Gamma}$ can be computed and local convergence is given along with the choice of $\rho_l$ as in (\ref{rholsec}). 
(Note that we have shown in step 1 and step 2 of this proof that the other constants in (\ref{rholsec}) depend only on the dimension and the initial data $\mathbf{h}$ and, hence, are independent of the time step number $l$.)

We have to show that the numbers $C^{l}_{1,2},~l\geq 0$ can be estimated in terms of a constant $C_{1,2}$ which is independent of $l$ in order to get a global scheme. We choose a time step size of order $\rho_l\geq \rho$ where we have to choose the control function $\mathbf{r}^l$ recusrsively in order to deal with the growth of the integral magnitudes related to the integral term of the Navier-Stokes equation in Leray projection form. Furthermore we have to show that the function $\mathbf{r}$ can be chosen together with a finite constant $C_r$ independently of the time step number $l$. In this step 3 of the proof we determine the constant $C_{1,2}$ and the constant $C_r$ which is a bound for the global function $\mathbf{r}$. In order to reduce the number of constants we shall determine $C_{1,2}$ such that
\begin{equation}\label{Cr}
|r^l_i|_{1,2}\leq C_r=C_{1,2}
\end{equation}
 for all $l\geq 1$ and all $1\leq i\leq n$.  
 Here in (\ref{Cr}) the  norm is defined on the local domain $D^{\tau}_l=(l-1,l]\times {\mathbb R}^n$, i.e.  we have
\begin{equation}
|r^{l}_i|_{1,2}:=\sum_{|\alpha|\leq 2}\sup_{(\tau,x)\in D^{\tau}_l}|r^l_{,\alpha}(\tau,x)|+\sup_{(\tau,x)\in D^{\tau}_l}|r^l_{,\tau}(\tau,x)|
\end{equation}
 Globally, we shall construct a bound of the norm $|.|_{0,1}$. We shall see that this implies a global bound with respect to the $|.|_{0,2}$-norm, and even a global bound of type (\ref{Cr}), except for the time dervative at the time points $\tau=l$ for some $l\geq 1$.  
  
As we mentioned in the introduction, our construction leads to a global control function which is differentiable with respect to the spatial variables but may not be time-differentiable accross the time points $\tau=l$ for integers $l\geq 1$.    
  Since we have a step size with respect to original time coordinates of order $\rho_lgeq \rho$ for some time step size independent of the time step number $l$ we ensure that for the growth of the integral magnitude we have
 \begin{equation}\label{rlnormstep3}
\int_{{\mathbb R}^n}\sum_{j,k=1}^n{\Big |}\left( \frac{\partial r^l_k}{\partial x_j}\frac{\partial r^l_j}{\partial x_k}\right) (\tau,y){\Big |}dy\leq C^*_n\left( C_{1,2}\right) ,
\end{equation}
This is a difference to the linear growth of the integral magnitude with respect to the time step number $l$ which we have in \cite{KB2}. It is the appropriate choice of a control function which ensures that related integral magnitudes of the equation for $\mathbf{v}^{r,\rho,l}$ in Leray projection form have a bound which is independent of the time step number $l$.
In order to estimate the quantity (\ref{rlnormstep3}) we may estimate
 \begin{equation}\label{rlnormstep31}
\int_{{\mathbb R}^n}\sum_{j,k,m,p=1}^n{\Big |}\frac{\partial r^l_k}{\partial x_j}{\Big |}{\Big |}\frac{\partial r^l_m}
{\partial x_p} {\Big |}(\tau,y)dy\leq C^*_n\left( C_{1,2}\right) 
\end{equation}
inductively. However, it is useful to reduce the estimate (\ref{rlnormstep31}) to an $L^2$ estimate of the form ($C^*_n$ generic)
\begin{equation}\label{rlnormstep32}
\int_{{\mathbb R}^n}\sum_{j,k,m,p=1}^n\left( {\Big |}\frac{\partial r^l_k}{\partial x_j}{\Big |}^2+{\Big |}\frac{\partial r^l_m}{\partial x_p}{\Big |}^2\right)  (\tau,y)dy\leq C^*_n\left( C_{1,2}\right) .
\end{equation}
Note that our local schemes of step 1 and step 2 of this proof shows that we need also 
estimates of derivatives of the integral terms, i.e., estimates of the form
\begin{equation}\label{rlnormstep31}
\int_{{\mathbb R}^n}\sum_{j,k,m,p,q=1}^n{\Big |}\frac{\partial^2 r^l_k}{\partial x_j\partial x_q}{\Big |}{\Big |}\frac{\partial r^l_m}
{\partial x_p} {\Big |}(\tau,y)dy\leq C^*_n\left( C_{1,2}\right).
\end{equation}
However, this can be treated similarly, i.e., all this reduces 
 to an estimate in terms of a $H^2$ norm, i.e., to an estimate of form
\begin{equation}
|r^l_k(\tau,.)|_{H^2}\leq C^*_n\left( C_{1,2}\right),
\end{equation}
where we note that $C^*_n$ is a generic constant.

Depending on the representation of the function $r^l_i$ we may iterate the reduction of mixed products to sums of squares. Therefore we shall establish the estimate
\begin{equation}\label{rlnormstep33}
\int_{{\mathbb R}^n} {\Big |}\frac{\partial r^l_k}{\partial x_j}{\Big |}^2  
(\tau,y)dy\leq C_{1,2},
\end{equation}
inductively when we have constructed the function $r^l_i$. This involves related estimates for functions $v^{r,\rho,l-1}_i$ and $\phi^l_i$ which occur in the construction of $r^l_i$. Note that in the case $n=3$ we may also use product estimates $|fg|_{H^s}\leq C_s|f|_{H^s}|g|_{H^s}$ for $s>\frac{n}{2}$.
We shall show that $\mathbf{r}$ can be constructed via certain source term functions $\phi^l_i$ of linear parabolic equations such that 
\begin{equation}
|r_i|_{0,1}\leq C_r
\end{equation}
for all $1\leq i\leq n$.
Furthermore we shall ensure that  for all $l\geq 1$ we have 
 \begin{equation}\label{globalboundvr}
  |\mathbf{v}^{r,\rho,l}|^n_{1,2}\leq C_{1,2}=C_r,
 \end{equation}
where $C_r$ depends only on the dimension $n$ and the initial data $\mathbf{h}$ of the Navier-Stokes Cauchy problem. We shall also show that
\begin{equation}\label{globalboundvrint}
  |\mathbf{v}^{r,\rho,l}|^n_{H^2}\leq C^*_n\left( C_{1,2}\right) ,
 \end{equation}
and 
\begin{equation}\label{globalboundrr}
  |\mathbf{r}^{\rho,l}|^n_{H^2}\leq C^*_n\left( C_{1,2}\right) ,
 \end{equation}
 and that similar estimates are available with respect to other $H^{s,p}$-Banach spaces. This means that our method can be applied to general dimension $n$.
This is outlined in the introduction, and we shall fill in the details now and in the next step of this proof. In order to have an the equivalent problem (i.e., a problem equivalent to the original Navier-Stokes Cauchy problem for $\mathbf{v}$) for $\mathbf{v}^{r,\rho,l}$ we know that $\mathbf{r}$ has to be bounded (well, a bound linear with respect to the time step number $l$ would be sufficient as well, but we show that we can construct a globally bounded function $\mathbf{r}$ such that the problem for $\mathbf{v}^{r,\rho}$ is equivalent to the problem for $\mathbf{v}$). The problem for $\mathbf{v}^{r,\rho}$ has the advantage that we can solve it step by step where we control the growth by the functions $\mathbf{r}^l$. Note that (\ref{globalboundvr}) implies that for the original equation we have the bound 
 \begin{equation}\label{globalboundvorg}
  |\mathbf{v}^{\rho,l}|^n_{1,2}\leq C_{1,2}+nC_r=(n+1)C_{1,2}.
 \end{equation}
The construction of the function $\mathbf{r}$ and of the global scheme for $\mathbf{v}^{r,\rho}$ is done inductively with respect to the time step number $l-1$. At each time step $l\geq 1$ assume that the functions $\mathbf{v}^{r,\rho,l-1}(l-1,.)$ and $\mathbf{r}^{l-1}(l-1,.)$ have been computed. In case $l=1$ we set $\mathbf{v}^{r,\rho,l-1}(l-1,.)=\mathbf{h}$ and $\mathbf{r}^0\equiv 0$. We use these data and proceed at each time step $l$  in three substeps as we outline next.

\begin{itemize}
 \item[3i)] At each time step $l\geq 1$ we construct first the functions $\phi^l_i$ which serve as 'consumption source terms' and are constructed such that the growth of the functions $\mathbf{v}^{\rho,r}$ and $\mathbf{r}^l$ is controlled on the domain $[l-1,l]\times {\mathbb R}^n$. Control of growth is both in the sense of supremum norms such as $|.|_0$ and $|.|_{0,1}$ and in the sense of integral norms such as $|.|_{L^2}$ and $|.|_{H^2}$. Note that we control the growth with respect to the norms $|.|_0$ and $|.|_{0,1}$ in an absolute sense with respect to time and that we control the growth with respect to the integral norms $|.|_{L^2}$, $|.|_{H^1}$,  and $|.|_{H^2}$ in an absolute sense for each fixed and also  independent of the time step number. This is sufficient since the time step size decreases linearly in our scheme. The functions $\phi^l_i,~1\leq i\leq n$ are constructed as a sum of two functions, i.e., 
 \begin{equation}
 \phi^l_i=\phi^{v,l}_i+\phi^{r,l}_i,
 \end{equation}
 where $\phi^{v,l}_i$ is mainly constructed in order to control the growth of $\mathbf{v}^{r,\rho,l}$ and $\phi^{r,l}_i$ is mainly constructed in order to control the growth of $\mathbf{r}^l$. However, since $\mathbf{v}^{r,\rho,l}=\mathbf{v}^{\rho,l}+\mathbf{r}^l$ involves the control function $\mathbf{r}^l$, the function $\phi^{r,l}_i$ can be defined such that it serves to control the growth of the functions $v^{r,\rho,l}_i$ and the function $r^l_i$ at the same time in regions where the function $r^{l-1}_i$ of the previous time step is large. Here, the small time step size becomes important. We control the growth via source functions and solve nonlinear equations locally as perturbations of linear parabolic equations. The small time step size transformes into small diffusion coefficients by the transformation from original time coordinates $t$ to transformed time coordinates $\tau$ and implies that a) the linear diffusion does not alter the information encoded in  the source terms too much, and b) the perturbation correction of the linear approximation is small.  Note that we control the growth of the supremum and of the first derivative and with repect to the intergral norm.

The superscript $s$ refers to the fact that we define these functions independently of time $\tau$.  The advantage is simplicity of the definition of the source functions.  The disadvantage is that we do not have differentiability with respect to time across the time points $\tau=l$ (cf. the discussion in the introduction).
Another aspect is that the control source term $\phi^{r,s,l}_i$ for $r^l_i$ should not interfer with the effect of the control source term $\phi^{v,s,l}_i$. Hence we define them on different scales, i.e., we define the control source term for $r^l_i$ on a scale with factor $\frac{1}{C}$, i.e.,
\begin{equation}
\phi^{r,l}_i(\tau,x)=\phi^{r,s,l}_i(x):=
 -\frac{2}{C}r^{l-1}_i(l-1,.).
\end{equation}
and we define the control source term for $v^{r,\rho,l}_i$ on a scale with factor $\frac{1}{C^2}$, i.e.,
\begin{equation}\label{phivproof}
\phi^{r,l}_i(\tau,x)=\phi^{r,s,l}_i(x):=
 -\frac{2}{C^2}v^{r,l}_i(l-1,.).
\end{equation}
the constant $C>0$ may be large and the size of the time step size may be of order $\rho\sim \frac{1}{C^3}$ such that the growth of all terms with factor $\rho_l$ close to $\rho$ are controlled by these source terms.  The chioce of $C$ depends on the right side of the equations for $r^l_i$ and $v^{r,\rho,0,l}_i$, and the source terms of the equations for the corrections $\delta v^{r,\rho,k,l}$ which we can estimate using the coniderations of step 1 and step 2. 
The sum $\phi^l_i$ will appear on the right side of a linearized equation for $\mathbf{r}^l$ first. Since $\mathbf{r}^l$ is a summand in $\mathbf{v}^{r,\rho,l}$ we need to estimate the difference between the solution of the linearized equation - which we define to be $\mathbf{r}^l$ and the right side of the equation for $\mathbf{v}^{r,\rho,l}$ above which contains nonlinear terms for $\mathbf{r}^l$. Again, the small  size $\rho_l$, i.e. the small time step size with respect to the original $t$-coordinates- ensures that this difference is suffiently small.  Since products of functions can be pointwise estimated by sums of squares it is sufficient to have linear growth of $H^2$-estimates (indeed, linear growth of $H^1$ estimates would be sufficient). The estimation of the solution expression for $r^l_i$ involves (first derivatives of) fundamental solutions of approximating linear equations integrated over time (one time step) and space. These integrals are absolutely bounded. The functions $r^l_i$ are defined in the second substep which we sketch in advance next.
 \item[3ii)] Once we have defined the functions $\phi^l_i$ we define the function $\mathbf{r}^l$ on the domain $[l-1,l]\times {\mathbb R}^n$ via linearized equations of the form
 \begin{equation}\label{Navlerayrrholdefrl}
\left\lbrace \begin{array}{ll}
r^l_{i,\tau}-\rho_l\nu \Delta r^l_i+\rho_l\sum_{j=1}^n r^{l-1}_j(l-1,.)\frac{\partial r^{l}_i}{\partial x_j}=\\
\\
+\rho_l\int_{{\mathbb R}^n}\left( \frac{\partial}{\partial x_i}K_n(x-y)\right) \sum_{j,k=1}^n\left( \frac{\partial r^{l-1}_k}{\partial x_j}\frac{\partial r^{l-1}_j}{\partial x_k}\right) (l-1,y)dy\\
\\
-\rho_l\sum_{j=1}^n r^{l-1}_j\frac{\partial v^{r,\rho,l-1}_i}{\partial x_j}-\rho_l\sum_{j=1}^n v^{r,\rho,l-1}_j\frac{\partial r^{l-1}_i}{\partial x_j}\\
\\ +2\rho_l\int_{{\mathbb R}^n}\left( \frac{\partial}{\partial x_i}K_n(x-y)\right) \sum_{j,k=1}^n\left( \frac{\partial r^{l-1}_k}{\partial x_j}\frac{\partial v^{r,\rho,l-1}_j}{\partial x_k}\right) (l-1,y)dy\\
\\-\rho_l\int_{{\mathbb R}^n}\left( \frac{\partial}{\partial x_i}K_n(x-y)\right) \sum_{j,m=1}^n\left( \frac{\partial v^{r,\rho,l-1}_m}{\partial x_j}\frac{\partial v^{r,\rho,l-1}_j}{\partial x_m}\right) (l-1,y)dy+\phi^l_i,\\
\\
\mathbf{r}^{l}(l-1,.)=\mathbf{r}^{l-1}(l-1,.).
\end{array}\right.
\end{equation}
Note that the functions on the right side with time index $l-1$ are defined on the domain $[l-2,l-1]\times {\mathbb R}^n$. Hence these functions are evaluated at $\tau=l-1$. Note that the function $\phi^l_i$ is defined on the domain $[l-1,l]\times {\mathbb R}^n$. However, we defined $\phi^l_i$ as functions which depend only on the spatial variables, i.e., they are constant with respect to time. Reasons for this are given in the introduction of this paper.  
Moreover, we assume that $\mathbf{r}^{l-1}(l-1,.)$ has been defined at the previous time step. At time step $l=1$ we have defined $\phi^{r,1}_i\equiv 0$ for all $1\leq i\leq n$.
Note that all functions on the right side of the first equation in (\ref{Navlerayrrholdefrl}) have a factor $\rho_l$ except for the function $\phi^l_i$. Choosing $\rho_l$  small enough we ensure that each component of the function $r^l_i,~1\leq i\leq n$ of the function $\mathbf{r}^l$ is bounded in terms of function $Cr ^l_i(l-1,.)=r^{l-1}_i(l-1,.),~1\leq i\leq n$ respectively for each $i$, and with respect to the releant norms, and where $C>0$ is a constant independent of the time step number $l$.
Especially we shall see that we have a bound to the integral magnitude (\ref{rlnormstep3}) where it is essential at time step $l$ that $\phi^l_k$ is defined in substep 1 such that the $L^p$-norm of
 \begin{equation}\label{philnormstep3}
\int_{{\mathbb R}^n} r^{l-1}_k(l-1,y)\Gamma^l_r(\tau,x;s,y)dy+\int_{l-1}^{\tau}\int_{{\mathbb R}^n} \phi^l_k(s,y)\Gamma^l_r(\tau,x;s,y)dy,
\end{equation}
is finite and bounded independently of the time step number $l$ (we may use $H^{s,p}$-bounds with $p>2$ for higher dimensions $n$). Recall that $\Gamma_r$ denotes the fundamental solution of the linearized equation for $r^l_i$. The term 'essential' means that the other source terms are dominated by the term
\begin{equation}\label{phiterm}
\int_{l-1}^{\tau}\int_{{\mathbb R}^n} \phi^l_k(s,y)\Gamma^l_r(\tau,x;s,y)dy.
\end{equation}
We choose $\rho_l$ small such that the first spatial partial derivatives of (\ref{phiterm}) define a dominant term as well. The bound of the $H^1$-norm will follow directly from our construction of $\phi^l_k$ as well. We shall get a bound of the $H^2$-norm as well. Furthermore, in order to get upper bounds with respect to classical norms it is essential to have a global bound with respect to a H\"{o}lder norm. The definition of $\phi^l_i$ in terms of the functions $r^{l-1}_i$ and $v^{r,\rho,l-1}_i$ makes all this possible.

\item[3iii)] Next we plug $\mathbf{v}^{r,\rho,l}=\mathbf{v}^{\rho,l}+\mathbf{r}^l$ into the Navier-Stokes equation on $[l-1,l]\times {\mathbb R}^n$. Then we get
\begin{equation}\label{Navlerayrrholvr1}
\left\lbrace \begin{array}{ll}
\frac{\partial v^{r,\rho,l}_i}{\partial \tau}-\rho_l\nu\sum_{j=1}^n \frac{\partial^2 v^{r,\rho,l}_i}{\partial x_j^2} 
+\rho_l\sum_{j=1}^n v^{r,\rho,l}_j\frac{\partial v^{r,\rho,l}_i}{\partial x_j}=\psi^l_i,\\
\\
\mathbf{v}^{r,\rho,l}(l-1,.)=\mathbf{v}^{r,\rho,l-1}(l-1,.),
\end{array}\right.
\end{equation}
where
\begin{equation}\label{Navlerayrrhol*psi1}
\begin{array}{ll}
\psi^l_i=r^l_{i,\tau}-\rho_l\nu \Delta r^l_i+\rho_l\sum_{j=1}^n r^l_j\frac{\partial r^l_i}{\partial x_j}\\
\\
-\rho_l\int_{{\mathbb R}^n}\left( \frac{\partial}{\partial x_i}K_n(x-y)\right) \sum_{j,k=1}^n\left( \frac{\partial r^l_k}{\partial x_j}\frac{\partial r^l_j}{\partial x_k}\right) (\tau,y)dy\\
\\
+\rho_l\sum_{j=1}^n r^l_j\frac{\partial v^{r,\rho,l}_i}{\partial x_j}+\rho_l\sum_{j=1}^n v^{r,\rho,l}_j\frac{\partial r^{l}_i}{\partial x_j}\\
\\ -2\rho_l\int_{{\mathbb R}^n}\left( \frac{\partial}{\partial x_i}K_n(x-y)\right) \sum_{j,k=1}^n\left( \frac{\partial r^{l}_k}{\partial x_j}\frac{\partial v^{r,\rho,l}_j}{\partial x_k}\right) (\tau,y)dy\\
\\
 +\rho_l\int_{{\mathbb R}^n}\left( \frac{\partial}{\partial x_i}K_n(x-y)\right) \sum_{j,k=1}^n\left( \frac{\partial v^{r,\rho,l}_k}{\partial x_j}\frac{\partial v^{r,\rho,l}_j}{\partial x_k}\right) (\tau,y)dy.
\end{array}
\end{equation}
In order to control the growth of $\mathbf{v}^{r,\rho,l}$ we consider the  functional series $\mathbf{v}^{r,\rho,l}=\mathbf{v}^{r,\rho,0,l}+\sum_{k=1}^{\infty}\delta \mathbf{v}^{r,\rho,k,l}$ and consider a) the growth of $\mathbf{v}^{r,\rho,0,l}$, and b) the fact the sum of correction terms $\sum_{k=1}^{\infty}\delta \mathbf{v}^{r,\rho,k,l}$ is small for small $\rho_l$. The equation for
\begin{equation}\label{Navlerayrrhol*++0step4}
\left\lbrace \begin{array}{ll}
\frac{\partial v^{r,\rho,0,l}_i}{\partial \tau}-\rho_l\nu\sum_{j=1}^n \frac{\partial^2 v^{r,\rho,0,l}_i}{\partial x_j^2} 
+\rho_l\sum_{j=1}^n v^{r,\rho,l-1}_j(l-1,.)\frac{\partial v^{r,\rho,0,l}_i}{\partial x_j}=\psi^{l,0}_i\\
\\
\mathbf{v}^{r,\rho,0,l}(l-1,.)=\mathbf{v}^{r,\rho,l-1}(l-1,.),
\end{array}\right.
\end{equation}
where
\begin{equation}\label{Navlerayrrhol*++0step4}
\begin{array}{ll}
\psi^{l,0}_i=
r^l_{i,\tau}-\rho_l\nu \Delta r^l_i+\rho_l\sum_{j=1}^n r^l_j\frac{\partial r^l_i}{\partial x_j}\\
\\
-\rho_l\int_{{\mathbb R}^n}\left( \frac{\partial}{\partial x_i}K_n(x-y)\right) \sum_{j,k=1}^n\left( \frac{\partial r^l_k}{\partial x_j}\frac{\partial r^l_j}{\partial x_k}\right) (\tau,y)dy\\
\\
+\rho_l\sum_{j=1}^n r^l_j\frac{\partial v^{r,\rho,l-1}_i}{\partial x_j}+\rho_l\sum_{j=1}^n v^{r,\rho,l-1}_j\frac{\partial r^{l}_i}{\partial x_j}\\
\\ -2\rho_l\int_{{\mathbb R}^n}\left( \frac{\partial}{\partial x_i}K_n(x-y)\right) \sum_{j,k=1}^n\left( \frac{\partial r^{l}_k}{\partial x_j}\frac{\partial v^{r,\rho,l-1}_j}{\partial x_k}\right) (\tau,y)dy\\
\\
+\rho_l\int_{{\mathbb R}^n}\left( \frac{\partial}{\partial x_i}K_n(x-y)\right) \sum_{j,k=1}^n\left( \frac{\partial v^{r,\rho,l-1}_k}{\partial x_j}\frac{\partial v^{r,\rho,l-1}_j}{\partial x_k}\right) (l-1,y)dy.
\end{array}
\end{equation}
Then we analyze the growth of the functions $v^{r,\rho,0,l}_i,~1\leq i\leq n$, where we estimate the  $\psi^{l,0}_i$  using (\ref{Navlerayrrholdefrl}), and then estimate the difference $\psi^l_i-\psi^{l,0}_i$. Our construction of functions $\phi^l_i$ in terms of the functions $v^{r,\rho,l-1}_i(l-1,.)$ and $r^{l-1}_i(l-1,.)$ ensures that for large $C>0$ and small $\rho>0$ the source term has a consumptive effect on the domain 
$[l-1,l]\times {\mathbb R}^n$- especially 
if one of the functions $v^{r,\rho,l-1}_i(l-1,.)$ and $r^{l-1}_i(l-1,.)$  exceeds a certain 
critical level. This ensdsures global boundedness with respect tp the relevant norms $|.|_{1,2}$ and $|.|_{H^2}$ (the latter for functions evaluated at some time $\tau$). 
\end{itemize}

ad 3i)  As we described in the introduction of this section the function $\phi^l_i$ is determined as a sum of two functions $\phi^{r,s,l}_i$ and $\phi^{v,s,l}_i$. A small step size $\rho_l\geq \rho>0$ means that the representations of the functions $r^l_i$ and $v^{r,\rho,0,l}_i$ in terms of fundamental solutions include integrated source terms which are close to
\begin{equation}
\begin{array}{ll}
\int_{l-1}^{\tau}\phi^{r,s,l}_i(\sigma,x)d\sigma=\int_{l-1}^{\tau}\phi^{r,\sigma,l}_i(x)d\sigma\\
\\
:=-\int_{l-1}^{\tau}\frac{2}{C}r^{l-1}_i(l-1,.)d\sigma=-(\tau-l)r^{l-1}_i(l-1,.).
 \end{array}
\end{equation}
and 
\begin{equation}\label{phiv2}
\int_{l-1}^{\tau}\phi^{r,l}_i(\sigma,x)d\sigma=\int_{l-1}^{\tau}\phi^{r,\sigma,l}_i(x)d\sigma:=
 -(\tau-l)\frac{2}{C^2}v^{r,l}_i(l-1,.).
\end{equation}
The choice of $\phi^{r,s,l}_i$ as a constant times the negative of $r^{l-1}_i(l-1,.)$ implies that $\phi^{r,s,l}_i$ has a 'consumption effect' with respect to the relevant norms, i.e., the tendency to decrease the value of the control function with respect to the norms $|.|_{1,2}$ and  the $H^2$-norm (for fixed $\tau\in [l-1,l]$. Similarly the choice of $\phi^{v,s,l}_i$ as a constant times the negative of $v^{r,\rho,l-1}_i(l-1,.)$ implies that $\phi^{v,s,l}_i$ has a 'consumption effect' with respect to the relevant norms, too. Furthermore since $\phi^{r,s,l}_i$ is larger than $\phi^{v,s,l}_i$ (by a large factor $C$)  such that the consumption effect of $\phi^{r,s,l}_i$ is not disturbed by the consumption effect by $\phi^{v,s,l}_i$ (we shall see this in more detail in ad ii) and ad iii) below. On the other hand it seems that the consumption effect of $\phi^{r,s,l}_i$ can disturb the consumption effect of $\phi^{v,s,l}_i$. However, whenever $v^{r,\rho,l-1}_i(l-1,.)=v^{\rho,l-1}_i(l-1,.)+r^{l-1}_i(l-1,.)$ is large and $r^l_i$ is large, then the consumption effect of $\phi^{r,s,l}_i$ is active in order to control the growth of $v^{r,\rho,l}_i$ with respect to the relevant norms. On the other hand, if $r^{l-1}_i(l-1,.)$ becomes small compared to $\frac{2}{C^2}v^{r,l}_i(l-1,.)$ then the consumptive effect of $\phi^{v,s,l}_i$ becomes active. Hence we control the growth of $v^{r,\rho,l}_i$ and $r^l_i$ in any case if a) the time step size $\rho_l$ is small enough and b) $C>0$ is large enough.
A choice $\rho_l$ of order $\frac{1}{C^3}$ for $C$ large enough implies that the moduli of source terms $\phi^{r,s,l}_i$ are larger than the sum of all moduli of the other source terms in the equation for $r^{l}_i$ which all have a factor $\rho_l$. Similarly for $\phi^{v,s,l}_i$ and the equation for $v^{r,\rho,l}_i$. We determine $C$ using the results of the preceding steps in the next addenda ad ii) and ad iii) blow. 

ad 3 ii): Next we construct the functions $r^l_i$ for $1\leq i\leq n$. We want to define constants
$C^0_r, C^1_r, C_r$ which are independent of the time step number $l$ and such that for all $l\geq 1$
\begin{equation}
|r^l_i|_0\leq C^0_r,~|r^l_i|_{0,1}\leq C^1_r,~|r^l_i|_{1,2}\leq C_r
\end{equation}
holds, where the norms are defined as above with respect to the domain $D^{\tau}_l=(l-1,l]\times {\mathbb R}^n$. Furthermore we want to ensure linear growth of the intergal magnitude for the function $\mathbf{r}^l$ (see below). At the first time step $l=1$ we have defined $\phi^{r,1}_i\equiv 0$. We may also define $\phi^{1}_i=0$ for all $1\leq i\leq n$, and consider a local scheme as in the first step of this proof. In this case $\mathbf{r}^1=0$ (it is a matter of taste wether we define a nontrivial consumption function $\phi^l_i$ at the first time step $l=1$). For $l\geq 2$
the functions $r^l_i$ are defined by the Cauchy problem (\ref{Navlerayrrholdefrl}). Note that the right side is evaluated at $\tau=l-1$.  We rewrite (\ref{Navlerayrrholdefrl}) in the form
\begin{equation}\label{Navlerayrrholdefrlrewr}
\left\lbrace \begin{array}{ll}
r^l_{i,\tau}-\rho_l\nu \Delta r^l_i+\rho_l\sum_{j=1}^n r^{l-1}_j(l-1,.)\frac{\partial r^{l}_i}{\partial x_j}=\\
\\
S^l_i(l-1,.)+\phi^l_i(l-1,.),\\
\\
\mathbf{r}^{l}(l-1,.)=\mathbf{r}^{l-1}(l-1,.),
\end{array}\right.
\end{equation}
where
\begin{equation}\label{slrep}
 \begin{array}{ll}
S^l_{i}(l-1,x)=
+\rho_l\int_{{\mathbb R}^n}\left( \frac{\partial}{\partial x_i}K_n(x-y)\right) \sum_{j,k=1}^n\left( \frac{\partial r^{l-1}_k}{\partial x_j}\frac{\partial r^{l-1}_j}{\partial x_k}\right) (l-1,y)dy\\
\\
-\rho_l\sum_{j=1}^n r^{l-1}_j\frac{\partial v^{r,\rho,l-1}_i}{\partial x_j}-\rho_l\sum_{j=1}^n v^{r,\rho,l-1}_j\frac{\partial r^{l-1}_i}{\partial x_j}\\
\\ +2\rho_l\int_{{\mathbb R}^n}\left( \frac{\partial}{\partial x_i}K_n(x-y)\right) \sum_{j,k=1}^n\left( \frac{\partial r^{l-1}_k}{\partial x_j}\frac{\partial v^{r,\rho,l-1}_j}{\partial x_k}\right) (l-1,y)dy\\
\\-\rho_l\int_{{\mathbb R}^n}\left( \frac{\partial}{\partial x_i}K_n(x-y)\right) \sum_{j,m=1}^n\left( \frac{\partial v^{r,\rho,l-1}_m}{\partial x_j}\frac{\partial v^{r,\rho,l-1}_j}{\partial x_m}\right) (l-1,y)dy.
\end{array}
\end{equation}
This is a linear parabolic equation for $\mathbf{r}^l$ which can be solved. Furthermore, the solution has a representation in terms of the fundamental solution $\Gamma^l_r$ of the equation 
\begin{equation}\label{fundrrr}
 \begin{array}{ll}
r^l_{i,\tau}-\rho_l\nu \Delta r^l_i+\rho_l\sum_{j=1}^n r^{l-1}_j(l-1,.)\frac{\partial r^{l}_i}{\partial x_j}=0,
\end{array}
\end{equation}
which allows us to analyze the growth of $\mathbf{r}^l$. We have
\begin{equation}\label{rlrep}
 \begin{array}{ll}
r^l_{i}(\tau,x)=\int_{{\mathbb R}^n}r^{l-1}_i(l-1,y)\Gamma^l_r(\tau,x;0,y)dy\\
\\
+\int_{l-1}^{\tau}\int_{{\mathbb R}^n}S^l_i(l-1,y)\Gamma^l_r(\tau,x;s,y)dsdy\\
\\
+\int_{l-1}^{\tau}\int_{{\mathbb R}^n}\phi^l_i(s,y)\Gamma^l_r(\tau,x;s,y)dsdy,
\end{array}
\end{equation}
We shall exploit the following ideas related to the representation (\ref{rlrep}): concerning the 
the relation of the second and the third term on the right side of (\ref{rlrep}) we note that all terms of the function $S^l$ have the factor $\rho_l$ while the third term $\phi^l_i$ does not have this factor. If the factor $\rho_l>0$ is small, then the diffusive effect of the fundamental solution $\Gamma^l_r$ is small in the sense that the values of the second and third term on the right side of (\ref{rlrep}) depend largely on the source terms $S^l_i(l-1,y)$ and $\phi^l_i(s,y)$ (we shall make this remark precise in a moment). This means that for $\rho_l$ small enough for each $1\leq i\leq n$ the term in (\ref{rlrep}) involving $\phi^l_i$ dominates the term involving $S^l_i$. Furthermore a bound for the first term of the right side in (\ref{rlrep}) can be obtained using the maximum principle. Considering derivatives and shifting first derivatives using the adjoint of the fundamental solution as we have done in the first step of this proof leads to an estimate of the functions $r^l_i$ in the $|.|_{0,1}$ norm, the $|.|_{H^2}$ nor and even the $|.|_{1,2}$ norm. Note that the functions $\phi^{r,l,k}_i$ and $\phi^{v,l,k}_i$ are designed in order to control the spatial derivatives in the critical regions where  $r^l_i(l-1,.)$ and/or $v^{r,\rho,l}_i(l-1,.)$ become large (the time step size may be chosen small enough such that the first spatial derivatives of the functions $S^{l-1}_i$ are dominated by the functions $\phi^l_i$). More precisely, for a small stepsize $\rho_l\sim \frac{1}{C^3}$ along with $C>0$ large enough we have   
\begin{equation}\label{rlrep000}
 \begin{array}{ll}
\sup_{(\tau,x)\in [l-1,l]\times {\mathbb R}^n}|r^l_{i}(\tau,x)|\leq {\Big|}\int_{{\mathbb R}^n}r^{l-1}_i(l-1,y)\Gamma^l_r(\tau,x;0,y)dy\\
\\
+\int_{l-1}^{\tau}\int_{{\mathbb R}^n}S^l_i(l-1,y)\Gamma^l_r(\tau,x;s,y)dsdy\\
\\
+\int_{l-1}^{\tau}\int_{{\mathbb R}^n}\phi^l_i(s,y)\Gamma^l_r(\tau,x;s,y)dsdy{\Big |}\\
\\
\leq \sup_{(\tau,x)\in [l-1,l]\times {\mathbb R}^n}{\Big|}\int_{{\mathbb R}^n}r^{l-1}_i(l-1,y)\Gamma^l_r(\tau,x;0,y)dy{\Big |}
\end{array}
\end{equation}
The right side of (\ref{rlrep000}) is a representation of a linear parabolic Cauchy problem with initial data $r^{l-1}_i$. Therefore, the maximum principle may be applied, and this leads to
\begin{equation}
\sup_{(\tau,x)\in [l-1,l]\times {\mathbb R}^n}|r^l_{i}(\tau,x)|\leq \sup_{x\in {\mathbb R}^n}{\Big|}r^{l-1}_i(l-1,x){\Big |}.
\end{equation}
We shall determine $C>0$ below with respect to a stronger norm.
Similarly, concerning the bound of first order derivatives for a small stepsize $\rho_l$  we have (part of the argument similar as in the first step of this proof) 
\begin{equation}\label{rlrep0000}
 \begin{array}{ll}
\sup_{(\tau,x)\in [l-1,l]\times {\mathbb R}^n}{\Big |}\frac{\partial}{\partial x_k}r^l_{i}(\tau,x){\Big |}\leq {\Big|}\int_{{\mathbb R}^n}\left( \frac{\partial}{\partial x_k}r^{l-1}_i(l-1,y)\right) \Gamma^{l,*}_r(\tau,x;0,y)dy\\
\\
+\int_{l-1}^{\tau}\int_{{\mathbb R}^n}\left( \frac{\partial}{\partial x_k}S^l_i(l-1,y)\right) \Gamma^{l,*}_r(\tau,x;s,y)dsdy\\
\\
+\int_{l-1}^{\tau}\int_{{\mathbb R}^n}\left( \frac{\partial}{\partial x_k}\phi^l_i(s,y)\right) \Gamma^{l,*}_r(\tau,x;s,y)dsdy{\Big |}\\
\\
\leq \sup_{(\tau,x)\in [l-1,l]\times {\mathbb R}^n}{\Big|}\int_{{\mathbb R}^n}\left( \frac{\partial}{\partial x_k}r^{l-1}_i(l-1,y)\right) \Gamma^{l,*}_r(\tau,x;0,y)dy{\Big |}.
\end{array}
\end{equation}
Here we may use the adjoint (indcated by the superscript $*$, but note that the same conclusion can be obtained by looking at the Levy expansion of $\Gamma^{l}_r$ which is a sum whith a Gaussian as its first term. We can shift the spatial derivative of this first term (which contributes to a convolution) and treat the other terms as a pertubation of this convolution which is small since all first order terms of the equation for $r^l_i$ have the factor $\rho_l$ and powers of order $k$ of $\rho^l$ appear in the higher order terms of order $k$ of the Levy exansion. This  leads immediately to an estimate for the perturbation of the Gaussian.
From (\ref{rlrep0000}) we get for small $\rho_l$
\begin{equation}\label{rlrep00000}
 \begin{array}{ll}
\sup_{(\tau,x)\in [l-1,l]\times {\mathbb R}^n}{\Big |}\frac{\partial}{\partial x_k}r^l_{i}(\tau,x){\Big |}\leq\sup_{x\in {\mathbb R}^n}{\Big|}\left( \frac{\partial}{\partial x_k}r^{l-1}_i(l-1,x)\right) {\Big |}.
\end{array}
\end{equation}
Well it is essential to keep the control function bounded. This is done within our construction. In order to see this we may use the following lemmas. Let
\begin{equation}
D^{l-1}_{+,i}:=\left\lbrace x\in {\mathbb R}^n||r^{l-1}_i(l-1,x)|\geq C\right\rbrace, 
\end{equation}
and let
\begin{equation}
D^{l-1}_{+,i}:=\left\lbrace x\in {\mathbb R}^n||r^{l-1}_i(l-1,x)|\leq -C\right\rbrace. 
\end{equation}
\begin{lem} Assume that $\rho_l$ is sufficiently small. Then
 for $1\leq i\leq n$, and for $(\tau,x)\in D^{l-1}_{+,i}$ we have
 \begin{equation}
 \int_{l-1}^{\tau}\phi^l_{i}(s,y)\Gamma^l_r(\tau,x;s,y)dsdy\leq -\frac{3}{4C}\left( \tau-(l-1)\right) , 
 \end{equation}
and for $(\tau,x)\in D^{l-1}_{-,i}$ we have
\begin{equation}
 \int_{l-1}^{\tau}\phi^l_{i}(s,y)\Gamma^l_r(\tau,x;s,y)dsdy\geq \frac{3}{4C}\left( \tau -(l-1)\right).
 \end{equation}
\end{lem}
\begin{proof}
You check first that this follows for the heat kernel instead of $\Gamma^l_r$. Then you use the Levy expansion representation for $\Gamma^l_r$ and show that for $\rho_l$ it is a perturbation of of the heat kernel. 
\end{proof}
Note also that
\begin{lem} Assume that $\rho_l$ is sufficiently small. Then
 for $1\leq i\leq n$, and for $(\tau,x)\in D^{l-1}_{+,i}\cup D^{l-1}_{-,i}$ we have
 \begin{equation}
 {\big |}\int_{l-1}^{\tau}S^l_{i}(s,y)\Gamma^l_r(\tau,x;s,y)dsdy{\big |}\leq \frac{1}{2}\left( \tau-(l-1)\right).
 \end{equation}
\end{lem}

Next we look at second order spatial derivatives. The main idea is that a global bound for spatial second derivatives follows from a global bound for first spatial derivatives which is independent of the time-step number $l>0$. This implies that the H\"{o}lder norm of the first order coefficients of the linear parabolic subproblems have a global upper bound which is independent of the time step number $l>0$.  Recall that the functions $r^{l-1}_i(l-1,.)=r^{l}(l-1,.)$ are the first order coefficients of the linear parabolic equations which determine the fundamental solutions $\Gamma^l_r$. Hence, in order to have classical solutions $r^l_i$ it is sufficient to control the $|.|_{0,1}$ of the functions $r^{l-1}_i(l-1,.)$ for all $l\geq 1$.  Nevertheless, we can control the growth of the derivatives up to second order since second order derivative estimates can be reduced to first order derivative estimates and the supremum norm of the derivatives of $r^l_i(l-1,.)$ can be estimated in terms of the supremum norm of $r^l_i$ times a constant $C>0$. For the first spatial derivatives we use properties of the fundamental solution. For the second spatial derivatives we use the adjoint in addition. Recall that we use the constant $C^*_n$ -which depends on $n$ essentially- generically. We show
\begin{lem}
There is a constant $C^*_n>0$ independent of the time step number $l$  such that the estimate
\begin{equation}
\sup_{x\in {\mathbb R}^n}|r^l_{i,j}(l,x)|\leq C\sup_{x\in {\mathbb R}^n}|r^l_{i}(l-1,x)|,
\end{equation}
implies that
\begin{equation}\label{2der}
\sup_{x\in {\mathbb R}^n}|r^l_{i,j,k}(l,x)|\leq C^*_n\sup_{x\in {\mathbb R}^n}|r^l_{i}(l-1,x)|
\end{equation}
holds for some generic $C^*_n$. Here, recall the notation
\begin{equation}
r^l_{i,j}(\tau,x):=\frac{\partial}{\partial x_j}r^l_{i}(\tau,x),~r^l_{i,j,k}(\tau,x):=\frac{\partial^2}{\partial x_j\partial x_k}r^l_{i}(\tau,x)
\end{equation}
 
\end{lem}

\begin{proof}
We provide an independent argument for the weaker result that 
\begin{equation}
\sup_{x\in {\mathbb R}^n}|r^l_{i,j}(l,x)|\leq C^*_n\sup_{x\in {\mathbb R}^n}|r^l_{i}(l-1,x)|,
\end{equation}
and then we show that (\ref{2der}) follows for generic $C^*_n>0$. 
For the first derivatives we have the expression
\begin{equation}\label{rlrepder}
 \begin{array}{ll}
\frac{\partial}{\partial x_k}r^l_{i}(\tau,x)=\int_{{\mathbb R}^n}r^{l-1}_i(l-1,y)\Gamma^l_{r,k}(\tau,x;0,y)dy\\
\\
+\int_{l-1}^{\tau}\int_{{\mathbb R}^n}\left( S^l_i(l-1,y)+\phi^l_i(s,y)\right) \Gamma^l_{r,k}(\tau,x;s,y)dsdy
\end{array}
\end{equation}
Next we consider the fundamental solution $\Gamma^l_r$. It has the Levy expansion:
\begin{equation}
\Gamma^l_r(\tau,x;s,y):=N^l(\tau,x;s,y)+\int_s^{\tau}\int_{{\mathbb R}^n}N^l(\tau,x;\sigma,\xi)\phi_r(\sigma,\xi;s,y)d\sigma d\xi,
\end{equation}
where
\begin{equation}
N^l(\tau,x;s,y)=\frac{1}{\sqrt{4\pi \rho_l\nu (\tau-s)}^n}\exp\left(-\frac{|x-y|^2}{4\rho_l\nu (\tau-s)} \right),
\end{equation}
and $\phi_r$ is a recursively defined function which is H\"{o}lder continuous in $x$, i.e.,
\begin{equation}
\phi_r(\tau,x;s,y)=\sum_{m=1}^{\infty}(L^r_lN^l)_m(\tau,x;s,y),
\end{equation}
along with the recursion
\begin{equation}
\begin{array}{ll}
(L^r_lN^l)_1(\tau,x;s,y)=L^r_lN^l(\tau,x;s,y)\\
\\
=\frac{\partial N^l}{\partial \tau}-\rho_l\nu \Delta N^l+\rho_l\sum_{j=1}^n r^{l-1}_j\frac{\partial N^l}{\partial x_j}\\
\\
=\rho_l\sum_{j=1}^n r^{l-1}_j\frac{\partial N^l}{\partial x_j},\\
\\
(L^rN^l)_{m+1}(\tau ,x):=\int_s^t\int_{\Omega}\left( L^rN^l(\tau,x;\sigma,\xi)\right)_m L^rN^l(\sigma,\xi;s,y)d\sigma d\xi.
\end{array}
\end{equation}
The lemma may be proved for $N_l$ instead of $\Gamma^l_r$ in the representation (\ref{rlrepder}) first and then the proof may be extended to the correction terms involving $\int_s^{\tau}\int_{{\mathbb R}^n}N^l(\tau,x;\sigma,\xi)\phi_r(\sigma,\xi;s,y)d\sigma d\xi$ using the classical Levy expansion estimates.
The proof may be obtained in sub-timesteps, i.e. the first the norm 
\begin{equation}
\sup_{x\in {\mathbb R}^n}|r^l_{i,j}(l-\frac{1}{2},x)|\leq C^*_n\sup_{x\in {\mathbb R}^n}|r^l_{i}(l,x)|\leq C^*_n\sup_{x\in {\mathbb R}^n}|r^l_{i}(l,x)|.
\end{equation}

Then we may use
\begin{equation}\label{rlrepder2}
 \begin{array}{ll}
 N^l_{,k}(\tau,x;s,y)= \frac{(x-y)_k}{4\rho_l\nu (\tau-s)}\frac{1}{\sqrt{4\pi \rho_l\nu (\tau-s)}^n}\exp\left(-\frac{|x-y|^2}{4\rho_l\nu (\tau-s)} \right)\\
 \\
 = \frac{(x-y)_k}{4\rho_l\nu (\tau-s)}\exp\left(-\frac{|x-y|^2}{8\rho_l\nu (\tau-s)} \right)\frac{\sqrt{2}^n}{\sqrt{8\pi \rho_l\nu (\tau-s)}^n}\exp\left(-\frac{|x-y|^2}{8\rho_l\nu (\tau-s)} \right),
\end{array}
\end{equation}
and
\begin{equation}\label{rlrepder2}
 \begin{array}{ll}
\frac{(x-y)^2_k}{4\rho_l\nu (\tau-s)}\exp\left(-\frac{|x-y|^2}{16\rho_l\nu (\tau-s)}\right) \frac{1}{(x-y)_k}
\exp\left(-\frac{|x-y|^2}{16\rho_l\nu (\tau-s)}\right) \leq C
\end{array}
\end{equation}
for some $C>0$.
This leads to the first statement of the lemma. 
Finally we use the argument above that
\begin{equation}
 \sup_{x\in {\mathbb R}^n}|r^l_{i}(l,x)|\leq \sup_{x\in {\mathbb R}^n}|r^l_{i}(l-1,x)|.
\end{equation}
For the second derivatives we use partial integration and the adjoint
\begin{equation}\label{rlrepsecder}
 \begin{array}{ll}
\frac{\partial^2}{\partial x_k\partial x_m}r^l_{i,k}(\tau,x)=\int_{{\mathbb R}^n}r^{l-1}_{i,k}(l-1,y)\Gamma^{*,l}_{r,m}(\tau,x;0,y)dy\\
\\
+\int_{l-1}^{\tau}\int_{{\mathbb R}^n}S^l_{i,k}(l-1,y)\Gamma^{*,l}_{r,m}(\tau,x;s,y)dsdy\\
\\
+\int_{l-1}^{\tau}\int_{{\mathbb R}^n}\phi^l_{i,k}(s,y)\Gamma^{*,l}_{r,m}(\tau,x;s,y)dsdy.
\end{array}
\end{equation}
It follows that the estimate can be reduced to the estimate for first derivatives. Since $C^*_n$ is generic we may replace $(C^*_n)^2$ by $C^*_n$.

\end{proof}

Furthermore, concerning the growth
let us first consider the $L^2$-norm. Since we have global boundedness of values and classical derivatives of $r^{l-1}_i(l-1,.)$ and $v^{r,\rho,l-1}_i(l-1,.)$ and global boundedness of the same functions in $|.|_{H^2}$-norm, products can be estimated by considering a supremum norm of one factor (value function or derivative function) and an $H^m$ norm (some $0\leq m\leq 2$) of the other factor, Furthermore, we may use estimates of convolutions by the generalized Young inequality, the fact that the restriction of second partial derivatives of the kernel $K$ to a domain outside a ball are $L^2$, and local regularity of Poisson equation inside a ball, and the $H^2$ bounds for $r^{l-1}_k(l-1,.)$ in order to get (cf. also \cite{KB2}) 
\begin{equation}\label{slrep2}
 \begin{array}{ll}
|S^l_{i}(l-1,.)|_{L^2}\leq 
\rho_l{\Big |}\int_{{\mathbb R}^n}\left( \frac{\partial}{\partial x_i}K_n(x-y)\right) \sum_{j,k=1}^n\left( \frac{\partial r^{l-1}_k}{\partial x_j}\frac{\partial r^{l-1}_j}{\partial x_k}\right) (l-1,y)dy{\Big |}_{L^2}\\
\\
+\rho_l{\Big |}\sum_{j=1}^n r^{l-1}_j\frac{\partial v^{r,\rho,l-1}_i}{\partial x_j}{\Big |}+\rho_l{\Big |}\sum_{j=1}^n v^{r,\rho,l-1}_j\frac{\partial r^{l-1}_i}{\partial x_j}{\Big |}_{L^2}\\
\\ +2\rho_l{\Big |}\int_{{\mathbb R}^n}\left( \frac{\partial}{\partial x_i}K_n(x-y)\right) \sum_{j,k=1}^n\left( \frac{\partial r^{l-1}_k}{\partial x_j}\frac{\partial v^{r,\rho,l-1}_j}{\partial x_k}\right) (l-1,y)dy{\Big |}_{L^2}\\
\\+\rho_l{\Big|}\int_{{\mathbb R}^n}\left( \frac{\partial}{\partial x_i}K_n(x-y)\right) \sum_{j,m=1}^n\left( \frac{\partial v^{r,\rho,l-1}_m}{\partial x_j}\frac{\partial v^{r,\rho,l-1}_j}{\partial x_m}\right) (l-1,y)dy{\Big |}_{L^2}\\
\\
\leq \rho_l2n^2C_K\max_{i\in \left\lbrace 1,\cdots,n \right\rbrace} |r^{l-1}_i(l-1,.)|_{H^2}
+\rho_l2nC^{l-1}_r C^{l-1}_v\\
\\
\rho_l4n^2C_K\left( \max_{i\in \left\lbrace 1,\cdots,n \right\rbrace}\left(  |r^{l-1}_i(l-1,.)|_{H^2}+|v^{l-1}_i(l-1,.)|_{H^2}\right) \right) \\
\\
\leq \rho_l2\left( n^2C_KC^{l-1}_r+2nC^{l-1}_r C^{l-1}_v+4n^2C_K(C^r_{l-1}+C^v_{l-1})\right).
\end{array}
\end{equation}
Hence (being very generous) for 
\begin{equation}\label{Cconst}
C\geq 4\left(1+ n^2C_KC^{l-1}_r+2nC^{l-1}_r C^{l-1}_v+4n^2C_K(C^r_{l-1}+C^v_{l-1})\right)
\end{equation}
 and $\rho_l\leq \frac{1}{C^3}$ we ensure that in classical approximations time intergals involving the source term $S^{l-1}_i(l-1,.)$ are dominated by time integrals involving the source terms $\phi^{l}_i$ or their summands.
Next we have 
\begin{equation}\label{rlrepL2}
 \begin{array}{ll}
\sup_{\tau\in[l-1,l]}|r^l_{i}(\tau,.)|_{L^2}\leq \sup_{\tau\in[l-1,l]}{\Big |}\int_{{\mathbb R}^n}r^{l-1}_i(l-1,y)\Gamma^l_r(\tau,x;0,y)dy\\
\\
+\int_{l-1}^{\tau}\int_{{\mathbb R}^n}S^l_i(l-1,y)\Gamma^l_r(\tau,x;s,y)dsdy\\
\\
+\int_{l-1}^{\tau}\int_{{\mathbb R}^n}\phi^l_i(s,y)\Gamma^l_r(\tau,x;s,y)dsdy{\Big |}_{L^2}\\
\\
\leq |r^{l-1}_i(l-1,.)|_{L^2},
\end{array}
\end{equation}
for small $\rho_l$ and  by construction of $\phi^l_i$.
Similarly, using the adjoint
\begin{equation}\label{phinormH1}
\begin{array}{ll}
|r^l_{i}(\tau,.)|_{H^1}\leq |r^{l-1}_i(l-1,.)|_{H^1}.
\end{array}
\end{equation}
Furthermore 
\begin{equation}\label{phinormH2}
\begin{array}{ll}
|r^l_{i}(\tau,.)|_{H^2}
\leq |r^{l-1}_i(l-1,.)|_{H^2}.
\end{array}
\end{equation} 
Here we use the definition of $\phi^l_i$ and the fact $r^{l-1}_i(l-1,.)$ and $\phi^{l}_i$ are in $H^2$ and the adjoint and partial integration in order to shift derivatives to these functions. Next we tdetermine $C$. The generic constant $C>0$ in (\ref{Cconst}) is determidned in terms of bounds for $C^{l-1}_r$ and $C^{l-1}_v$ for $r^{l-1}_i(l-1,.)$ and $v^{r,\rho,l-1}_i(l-1,.)$ with repsct to the relevant norm. We define the latter independently of $l$ to be $C_r$ and $C_v$, where
We define
\begin{equation}\label{definitioncr}
\begin{array}{ll}
C_v=C_r=C_{1,2}=2+2|\mathbf{h}|^n_{0,2}+\int_{{\mathbb R}^n}\sum_{j,k,l,m=1}^n{\Big |}\frac{\partial h_k}{\partial x_j}(y){\Big |}{\Big |}\frac{\partial h_l}{\partial x_m}(y){\Big |}dy\\
\\
+2n^2|\mathbf{h}|^n_2
\end{array}
\end{equation}
Recall that for multiindex $\alpha$ we have 
\begin{equation}
|\mathbf{h}|^n_{0,2}:=\max_{i\in \left\lbrace 1,\cdots ,n \right\rbrace} \sum_{|\alpha|\leq 2}\sup_{x\in {\mathbb R}^n}{\Big |}\frac{\partial^{\alpha}}{\partial x^{\alpha}}h_i{\Big |}.
\end{equation}
This defines $C>0$ if we consider equality in (\ref{Cconst}).
The last term in the definition (\ref{definitioncr}) is due to the fact that we do the estimates for local $|.|_{1,2}$-norms, i.e., on the domains $(l-1,l]\times {\mathbb R}^n$ and global $H^2$ norms for the functions $v^{r,\rho,l}_i(\tau,.)$ and $r^l_i(\tau,.)$. We have to verify that our choice of $C$ is really suuficient in order to keep the controlled value functions $v^{r,\rho,l}_i$ under control with respect to the relevant norms.
 Hence, we assume inductively that
\begin{equation}
|v^{r,\rho,l-1}_i|_0\leq C^{l-1}_0,~|v^{r,\rho,l-1}_i|_{0,1}\leq C^{l-1}_1,~~|v^{r,\rho,l-1}_i|_{0,2}\leq C^{l-1}_{0,2},~|v^{r,\rho,l-1}_i|_{1,2}\leq C^{l-1}_{1,2}
\end{equation}
for some finite constants $C^{l-1}_0,~C^{l-1}_1,~C^{l-1}_{1,2}$. In the end we shall have
\begin{equation}
C^{l-1}_{1,2}\leq C_{1,2}:=C_r
\end{equation}
and we may set $C^{l-1}_0:=C_{1,2}$, and $C^{l-1}_1=C^{l-1}_{1,2}$, and $C^{l-1}_{0,2}:=C^{l-1}_{1,2}$. This is shown next in (iii) below.

ad iii) We have defined the functions $\phi^l_i$ and $r^l_i$ in terms of the functions $v^{r,\rho,l-1}_i$ and $r^{l-1}_i$. In the case $l=1$ we start with $v^{r,\rho,l-1}_i=h_i$, the initial data and with $r^0_i=r^1_i\equiv 0$. We have to estimate $v^{r,\rho,1}_i$ first in this special situation. Then we proceed with the case $l\geq 2$.
In order to finish the first induction step $l=1$ we have to consider the growth of the local solution of $\mathbf{v}^{r,\rho,1}$. Note that for this first step this equals the local solution of the Navier-Stokes equation $\mathbf{v}^{\rho,l}$ because we have $\mathbf{r}^1\equiv 0$. Therefore at time step $l=1$ we may apply the machinery of the first step of this proof. Hence for $l=1$ we consider the functional series $\left( \mathbf{v}^{\rho,k,1}\right)_k$, where
\begin{equation}\label{vrhok1*}
 v^{\rho,k,1}_i=v^{\rho,0,1}_i+\sum_{m=1}^k\delta v^{\rho,m,1}_i.
 \end{equation}
 This series converges to the local solution $\mathbf{v}^{\rho,l}$ of the Navier-Stokes equation.
 Note that our symbolism allows for two alternative expressions for the members of this series. It can be denoted by $v^{r,\rho,k,1}_i=v^{\rho,k,l}_i+r^1_i=v^{\rho,k,l}_i$, or by the values of the map $F_1$ of the first step of this proof. 
Recall that the series (\ref{vrhokl*}) can be generated by an iterative application of the map
\begin{equation}\label{Fldef}
 F_1:\mathbf{f}\rightarrow \mathbf{v}^{f,\rho,l},
\end{equation}
as described in step 1 above in case of general time step number $l$ starting with the initial data $\mathbf{h}$.
Then the function $F_1(\mathbf{h})=\mathbf{v}^{h,\rho,1}=\mathbf{v}^{\rho,0,1}$ satisfies the equation
\begin{equation}\label{Navlerayiterlfff1}
\left\lbrace \begin{array}{ll}
\frac{\partial v^{h,\rho,1}_i}{\partial \tau}-\rho_l\nu\sum_{j=1}^n \frac{\partial^2 v^{h,\rho,1}_i}{\partial x_j^2} 
+\rho_l\sum_{j=1}^n h_j\frac{\partial v^{h,\rho,1}_i}{\partial x_j}= \\
\\ \hspace{1cm}\rho_l\int_{{\mathbb R}^n}\left( \frac{\partial}{\partial x_i}K_n(x-y)\right) \sum_{j,k=1}^n\left( \frac{\partial h_k}{\partial x_j}\frac{\partial h_j}{\partial x_k}\right) (\tau,y)dy,\\\\
\\
\mathbf{v}^{h,\rho,1}(0,.)=\mathbf{h}.
\end{array}\right.,
\end{equation}
where $1\leq i\leq n$.
Let $\Gamma^1_0$ be the fundamental solution of the equation
\begin{equation}\label{Navlerayiterlfff1fund}
\begin{array}{ll}
\frac{\partial v^{h,\rho,1}_i}{\partial \tau}-\rho_l\nu\sum_{j=1}^n \frac{\partial^2 v^{h,\rho,1}_i}{\partial x_j^2} 
+\rho_l\sum_{j=1}^n h_j\frac{\partial v^{h,\rho,1}_i}{\partial x_j}=0
\end{array}
\end{equation}
Let $C^0_h$, $C^1_h$, and $C^2_h$ be some constants such that
\begin{equation}
|h_i|_0\leq C^0_h,~|h_i|_1\leq C^1_h~,|h_i|_2\leq C^0_{1,2}
\end{equation}
for all $1\leq i\leq n$.
In terms of the fundamental solution $\Gamma^0_0$ of the equation  (\ref{Navlerayiterlfff1fund}) the components $v^{h,\rho,1}_i$ of the solution of (\ref{Navlerayiterlfff1}) have the representation
\begin{equation}\label{repvrho01}
\begin{array}{ll}
v^{\rho,0,1}_i(\tau,x)=\int_{{\mathbb R}^n}h_i(y)\Gamma^1_0(\tau,x;0,y)dy\\
\\
+\int_0^{\tau}\int_{{\mathbb R}^n}\rho_1\int_{{\mathbb R}^n}\left( \frac{\partial}{\partial x_i}K_n(y-z)\right) \sum_{j,k=1}^n\left( \frac{\partial h_k}{\partial x_j}\frac{\partial h_j}{\partial x_k}\right) (s,z)\Gamma^1_0(\tau,x;s,y)dydzds.
\end{array}
\end{equation}
From the maximum principle (cf. also Corollary 8.1.3. of \cite{Kr}) we observe that 
\begin{equation}
{\big |}\int_{{\mathbb R}^n}h_i(y)\Gamma^1_0(\tau,x;0,y)dy{\big |}\leq |h|_0
\end{equation}
for all $(\tau,x)\in [0,1]\times {\mathbb R}^n$.
Hence, the solution $\mathbf{v}^{h,\rho,1}=\mathbf{v}^{\rho,0,1}=\mathbf{v}^{r,\rho,0,1}$ of equation (\ref{repvrho01}) satisfies
\begin{equation}
\begin{array}{ll}
|v^{h,\rho,1}_i|_0\leq C^0+|\rho_1\int_{{\mathbb R}^n}\left( \frac{\partial}{\partial x_i}K_n(x-y)\right) \sum_{j,k=1}^n\left( \frac{\partial h_k}{\partial x_j}\frac{\partial h_j}{\partial x_k}\right) (\tau,y)dy|C_{\Gamma}\\
\\
\leq C^0_h+\rho_1C_{\Gamma}C_{1,2}\leq C_{1,2}+\rho_1C_{\Gamma}C_{1,2}
\end{array}
\end{equation}
where the definition of $C_{1,2}$ in terms of the function $h$ is used. We are quite generous with the use of the bound $C_{1,2}$.  As we described in the first step of this proof we may use the adjoint of the fundamental solution $\Gamma^l_0$ and shift derivatives, we get the estimate 
\begin{equation}\label{repvrho01}
\begin{array}{ll}
|v^{\rho,0,1}_i|_{1,2}\leq C_{\Gamma}^2\left( C^0_{1,2}+\rho_1C_{1,2}\right)
\end{array}
\end{equation}
Choosing
\begin{equation}
\rho_1\leq \frac{1}{  4C_{\Gamma}^2C_{1,2}},
\end{equation}
we get
\begin{equation}\label{repvrho01}
\begin{array}{ll}
|v^{\rho,0,1}_i|_{1,2}\leq C_{\Gamma}^2\left( C^0_{1,2}+\frac{1}{4}\right)
\end{array}
\end{equation}
Since $\mathbf{r}^1\equiv 0$ we have convergence of the local scheme on $[0,1]\times {\mathbb R}^n$ where we choose $C_r=0$ and $l=1$ in the definition of $\rho_l$ of the second step (or first step) of this proof.
We may use $\rho_1$ with
\begin{equation}
\rho_1\leq \frac{1}{C^*_n4C_{\Gamma}^2C_{1,2}},
\end{equation}
in order to get an estimate similar to (\ref{repvrho01}) for the vector-valued function $\mathbf{v}$. 
Adding the estimate of the correction $\mathbf{v}^{\rho,1}-\mathbf{v}^{\rho,0,1}=\sum_{k\geq 1}|\delta v^k_i|_{1,2}\leq 1/4$ from the first step of this proof we get
\begin{equation}\label{repvrho01end}
\begin{array}{ll}
|\mathbf{v}^{\rho,1}|^n_{1,2}=\sum_{i=1}^n|v^{\rho,1}_i|_{1,2}
\leq C_{\Gamma}^2 \left( C^0_{1,2}+1\right)
\end{array}
\end{equation}
where the constant $C_{\Gamma}\geq 1$ is from step 1 of this proof. Now  we have to show that this constant $C_{1,2}=C_r$ is preserved for all $l\geq 2$. It is clear that our definition of the constant $C_r$ fits for the first step $l=1$, since we have chosen $\mathbf{r}^1\equiv 0$ in this first step. Indeed the constant $C_r$ is determined in the induction step with respect to the time step number $l$. Furthermore we have determined $\rho_1$ for the first time step according to the convergence rule of the local scheme established in step 1 of this proof.
Let us assume that $\mathbf{v}^{r,\rho,m}$ and $\mathbf{r}^{m}$ have been constructed for $m=1,\cdots,l-1$ as classical local solutions on the domains $[m-1,m]\times {\mathbb R}^n$. Furthermore assume that we have
\begin{equation}
|v^{r,\rho,m}_i|_0\leq C~~\mbox{and}~~|v^{r,\rho,m}_i|_{1,2}\leq C_{1,2}
\end{equation}
for all $1\leq m\leq l-1$ and $1\leq i\leq n$. Furthermore, let us assume that
\begin{equation}
\begin{array}{ll}
|v^{r,\rho,m}_i(\tau,.)|_{L^2}\leq  C_{1,2}\\
\\
|v^{r,\rho,m}_i(\tau,.)|_{H^2}\leq  C_{1,2} 
\end{array}
\end{equation}
for all $\tau\in [m-1,m]$.
\begin{rem}
Since we consider $C^*_n$ in the definition of the constant $C_{1,2}$ to be generic we note that
\begin{equation}
|\mathbf{v}^{r,\rho,m}|^n_{1,2}\leq C_{1,2}.
\end{equation}
\end{rem}
 
 Let us go back to the local Navier-Stokes equation on the domain $[l-1,l]\times{\mathbb R}^n$ which we wrote in the form (\ref{Navlerayrrhol*++}). Let us repeat this equation here for the convenience of the reader. We have
\begin{equation}\label{Navlerayrrhol*++rep}
\left\lbrace \begin{array}{ll}
\frac{\partial v^{r,\rho,l}_i}{\partial \tau}-\rho_l\nu\sum_{j=1}^n \frac{\partial^2 v^{r,\rho,l}_i}{\partial x_j^2} 
+\rho_l\sum_{j=1}^n v^{r,\rho,l}_j\frac{\partial v^{r,\rho,l}_i}{\partial x_j}=\psi^l_i,\\
\\
\mathbf{v}^{r,\rho,l}(l-1,.)=\mathbf{v}^{r,\rho,l-1}(l-1,.),
\end{array}\right.
\end{equation}
where
\begin{equation}\label{Navlerayrrhol*+++}
\begin{array}{ll}
\psi^l_i=r^l_{i,\tau}-\rho_l\nu \Delta r^l_i+\rho_l\sum_{j=1}^n r^l_j\frac{\partial r^l_i}{\partial x_j}\\
\\
-\rho_l\int_{{\mathbb R}^n}\left( \frac{\partial}{\partial x_i}K_n(x-y)\right) \sum_{j,k=1}^n\left( \frac{\partial r^l_k}{\partial x_j}\frac{\partial r^l_j}{\partial x_k}\right) (\tau,y)dy\\
\\
+\rho_l\sum_{j=1}^n r^l_j\frac{\partial v^{r,\rho,l}_i}{\partial x_j}+\rho_l\sum_{j=1}^n v^{r,\rho,l}_j\frac{\partial r^{l}_i}{\partial x_j}\\
\\ -2\rho_l\int_{{\mathbb R}^n}\left( \frac{\partial}{\partial x_i}K_n(x-y)\right) \sum_{j,k=1}^n\left( \frac{\partial r^{l}_k}{\partial x_j}\frac{\partial v^{r,\rho,l}_j}{\partial x_k}\right) (\tau,y)dy\\
\\
 +\rho_l\int_{{\mathbb R}^n}\left( \frac{\partial}{\partial x_i}K_n(x-y)\right) \sum_{j,k=1}^n\left( \frac{\partial v^{r,\rho,l}_k}{\partial x_j}\frac{\partial v^{r,\rho,l}_j}{\partial x_k}\right) (\tau,y)dy.
\end{array}
\end{equation}
Note that the right side of this equation involves the function $\mathbf{v}^{r,\rho,l}$ which is not known.
However, the function $\mathbf{r}^l$ is known by our construction in the previous substep, and determines together with the function $\mathbf{v}^{r,\rho,l-1}(l-1,.)$ the right side in the equation for the first approximation $\mathbf{v}^{r,\rho,0,l}$ of the function $\mathbf{v}^{r,\rho,0,l}$. We may estimate the growth of the function $\mathbf{v}^{r,\rho,0,l}$ first and then estimate the growth of the function $\mathbf{v}^{r,\rho,l}$ by estimating the difference
\begin{equation}
 \mathbf{v}^{r,\rho,l}-\mathbf{v}^{r,\rho,0,l}=\sum_{k=1}^{\infty}\delta\mathbf{v}^{r,\rho,k,l}
\end{equation}
which we know from the estimates of the previous steps of our proof. We have
\begin{equation}\label{Navlerayrrhol*++0step4}
\left\lbrace \begin{array}{ll}
\frac{\partial v^{r,\rho,0,l}_i}{\partial \tau}-\rho_l\nu\sum_{j=1}^n \frac{\partial^2 v^{r,\rho,0,l}_i}{\partial x_j^2} 
+\rho_l\sum_{j=1}^n v^{r,\rho,l-1}_j(l-1,.)\frac{\partial v^{r,\rho,0,l}_i}{\partial x_j}=\psi^{l,0}_i\\
\\
\mathbf{v}^{r,\rho,0,l}(l-1,.)=\mathbf{v}^{r,\rho,l-1}(l-1,.),
\end{array}\right.
\end{equation}
where
\begin{equation}\label{Navlerayrrhol*++0step4*}
\begin{array}{ll}
\psi^{l,0}_i=
r^l_{i,\tau}-\rho_l\nu \Delta r^l_i+\rho_l\sum_{j=1}^n r^l_j\frac{\partial r^l_i}{\partial x_j}\\
\\
-\rho_l\int_{{\mathbb R}^n}\left( \frac{\partial}{\partial x_i}K_n(x-y)\right) \sum_{j,k=1}^n\left( \frac{\partial r^l_k}{\partial x_j}\frac{\partial r^l_j}{\partial x_k}\right) (\tau,y)dy\\
\\
+\rho_l\sum_{j=1}^n r^l_j\frac{\partial v^{r,\rho,l-1}_i}{\partial x_j}+\rho_l\sum_{j=1}^n v^{r,\rho,l-1}_j\frac{\partial r^{l}_i}{\partial x_j}\\
\\ -2\rho_l\int_{{\mathbb R}^n}\left( \frac{\partial}{\partial x_i}K_n(x-y)\right) \sum_{j,k=1}^n\left( \frac{\partial r^{l}_k}{\partial x_j}\frac{\partial v^{r,\rho,l-1}_j}{\partial x_k}\right) (\tau,y)dy\\
\\
+\rho_l\int_{{\mathbb R}^n}\left( \frac{\partial}{\partial x_i}K_n(x-y)\right) \sum_{j,k=1}^n\left( \frac{\partial v^{r,\rho,l-1}_k}{\partial x_j}\frac{\partial v^{r,\rho,l-1}_j}{\partial x_k}\right) (l-1,y)dy.
\end{array}
\end{equation}
We have determined $\mathbf{r}^l$ in the previous time step such that the first three terms on the right side of (\ref{Navlerayrrhol*++0step4}) can be replaced.  Indeed recall that we have
\begin{equation}\label{Navlerayrrhol+++++r}
\begin{array}{ll}
r^l_{i,\tau}-\rho_l\nu \Delta r^l_i+\rho_l\sum_{j=1}^n r^{l-1}_j(l-1,.)\frac{\partial r^{l}_i}{\partial x_j}=\\
\\
+\rho_l\int_{{\mathbb R}^n}\left( \frac{\partial}{\partial x_i}K_n(x-y)\right) \sum_{j,k=1}^n\left( \frac{\partial r^{l-1}_k}{\partial x_j}\frac{\partial r^{l-1}_j}{\partial x_k}\right) (l-1,y)dy\\
\\
-\rho_l\sum_{j=1}^n r^{l-1}_j\frac{\partial v^{r,\rho,l-1}_i}{\partial x_j}-\rho_l\sum_{j=1}^n v^{r,\rho,l-1}_j\frac{\partial r^{l-1}_i}{\partial x_j}\\
\\ +2\rho_l\int_{{\mathbb R}^n}\left( \frac{\partial}{\partial x_i}K_n(x-y)\right) \sum_{j,k=1}^n\left( \frac{\partial r^{l-1}_k}{\partial x_j}\frac{\partial v^{r,\rho,l-1}_j}{\partial x_k}\right) (\tau,y)dy\\
\\-\rho_l\int_{{\mathbb R}^n}\left( \frac{\partial}{\partial x_i}K_n(x-y)\right) \sum_{j,m=1}^n\left( \frac{\partial v^{r,\rho,l-1}_m}{\partial x_j}\frac{\partial v^{r,\rho,l-1}_j}{\partial x_m}\right) (l-1,y)dy+\phi^l_i.
\end{array}
\end{equation}
We may rewrite the function $\psi^{l,0}_i$ of (\ref{Navlerayrrhol*++0step4*}), i.e., the right side of the first equation in (\ref{Navlerayrrhol*++0step4}) in the form
\begin{equation}
\psi^{l,0}_i=\phi^{l}_i+\left(\psi^{l,0}_i-\phi^l_i\right), 
\end{equation}
where
\begin{equation}\label{Navlerayrrhol*++000+}
\begin{array}{ll}
\psi^{l,0}_i-\phi^l_i=\rho_l\sum_{j=1}^n (r^{l}_j-r^{l-1}_j)\frac{\partial r^{l}_i}{\partial x_j}\\
\\
-\rho_l\int_{{\mathbb R}^n}\left( \frac{\partial}{\partial x_i}K_n(x-y)\right) \sum_{j,k=1}^n\left( \frac{\partial r^{l}_k}{\partial x_j}\frac{\partial r^{l}_j}{\partial x_k}\right) (\tau,y)dy\\
\\
+\rho_l\int_{{\mathbb R}^n}\left( \frac{\partial}{\partial x_i}K_n(x-y)\right) \sum_{j,k=1}^n\left( \frac{\partial r^{l-1}_k}{\partial x_j}\frac{\partial r^{l-1}_j}{\partial x_k}\right) (l-1,y)dy\\
\\
+\rho_l\sum_{j=1}^n (r^l_j-r^{l-1}_j)\frac{\partial v^{r,\rho,l-1}_i}{\partial x_j}+\rho_l\sum_{j=1}^n v^{r,\rho,l-1}_j\left( \frac{\partial r^{l}_i}{\partial x_j}-\frac{\partial r^{l-1}_i}{\partial x_j}\right) \\
\\ -2\rho_l\int_{{\mathbb R}^n}\left( \frac{\partial}{\partial x_i}K_n(x-y)\right) \sum_{j,k=1}^n\left( \left( \frac{\partial r^{l}_k}{\partial x_j}(\tau,y)-\frac{\partial r^{l-1}_k}{\partial x_j}(l-1,y)\right) \frac{\partial v^{r,\rho,l-1}_j}{\partial x_k}(l-1,y)\right) dy.
\end{array}.
\end{equation}
Indeed, this difference of $\psi^{l,0}_i$ and $\phi^l_i$ may be obtained by subtracting equation (\ref{Navlerayrrhol+++++}) from equation (\ref{Navlerayrrhol*++00}) which we computed in (\ref{Navlerayrrhol*++000}) above. 
Note that all terms on the right side have the factor $\rho_l$. Estimating this difference boils down to estimating the difference $\mathbf{r}^l-\mathbf{r}^{l-1}$. In order to control the growth of the functions $v^{r,\rho,0,l}_i$ the estimate should be such that the functions $\phi^l_i$ dominate the differences $\psi^{l,0}_i-\phi^l_i$ in critcal areas where the releavnt notms of the functions $v^{r,\rho,l-1}_i$ exceed a certain level. However, this is staisfied by the construction of $\phi^l_i$: if $r^{l-1}_i(l-1,.)$ and $v^{r,\rho,l-1}_i(l-1,.)$ both are large with repsect to any of the relevant norms, then the control $r^l$ ensures control of $v^{r,\rho,0,l}_i$ on the domain $[l-1,l]\times {\mathbb R}^n$. If the control function $r^l_i$ is small, i.e. comparable to $\rho_l\sim \frac{1}{C^3}$ then the source $\phi^{l,s,v}_i$ ensures the control of the function $v^{r,\rho,l}_i$ on the domain $[l-1,l]\times {\mathbb R}^n$.
 
 Summarizing we first estimate $v^{r,\rho,0,l}_i$, where we rewrite (\ref{Navlerayrrhol*++0step4}) in the form
 \begin{equation}\label{Navlerayrrhol*++0step4+}
\left\lbrace \begin{array}{ll}
\frac{\partial v^{r,\rho,0,l}_i}{\partial \tau}-\rho_l\nu\sum_{j=1}^n \frac{\partial^2 v^{r,\rho,0,l}_i}{\partial x_j^2} 
+\rho_l\sum_{j=1}^n v^{r,\rho,l-1}_j(l-1,.)\frac{\partial v^{r,\rho,0,l}_i}{\partial x_j}=\phi^ l_i+\left( \psi^{l,0}_i-\phi^l_i\right) \\
\\
\mathbf{v}^{r,\rho,0,l}(l-1,.)=\mathbf{v}^{r,\rho,l-1}(l-1,.),
\end{array}\right.
\end{equation}
and then we estimate $v^{r,\rho,l}_i-v^{r,\rho,0,l}_i$, where we use the second step of this proof, and the fact that the difference $\psi^l_i-\psi^{l,0}_i$ is given by
\begin{equation}\label{Navlerayrrhol*+++l}
\begin{array}{ll}
\psi^l_i-\psi^{l,0}_i=\\
\\
+\rho_l\sum_{j=1}^n r^l_j\left( \frac{\partial v^{r,\rho,l}_i}{\partial x_j}-\frac{\partial v^{r,\rho,l-1}_i}{\partial x_j}\right) +\rho_l\sum_{j=1}^n \left( v^{r,\rho,l}_j-v^{r,\rho,l-1}_j\right) \frac{\partial r^{l}_i}{\partial x_j}\\
\\ -2\rho_l\int_{{\mathbb R}^n}\left( \frac{\partial}{\partial x_i}K_n(x-y)\right) \sum_{j,k=1}^n\left( \frac{\partial r^{l}_k}{\partial x_j}\left( \frac{\partial v^{r,\rho,l}_j}{\partial x_k}-\frac{\partial v^{r,\rho,l-1}_j}{\partial x_k}\right) \right) (\tau,y)dy\\
\\
 +\rho_l\int_{{\mathbb R}^n}\left( \frac{\partial}{\partial x_i}K_n(x-y)\right) \sum_{j,k=1}^n\left( \frac{\partial v^{r,\rho,l}_k}{\partial x_j}\frac{\partial v^{r,\rho,l}_j}{\partial x_k}\right) (\tau,y)dy\\
 \\
-\rho_l\int_{{\mathbb R}^n}\left( \frac{\partial}{\partial x_i}K_n(x-y)\right) \sum_{j,k=1}^n\left( \frac{\partial v^{r,\rho,l-1}_k}{\partial x_j}\frac{\partial v^{r,\rho,l-1}_j}{\partial x_k}\right) (\tau,y)dy.
\end{array}
\end{equation} 
Indeed this is essentially estimated in the second step of this proof once the first difference is estimated. 

Recall that
\begin{equation}
 |r^l_i|_0\leq C^0_r,~|r^l_i|_{1,2}\leq C^1_r,~\mbox{and }~|r^l_i|_{1,2}\leq C_r
\end{equation}
We see that we have to estimate the difference $\delta r^l_i\equiv r^{l}_j-r^{l-1}_j(l-1,.)$  in order to get the estimate of that difference. Well, we an afford to do a rough estimate for this difference in the form
\begin{equation}
|\delta r^l_i|_{1,2}= |r^{l}_j-r^{l-1}_j(l-1,.)|_{1,2}\leq 2C_r.
\end{equation}
Next we estimate the difference $\psi^{l,0}_i-\phi^l_i$. From equation (\ref{Navlerayrrhol*++000+}) we get 

\begin{equation}\label{diffest}
\begin{array}{ll}
|\psi^{l,0}_i-\phi^l_i|_{0,1}\leq \rho_l\left( 4nC_r^2+4n^2C_KC_r^2+4nC_rC_{0,2}+8C_Kn^2C_rC_{0,2}\right) \\
\end{array}.
\end{equation}
We get
\begin{equation}\label{diffest}
\begin{array}{ll}
|\psi^{l,0}_i-\phi^l_i|_{0,1}\leq \frac{1}{4}\\
\end{array}.
\end{equation}
where we choose
\begin{equation}
 \rho_l\leq \frac{1}{C^*_nC^2_r }.
\end{equation}
Note the additional factor $C_r$ and recall that $C^*_n$ is a generic constant dependent on dimension $n$. Recall also that we have determined $C_r$ in terms of the initial data function $\mathbf{h}$ above. Hence the scheme defined is global. We can now finish the proof.
In order to control the growth of the functions $\mathbf{v}^{r,\rho,l}$ and the function $\mathbf{r}^l$ at time step $l$ we fir $\mathbf{r}^l$ control 1) the growth of a linearized equation for $\mathbf{v}^{r,\rho,0,l}$, and 2) ensure that the correction $\mathbf{v}^{r\rho,l}-\mathbf{v}^{r,\rho,0,l}=\sum_{k=1}^{\infty}\delta \mathbf{v}^{r,\rho,k,0}$ is small enough by choosing $\rho_l$ appropriately. Let us look at the linear approximation $\mathbf{v}^{r,\rho,0,l}$ of the function $\mathbf{v}^{r,\rho,l}$ first. From equation (\ref{Navlerayrrhol*++0}) we get the representation
\begin{equation}\label{v0rep0}
\begin{array}{ll}
v^{r,\rho,0,l}_i(\tau,x)=\int_{{\mathbb R}^n}v^{r,\rho,l-1}(l-1,y)\Gamma^l_{v,0}(\tau,x;0,y)dy\\
\\
+\int_{l-1}^{\tau}\int_{{\mathbb R}^n}\psi^{l,0}_i(s,y)\Gamma^l_{v,0}(\tau,x;s,y)dyds
\end{array}
\end{equation}
where $\Gamma^l_{v,0}$ is the fundamental solution of
\begin{equation}
\begin{array}{ll}
\frac{\partial v^{r,\rho,0,l}_i}{\partial \tau}-\rho_l\nu\sum_{j=1}^n \frac{\partial^2 v^{r,\rho,0,l}_i}{\partial x_j^2}
+\rho_l\sum_{j=1}^n v^{r,\rho,l-1}_j(l-1,.)\frac{\partial v^{r,\rho,0,l}_i}{\partial x_j}=0,
\end{array}
\end{equation}
and $\psi^{l,0}_i$ is as in (\ref{Navlerayrrhol*++0}). We may rewrite
\begin{equation}\label{v0rep1}
\begin{array}{ll}
v^{r,\rho,0,l}_i(\tau,x)=\int_{{\mathbb R}^n}v^{r,\rho,l-1}(l-1,y)\Gamma^l_{v,0}(\tau,x;0,y)dy\\
\\
+\int_{l-1}^{\tau}\int_{{\mathbb R}^n}\left( \left( \psi^{l,0}_i(s,y)-\phi^l_i(s,y)\right) +\phi^l_i(s,y)\right) \Gamma^l_{v,0}(\tau,x;s,y)dyds
\end{array}
\end{equation}
where $\phi^l_i$ is constructed above and the difference $\psi^{l,0}_i-\phi^l_i$ is as defined in
(\ref{Navlerayrrhol*++000+}).  Applying the maximum principle we may estimate the first term on the right side of equation (\ref{v0rep1}) by the supremum (indeed, maximum) of the initial data. We have for all $(\tau,x)\in [l-1,l]\times {\mathbb R}^n$
\begin{equation}\label{v0rep11}
\begin{array}{ll}
|v^{r,\rho,0,l}_i(\tau,x)|\leq \sup_{(\tau,x)\in [l-1,l]\times {\mathbb R}^n}{\Big |}\int_{{\mathbb R}^n}v^{r,\rho,l-1}(l-1,y)\Gamma^l_{v,0}(\tau,x;0,y)dy\\
\\
+\int_{l-1}^{\tau}\int_{{\mathbb R}^n}\left( \frac{1}{4} +\phi^l_i(s,y)\right) \Gamma^l_{v,0}(\tau,x;s,y)dyds {\Big |}\leq \sup_{x\in {\mathbb R}^n}|v^{r,\rho,l-1}_i(l-1,x)|.
\end{array}
\end{equation}
For this part the reasoning is the same as before in the case of the function $\mathbf{r}^l$. We use the maximum principle in order to estimate the term
\begin{equation}
\int_{{\mathbb R}^n}v^{r,\rho,l-1}(l-1,y)\Gamma^l_{v,0}(\tau,x;0,y)dy
\end{equation}
and then consider the different cases making up the definition of $\phi^l_i$ together with the a priori estimate of the fundamental solution $\Gamma^l_{v,0}$ involved. Next we apply the second step of this proof which gives
\begin{equation}
\sum_{m=1}^{\infty} |\delta v^{r,\rho,k,l}_i|_{1,2}\leq \frac{1}{2}.
\end{equation}
We conclude that
\begin{equation}
|v^{r,\rho,l}_i|_{1,2}\leq C_{1,2}. 
\end{equation}
The reasoning that 
\begin{equation}
|v^{r,\rho,l}_i|_{H^2}\leq C_{1,2}
\end{equation}
reduces to the reasoning that
\begin{equation}\label{vrh2est}
|v^{r,\rho,0,l}_i|_{H^2}\leq C_{1,2}
\end{equation}
The estimate (\ref{vrh2est}) is obtained analogously to the 
$H^2$ estimate for $r^l_i$ in the second substep of this third substep.

\subsection{step 4: Global existence of classical solutions $\mathbf{v}^{\rho}$ and $\mathbf{v}$ resp.}
The functions $\phi^l_i$ are bounded functions with supremum less or equal to $1$ which vary with the time step number in general. At each time step $l\geq 1$ we defined them on the domain $(l-1,l]\times {\mathbb R}^n$. The functions $\mathbf{v}^{r,\rho}$ and $\mathbf{r}^l$ are only weakly differentiable with respect to time across the points $l=0,1,2,\cdots$ in general. They are also Lipschitz. Furthermore, in the construction of each time step the functions $\phi^l_i$ are only Lipschitz with respect to the spatial variables in general. However, this is sufficient in order to find that the solution $\mathbf{v}^{\rho}=\mathbf{v}^{r,\rho}+\mathbf{r}$ constructed is classical. Consider the function $\mathbf{v}^{r,\rho,l}$ constructed at time step $l\geq 1$. The components $v^{r,\rho,l}_i,~1\leq i\leq n$ have the representation
\begin{equation}
v^{r,\rho,l}_i=v^{r,\rho,0,l}_i+\sum_{k=1}^{\infty}\delta v^{r,\rho,k,l}_i,
\end{equation}
where $\delta v^{r,\rho,k,l}_i=v^{r,\rho,k,l}_i-v^{r,\rho,k-1,l}_i$ are determined successively by linear Cauchy problems with zero initial conditions. These linear Cauchy problems have bounded classical solutions with 
\begin{equation}
\lim_{\tau\downarrow l-1}\frac{\partial \delta v^{r,\rho,k,l}}{\partial \tau}(\tau,x)=0.
\end{equation}
This implies that the regular behavior of the function $\mathbf{v}^{r,\rho,l}$ with respect to time is determined by the behavior of the function $\mathbf{v}^{r,\rho,0,l}$ as $\tau\downarrow l-1$ and the initial data (resp. the final data of the previous time step) $\mathbf{v}^{r,\rho,0,l}(l-1,.)=\mathbf{v}^{r,\rho,l-1}(l-1,.)$. We conclude that the function $\mathbf{v}^{r,\rho}$ is uniformly bounded continuous with respect to time. Furthermore it is H\"{o}lder continuous with respect to the spatial variables uniformly with respect to time.
Similarly, the function $\tau \rightarrow \mathbf{r}(\tau,x)$ is only weakly differentiable at the integer values $l\in {\mathbb N}$. Especially it is uniformly continuous with respect to the time variable $\tau$ and it is H\"{o}lder continuous with respect to the spatial variables uniformly in $\tau$. Hence we conclude that the solution $\mathbf{v}^{r,\rho,l}$ on $[l-1,l]\times {\mathbb R}^n$ of the Navier-Stokes equation in transformed time coordinates $\tau$, i.e., the function
\begin{equation}
\mathbf{v}^{\rho,l}=\mathbf{v}^{r,\rho,l}-\mathbf{r}^l
\end{equation}
shares these properties.
 Hence, if we consider the family of fundamental solutions $\Gamma^{\rho,l}_v$ of the equations
\begin{equation}\label{Gammarholstep4}
\frac{\partial \Gamma^{\rho,l}_v}{\partial \tau}-\rho_l\nu \Delta \Gamma^{\rho,l}_v+\rho_l\sum_{j=1}^nv^{\rho,l}_j\frac{\partial \Gamma^{\rho,l}_v}{\partial x_j}=0,
\end{equation}
then we observe 1) that all
exist in their Levy expansion form since the coefficient functions $v^{\rho,l}_j$ are uniformly continuous with respect to time and H\"{o}lder continuous with respect to the spatial variables 2) we can build a global bounded continuous coefficient functions $v^{\rho}_j:[0,\infty)\times {\mathbb R}^n$ out of the local coefficient functions $v^{\rho,l}_j$, where the local restrictions of $v^{\rho}_j$ to the domain $[l-1,l]\times {\mathbb R}^n$ are equal to the functions $v^{\rho,l}_j$. The different coefficients $\rho_l$ in the family of equations (\ref{Gammarholstep4}) are artefacts of the time transformations, of course. We can get rid of them by transforming back to original time coordinates
 \begin{equation}
\tau \rightarrow t(\tau):=\rho_l\tau,~\mbox{if}~\tau\in [l-1,l]
\end{equation}
for all $l\geq 1$. Note that we can consider the function $v^{\rho}_j$ as given because we have constructed it- the turn to the fundamental solution then gives us global classical solutions. 
 Furthermore, the fundamental solutions $\Gamma^l_v$ of
\begin{equation}
\frac{\partial \Gamma^{l}_v}{\partial t}-\nu \Delta \Gamma^{l}_v+\sum_{j=1}^nv^{l}_j\frac{\partial \Gamma^{l}_v}{\partial x_j}=0
\end{equation}
exists on the transformed domains and in original coordinates, and we can build global bounded continuous coefficient functions $v_j:[0,\infty)\times {\mathbb R}^n$ out of the local coefficient functions $v^{l}_j$ in original time coordinates, where the local restrictions of $v_j$ to the domain $[l-1,l]\times {\mathbb R}^n$ are equal to the functions $v^l_j$. The related global fundamental solution can be obtained also by successive use of the the relation
\begin{equation}
\Gamma (t,x;s,y)=\int_{{\mathbb R}^n}\Gamma (t,x;s_1,z)\Gamma (s_1,z;s,y)dz
\end{equation}
for $t >s_1>s$, and where $s_1$ will run through the time step sizes $\sum_{m=1}^l\rho_m$ in orignal time coordinates. We conclude that
the fundamental solution $\Gamma_v$ of
\begin{equation}
\frac{\partial \Gamma_v}{\partial t}-\nu \Delta \Gamma_v+\sum_{j=1}^nv_j\frac{\partial \Gamma_v}{\partial x_j}=0 
\end{equation}
exists on $[0,T]\times {\mathbb R}^n$ for arbitrary $T>0$, where $v_i$ is the function that equals $v^l_i$ on the domain $[l-1,l]\times {\mathbb R}^n$ for all $l\geq 1$.
Now for all $t\in [0,T]$ and $x\in {\mathbb R}^n$ we have the representation
\begin{equation}\label{navstokesfundreg}
\begin{array}{ll}
v_i(t,x)=\int_{{\mathbb R}^n}h_i(y)\Gamma_v(t,x;0,y)dy+\int_0^t\int_{{\mathbb R}^n}\left( \frac{\partial v_m}{\partial x_l}\frac{\partial v_l}{\partial x_m}\right) (s,z)\times\\
\\
\times K_{,i}(z-y)\Gamma_v(t,x;s,y)dsdydz
\end{array}
\end{equation}
for $1\leq i\leq n$. Hence we have that $v_i\in C^{1,2}_b\left(\left[0,T\right]\times {\mathbb R}^n\right)$ for all $T>0$. 
Furthermore the latter representation and the standard a priori estimates for classical fundamental solutions together with our assumptions on the initial data $h_i$ show us that we have
\begin{equation}
v_i(t,.)\in H^{2}
\end{equation}
for all $t\geq 0$, where we use Gaussian estimates and Young's inequality for convolutions.
\end{proof}

\section{Regularity, uniqueness, and extensions with external forces}

The existence of global bounded classical solutions is essential in order to prove further regularity and uniqueness. Furthermore this is essential in order to extend the proof to equations with external forces, and to study asymptotic behavior. It may be well known that bounded classical solutions lead to full regularity. However, we shall show this for the sake of completeness. 

Next we prove regularity, i.e. we prove that for $1\leq i\leq n$ we have
\begin{equation}
v_i\in C^{\infty}\left(\left[0,\infty \right) \times{\mathbb R}^n \right), 
\end{equation}
and for all $t\in [0,\infty)$ and $s\in {\mathbb R}$
\begin{equation}
v_i(t,.)\in H^s\left({\mathbb R}^n\right).
\end{equation}
One way to do this starts from the representation of the solution in (\ref{navstokesfundreg})
Note that we may write
\begin{equation}\label{Navsum}
\begin{array}{ll}
v_i(t,x)=v^b_i(t,x)+p_{,i}(t,x),
\end{array}
\end{equation}
where
\begin{equation}
\begin{array}{ll}
v^b_i(t,x)=\int_{{\mathbb R}^n}h_i(y)\Gamma_v(t,x;0,y)dy
\end{array}
\end{equation}
solves the Cauchy problem for a Burgers type equation (but with first order coefficients $v_i$ of the solution of the Navier-Stokes equation), i.e.,
\begin{equation}\label{Cauchyreg}
\left\lbrace \begin{array}{ll}
\frac{\partial v^b_i}{\partial t}=\nu\sum_{j=1}^n \frac{\partial^2 v^b_i}{\partial x_j^2} 
-\sum_{j=1}^n v_j\frac{\partial v^b_i}{\partial x_j},\\
\\
\mathbf{v}(0,.)=\mathbf{h},
\end{array}\right.
\end{equation}
and $p_{,i}$ describes the gradient of the pressure, i.e., the solutions to the Poisson equations
\begin{equation}\label{poissonreg}
\begin{array}{ll} 
- \Delta p_{,i}=\sum_{j,k=1}^n \frac{\partial}{\partial x_i}\left( \frac{\partial}{\partial x_k}v_j\right) \left( \frac{\partial }{\partial x_j}v_k\right)
\end{array}
\end{equation}
for $1\leq i\leq n$. Note that in terms of the Navier-Stokes solution function $\mathbf{v}$ the first equation in (\ref{Cauchyreg}) is a linear parabolic equation for $v^b_i$ for each $1\leq i\leq n$. Hence a simple strategy for regularity is the following: start with some known regularity of $\mathbf{v}$ and consider (\ref{Navsum}). Then use regularity theory for linear parabolic equations in order to prove more regularity for the summand $\mathbf{v}^b$. Then starting with some regularity for the right side of (\ref{poissonreg}) use elliptic regularity theory for Poission equations in order to get more regularity for $p_{,i}$. This gives more regularity for $\mathbf{v}$. Iteration of the process leads to full regularity.
Since the classical solution $\mathbf{v}$ is known local regularity of scalar linear parabolic equations is sufficient in this context. We may use a local result which we apply to all domains $[m,m+1]\times\alpha+B^3_n(\alpha)$ with $\alpha \in {\mathbb Z^n}$ and natural numbers $m\geq 0$, and where $B^3_n$ is a ball of radius $3$ around $\alpha$, i.e.,
\begin{equation}
B^3_n(\alpha):=\left\lbrace x=\alpha +y||y|\leq 3 \right\rbrace 
\end{equation}
(we may use any radius which leads to a cover of the whole domain as $\alpha \in {\mathbb Z^n}$). We have
\begin{thm}
Consider a linear parabolic equation 
\begin{equation}
Lu\equiv \frac{\partial u}{\partial t}-\sum_{ij=1}^na_{ij}\frac{\partial^2 u}{\partial x_i\partial x_j}-\sum_{i=1}^nb_i\frac{\partial u}{\partial x_i}-cu=f,
\end{equation}
on a bounded domain $D$ and assume that derivatives of the coefficient functions and the source function $f$ satsify the following condition:
\begin{equation}
D^{\alpha}_xD^k_ta_{ij},~D^{\alpha}_xD^k_tb_{i},~D^{\alpha}_xD^k_tc,~D^{\alpha}_xD^k_tf
\end{equation}
are H\"{o}lder continuous with some exponent $\delta \in (0,1)$ on $D$ and $0\leq |\alpha|+2k\leq p$ and $m\leq q$. Then a solution $u$ of the equation $Lu=f$ satisfies that
\begin{equation}
D^{\alpha}_xD^k_t u\in C^{1+\delta, 2+\delta }(D)
\end{equation}
for all $0\leq |\alpha|+2k\leq p$ and $k\leq q$.
\end{thm}
Applying this theorem locally for all $\alpha+B^3_n$ and then inductively with respect to the time $m$ we get $v_i\in C^{1+\alpha,2+\alpha}\left([0,\infty)\times {\mathbb R}^n\right)$. 
 
In a first step we look at the Poisson equation for $p$ itself. Our main theorem tells us that we have 
\begin{equation}\label{l1est}
\sum_{j,k=1}^n \left( \frac{\partial}{\partial x_k}v_j\right) \left( \frac{\partial }{\partial x_j}v_k\right)\in L^1.
\end{equation}
(Well, it tells us that the right side is even in $H^1$).
Hence the convolution with the fundamental solution $n$ of the Poisson equation is a well defined locally integrable function, and we have
\begin{equation}\label{Poissonavreg}
\begin{array}{ll}
\Delta \left( \sum_{j,k=1}^n \left( \frac{\partial}{\partial x_k}v_j\right) \left( \frac{\partial }{\partial x_j}v_k\right)\star N\right)\\
\\
=\sum_{j,k=1}^n \left( \frac{\partial}{\partial x_k}v_j\right) \left( \frac{\partial }{\partial x_j}v_k\right).
\end{array}
\end{equation}
Here recall that an $L^1$ estimate for the term in (\ref{l1est}) can be reduced to an $L^2$ estimate. From our main result we know that for each time $t\in [0,\infty)$ the right side of (\ref{Poissonavreg}) is in $C^{1+\alpha}$.
We get more regularity of this distributive solution for $p$ from the standard result  
\begin{lem}\label{poissonlemreg}
Assume that $k\geq 0$, and $\Omega$ is an open set in ${\mathbb R}^n$. Assume $u$ is a distribution solution of
\begin{equation}
\Delta u=f
\end{equation}
where the  data $f\in C^{k+\alpha}\left( \Omega\right)$ along with $\alpha\in (0,1)$ and $k\geq 0$. Then $u\in C^{k+2+\alpha}\left( \Omega\right)$.
\end{lem}
Hence we have $p(t,.)\in C^{3+\alpha}$ in the first step. Iteration of this argument leads to spatial regularity.

The Navier-Stokes equation itself then tells us immediately some regularity of the first time derivative. Higher order time derivatives of the velocity can be expressed in terms of spatial derivatives of the velocity, of course. It follows that
\begin{equation}
\left( \frac{\partial}{\partial t}v_i\right) (t,.)\in C^{\infty} \mbox{ for all  }~t> 0.
\end{equation}
Another way to prove regularity is the following which uses both the Leray projection form and the original form of the Navier-Stokes equation. 
We shall use the following
We have $v_i\in C^{1,2}_b$ for all $1\leq i\leq n$. Hence $v_{i,j}(t,.)\in C^{1}_b$, which implies for all $t\in [0,\infty)$ that
\begin{equation}
\sum_{j,k=1}^n\left( v_{j,k}v_{k,j}\right)(t,.)\in C^1_b\subset C^{\alpha}.
\end{equation}
Next consider multiindices $\alpha=(\alpha_1,\cdots,,\alpha_n)$ and the multivariate partial derivative of order $\alpha$, i.e.,
\begin{equation}
\frac{\partial^{|\alpha|}}{\partial x^{|\alpha|}}=:D^{\alpha}_x=:_{,\alpha}.
\end{equation}
The spatial derivative of order $\alpha$ with $|\alpha|\geq 1$ of the Navier Stokes equation leads to
\begin{equation}
v_{i,\alpha,t}-\nu\Delta v_{i,\alpha}+\sum_{j=1}^n v_jv_{i,j,\alpha}=-p_{,i,\alpha}-
\sum_{j=1}^nv_{j,\alpha}v_{i,j}-\sum_{0<\beta <\alpha}\binom{\alpha}{\beta}v_{i,\beta}v_{i,j,\alpha-\beta}.
\end{equation}
Next, assume inductively that
\begin{equation}\label{induca}
v_{i,\beta}\in C^{1,2}_b,~~\mbox{for all }\beta<\alpha \mbox{, and }~1\leq i\leq n,
\end{equation}
and
\begin{equation}\label{inducb}
v_{i,\beta}(t,.)\in L^2,~~\mbox{for all }\beta<\alpha \mbox{ and all }t\in [0,\infty)
\end{equation}
for all $\beta<\alpha$, and $1\leq i\leq n$.
Then
\begin{equation}
v_{i,\alpha}(t,.)\in C^{1}_b\mbox{ $\&$ }\left( \sum_{0<\beta <\alpha}\binom{\alpha}{\beta}v_{i,\beta}v_{i,j,\alpha-\beta}\right)(t,.)\in C^{1}_b\subset C^{\alpha},
\end{equation}
uniformly in the time variable $t$.
Hence, we have the representation with fundamental solution $\Gamma$
\begin{equation}
\begin{array}{ll}
v_{i,\alpha}(t,x)=\int_{{\mathbb R}^n}h_{i,\alpha}(y)\Gamma(t,x,0,y)dy+\\
\\
-\int_0^t p_{,i,\alpha}(s,y)\Gamma(t,x,s,y)dyds\\
\\
+\sum_{j=1}^n\left( v_{j,\alpha}v_{i,j}\right) (s,y)\Gamma(t,x,s,y)dsdy\\
\\
+\left( \sum_{0<\beta <\alpha}\binom{\alpha}{\beta}v_{i,\beta}v_{i,j,\alpha-\beta}\right)(s,y)\Gamma(t,x,s,y)dsdy,
\end{array}
\end{equation}
where $\Gamma$ is the fundamental solution of
\begin{equation}
\frac{\partial \Gamma}{\partial t}-\nu \Delta \Gamma +\sum_{j=1}^n v_j \frac{\partial \Gamma}{\partial x_j}=0.
\end{equation}
Hence, we have
\begin{equation}
v_{i,\alpha}\in C^{1,2}_b
\end{equation} 
Regularity with repect to time can be treated similarly using regularity with respect to the spatial variables.

Next we have uniqueness. Assume that $v_1,v_2\in C^{1,2}_b$ are two solutions of the Navier-Stokes equation. Then from the basic energy estimate
\begin{equation}
\sup_{0\leq t\leq T}|v_1-v_2|_0\leq |(v_1(0,.)-v_2(0,.))|_0\exp\left(\int_0^T|\nabla v_2|_{L^{\infty}}dt \right)
\end{equation}
we get uniqueness. This is an application of Gr\"{o}nwall's lemma and can be found in standard texts. We mention it here for the sake of completeness. Finally we mention extensions to equations with external forces and asymptotic behaviour. The extension of the scheme to the Navier-Stokes equation wih external forces $f_{ex}$ is straightforword. At each time step in the substep $k=0$ a source term is added. The correction terms $\delta v^{r,\rho,l,k}_i$ are estimated using the adjoint and shift of a first spatial derivative to the force term as in step 1 of the proof.

\section{The algorithm}

The construction of the solution above can be extended to boundary value problems straightforwardly. We consider first initial-boundary value problems and second initial-boundary value problems. The observation is that the global existence for these problems boils down to the existence of related scalar first initial-boundary value problems and second initial-boundary value problems of parabolic type. Furthermore, the scheme proposed has linear subproblems of parabolic type which can be computed explicitly. Expansions of this form are considered in \cite{Ka5}. Note that there is a difference to the Taylor expansion (operator form) which applies only for affine coefficients in general (cf. \cite{BKS} and \cite{Kaff}). Furthermore the subproblems considered here have been implemented in the context of weighted Monte-Carlo methods in finance (cf. \cite{KKS}, \cite{FrKa2}, and \cite{FrKa}). The discussion here is rather conceptual. Further details of implementation and error estimates will be considered in subsequent paper.

Since we have proved that the solution is bounded we may set up a uniform time discretization scheme. The size of the time steps is limited by a time step size which ensures the convergence of the local time scheme. A lower bound of this time step size can be extracted the global existence proof. The step size numbers $\rho_l$ may be increased as time goes by if there are smoothing effects due to the strictly parabolic subproblems.

Since we have proved that the solution $v_i,~1\leq i\leq n$ of the incompressible Navier-Stokes equation is globaly bounded and H\"{o}lder continuous with respect to the spatial variables uniformly in time, we know that the fundamental solution $\Gamma_v$ of the equation
\begin{equation}
\frac{\partial \Gamma_v}{\partial t}-\nu \Delta \Gamma_v+\sum_j v_j\frac{\partial \Gamma_v}{\partial x_j}=0
\end{equation}
exists (given $\mathbf{v}$).
Hence, in terms of the solution itself the solution of the Cauchy problem has the representation (original time coordinates)
\begin{equation}
\begin{array}{ll}
v_i(t,x)=\int_{{\mathbb R}^n}h_i(y)\Gamma_v (t,x;0,y)dy+\\
\\
\int_{l-1}^l\int_{{\mathbb R}^n}\sum_{m,l=1}^n\left(v_{l,m}v_{m,l}\right)(s,z)K(y-z) \Gamma_v (t,x;s,y)dy.
\end{array}
\end{equation}
This representation cannot be used for computation of course, since we do not now the solution.
However, we have shown that for a time step size $\rho<1$ which is small enough we may compute the solution in a time -discretized scheme where at each time step $l$ we compute successive approximations $v^{\rho,k,l}$ with initial values from the previous time step.

 Note that the precise values from the the previous time step are $v^{\rho,l-1}$ which we do not know except for the case $l=1$ where $v^{\rho ,l-1}_i=h_i$, i.e. equal the original initial data (well even these have to be approximated upon implementation). Hence, the initial data of the previous time step are some given by a function $v^{\rho,l-1,*}$ which approximates $v^{\rho,l-1}_i$. The approximation of $v^{\rho,l}_i$ is then computed by iterative approximations $v^{\rho,k,l,*}_i$ for $k\geq 0$ where $v^{\rho,0,l,*}_i$ solves   
\begin{equation}\label{Navlerayl0intalg}
\left\lbrace \begin{array}{ll}
\frac{\partial v^{\rho,0,l,*}_i}{\partial \tau}-\rho\nu\sum_{j=1}^n \frac{\partial^2 v^{\rho,0,l,*}_i}{\partial x_j^2} 
+\rho\sum_{j=1}^n v^{\rho,l-1,*}_j\frac{\partial v^{\rho,0,l,*}_i}{\partial x_j}=\\
\\
  \rho\int_{{\mathbb R}^n}\left( \frac{\partial}{\partial x_i}K_n(x-y)\right) \sum_{j,k=1}^n\left( \frac{\partial v^{\rho,l-1,*}_k}{\partial x_j}\frac{\partial v^{\rho,l-1,*}_j}{\partial x_k}\right) (\tau,y)dy,\\
\\
\mathbf{v}^{\rho,0,l,*}(l-1,.)=\mathbf{v}^{\rho,l-1,*}(l-1,.),
\end{array}\right.
\end{equation}
and for $k\geq 1$ recursively defined functions $\mathbf{v}^{\rho,k,l,*}$ are determined by the respective solutions of
\begin{equation}\label{Navleraylkintalg}
\left\lbrace \begin{array}{ll}
\frac{\partial v^{\rho,k,l,*}_i}{\partial \tau}-\rho\nu\sum_{j=1}^n \frac{\partial^2 v^{\rho,k,l,*}_i}{\partial x_j^2} 
+\rho\sum_{j=1}^n v^{\rho,k-1,l,*}_j\frac{\partial v^{\rho,k,l,*}_i}{\partial x_j}=\\
\\ \rho\int_{{\mathbb R}^n}\left( \frac{\partial}{\partial x_i}K_n(x-y)\right) \sum_{j,k=1}^n\left( \frac{\partial v^{\rho,k-1,l,*}_k}{\partial x_j}\frac{\partial v^{\rho,k-1,l,*}_j}{\partial x_k}\right) (\tau,y)dy,\\
\\
\mathbf{v}^{\rho,k,l,*}(l-1,.)=\mathbf{v}^{\rho,l-1,*}(l-1,.).
\end{array}\right.
\end{equation}
We have to stop after finitely many steps. Hence we shall have
\begin{equation}
v^{\rho,l,*}_i=v^{\rho,m,l,*}_i
\end{equation}
for some $m$ if we perform iterations $k=0,\cdots ,m$ at time step $l$.
Let $\Gamma^{l,0}_*$ be the fundamental solution of the equation
\begin{equation}
\frac{\partial \Gamma^{l,0}_*}{\partial \tau}-\nu \Delta \Gamma^{l,0}_*+\sum_j v^{\rho,l-1,*}_j\frac{\partial \Gamma^{l,0}_*}{\partial x_j}=0.
\end{equation}
Then the solution $v^{\rho,l,0,*}_i$ of (\ref{Navlerayl0intalg}) has the representation
\begin{equation}
\begin{array}{ll}
v^{\rho,l,0,*}_i(t,x)=\int_{{\mathbb R}^n}v^{\rho,l-1,*}_i(l-1,y)\Gamma^{l,0}_* (\tau,x;l-1,y)dy+\\
\\
\int_{l-1}^{\tau}\int_{{\mathbb R}^n}\sum_{m,l=1}^n\left(v^{\rho,k-1,l,*}_{l,m}v^{\rho,k-1,l,*}_{m,l}\right)(s,z)K_{,i}(y-z) \Gamma^{l,0}_* (\tau,x;s,y)dy.
\end{array}
\end{equation}
Furthermore, let $\Gamma^{l,k}_*$ be the fundamental solution of the equation
\begin{equation}
\frac{\partial \Gamma^{l,k}_*}{\partial \tau}-\nu \Delta \Gamma^{l,k}_*+\sum_j v^{\rho,k-1,l,*}_j\frac{\partial \Gamma^{l,k}_*}{\partial x_j}=0.
\end{equation}
Then the solution (\ref{Navleraylkintalg}) has the representation
\begin{equation}
\begin{array}{ll}
v^{\rho,k,l,*}_i(t,x)=\int_{{\mathbb R}^n}v^{\rho,l-1,*}_i(y)\Gamma^{l,k}_* (\tau,x;0,y)dy+\\
\\
\int_{l-1}^{\tau}\int_{{\mathbb R}^n}\sum_{m,l=1}^n\left(v^{\rho,k-1,l,*}_{l,m}v^{\rho,k-1,l,*}_{m,l}\right)(s,z)K_{,i}(y-z) \Gamma^{l,k}_* (\tau,x;s,y)dy.
\end{array}
\end{equation}
This scheme involves
the fundamental solutions of equations of type 
\begin{equation}\label{parasystthm}
\frac{\partial u}{\partial t}=\sum_{j=1}^n \frac{\partial^2 u}{\partial x_j^2} 
+\sum_{i=1}^n b_{i}\frac{\partial u}{\partial x_i}
\end{equation}
essentially. For computational point of view it is interesting that 
 the solution has the locally pointwise valid representation
\begin{equation}\label{pj}
p(t,x,0,y)=\frac{1}{\sqrt{4\pi t}^n}\exp\left(-\frac{\sum_{i=1}^n \Delta x_i^2}{4t}\right)\left(\sum_{k=0}^{\infty}d_{k}(t,x,y)t^k \right),  
\end{equation}
for $j=1,\cdots ,n$, i.e., the representation is valid on some time interval. However, this fits with our scheme since this is defined locally in time anyway. Note that we have the coefficients outside the exponential. This implies that the first term is damping the polynomial terms $d_k$ as the moduli of the $\Delta x_i=(x_i-y_i)$ become large.
The coefficient functions $d_k$ have explicit representations in terms of the coefficient functions $b_i$: for $k=0$ we have
\begin{equation}\label{c0}
 d_{0}(t,x,y)= \exp\left( \sum_m (y_m-x_m)\int_0^1 b_m (t,y+s(x-y))ds\right) ,
\end{equation}

\begin{equation}
d_m(t,x,y)=\sum_{k=1}^m \frac{k}{m}d_{m-k}\int_0^1 R_{k-1}(t,y+s(x-y),y)s^{k}ds
\end{equation}
with 
\begin{equation}\label{tk}
\begin{array}{ll}
R_{k-1}(t,x,y)=&\frac{\partial}{\partial t}c_{k-1}+\Delta c_{k-1}+\sum_{l=1}^n\sum_{r=0}^{k-1}\left( \frac{\partial}{\partial x_l}c_{r}\frac{\partial}{\partial x_l}c_{k-1-r}\right)\\
\\
&+\sum_{i} b_i(x)\frac{\partial}{\partial x_i}c_{k-1}
\end{array}
\end{equation}
If the coefficients $b_i$ are given in terms of bounded analytical expansions (finite Fourier series for example), then the functions $d_k$ can be computed explicitly. Note that we cannot use Taylor expansions (operator form) as in \cite{BKS} since complete sets of analytic vectors are difficult to define if the coefficient functions are not affine (which is true in case of the Navier-Stokes equation).
These analytical expansions have been proved to be quite efficient in a different context (cf. \cite{FrKa,FrKa2, KKS, K2}). Fluid dynamical models in applied sciences have boundaries of course, so let have a look how our algorithm can be adapted to these situations. We consider boundary value problems which are related to the second initial-boundary value problem for scalar parabolic equations - cf. \cite{F2} for a classical treatment. 
We consider problems on the domain $[0,T]\times \Omega$, where $\Omega\subset {\mathbb R}^n$ is a bounded domain. Let $B:=\left\lbrace 0\right\rbrace \times \Omega$, $B_T=\left\lbrace 0\right\rbrace \times \Omega$, and let $S=\partial \Omega\setminus \left( B\cup B_T \right)$. Then we consider the following initial-boundary value problem on $[0,T]\times \Omega$. Let $\alpha_i :[0,T]\times \Omega\rightarrow {\mathbb R}$, and $g_i :[0,T]\times \Omega\rightarrow {\mathbb R}$ be $2n$ functions. We consider a problem for $v_i,~1\leq i\leq n$, where

\begin{equation}\label{Navleraybound}
\left\lbrace \begin{array}{ll}
\frac{\partial v_i}{\partial t}-\nu\sum_{j=1}^n \frac{\partial^2 v_i}{\partial x_j^2} 
+\sum_{j=1}^n v_j\frac{\partial v_i}{\partial x_j}=\\
\\ \hspace{1cm}\int\left( \frac{\partial}{\partial x_i}K_n(x-y)\right) \sum_{j,k=1}^n\left( \frac{\partial v_k}{\partial x_j}\frac{\partial v_j}{\partial x_k}\right) (t,y)dy,\\
\\
\frac{\partial}{\partial \nu}v_i(t,x)+\alpha_i(t,x)v_i(t,x)=g_i(t,x) \mbox{ on }[0,T]\times S,\\
\\
\mathbf{v}(0,.)=\mathbf{h}.
\end{array}\right.
\end{equation}
The scheme we proposed for the Cauchy problem can be adapted to this situation straightforwardly. Consider a time discretization in transformed coordinates and assume thhat $\mathbf{v}^{\rho,l-1}$ has computed for $ l-1$, where $l\geq 0$ If $l=0$ we set $\mathbf{v}^{\rho,-1}=\mathbf{h}$ which is known. We choose a fixed $\rho$ independent of the time step number $l$. For $l\geq 1$ we have functions $\mathbf{v}^{\rho,l}:[l-1,l]\times \Omega\rightarrow {\mathbf R}^n$ on successive domains where the final data of the function $\mathbf{v}^{\rho,l-1}$ are the initial data of the function ${\mathbf v}^{\rho,l}$. Each function $\mathbf{v}^{\rho,l}$ is determined as a limit of a functional series $\left( \mathbf{v}^{\rho,k,l}\right)_k$. Note that transformation to coordinates $t=\rho\tau$ leads to a domain $[0,RT]\times {\Omega}$, where $R=\frac{1}{\rho}$. Having computed the $k-1$ iteration step of the $l$th time step the problem for $\mathbf{v}^{\rho,k,l}$ is given by $n$ linear parabolic scalr problem which are classical initial-boundary value problems of second type. 
\begin{equation}\label{Navlerayboundkl}
\left\lbrace \begin{array}{ll}
\frac{\partial v^{\rho,k,l}_i}{\partial \tau}-\rho\nu\sum_{j=1}^n \frac{\partial^2 v^{\rho,k,l}_i}{\partial x_j^2} 
+\rho\sum_{j=1}^n v^{\rho,k-1,l}_j\frac{\partial v^{\rho,k,l}_i}{\partial x_j}=\\
\\ \hspace{1cm}\rho\sum_{j,m=1}^n\int_{{\mathbb R}^n}\left( \frac{\partial}{\partial x_i}K_n(x-y)\right) \left( \frac{\partial v^{\rho,k-1,l}_m}{\partial x_j}\frac{\partial v^{\rho,k-1,l}_j}{\partial x_m}\right) (\tau,y)dy,\\
\\
\frac{\partial}{\partial \nu}v^{\rho,k,l}_i(\tau,x)+\alpha_i(t,x)v^{\rho,k,l}_i(\tau,x)=g_i(t,x) \mbox{ on }[0,RT]\times S,\\
\\
\mathbf{v}^{\rho,k,l}(l-1,.)=\mathbf{v}^{\rho,l-1}(l-1,.),
\end{array}\right.
\end{equation}
where for $k=0$ we set $\mathbf{v}^{\rho,k-1,l}=\mathbf{v}^{\rho,-1,l}=\mathbf{v}^{\rho,l-1}(l-1,.)$ in order to determine the first order coefficients in the first step of a local iteration. We could add external forces in (\ref{Navlerayboundkl}) but we leave it out for simplicity.  The local convergence of the scheme (with the right choice of $\rho$ is proved similarly as in the first step of the proof of the main theorem above, i.e., by proving the convergence of the functional series in the form
\begin{equation}
 \mathbf{v}^{\rho,l}=\mathbf{v}^{\rho,0,l}+\sum_{k=1}^{\infty}\delta \mathbf{v}^{\rho,k,l},
\end{equation}
where
\begin{equation}
\delta \mathbf{v}^{\rho,k,l}\downarrow 0~\mbox{as}\downarrow 0.
\end{equation}
Note that the initial conditions and the boundary conditions for the functions $\delta v^{\rho,k,l}_i$ simplify to
\begin{equation}\label{Navlerayboundkldelta}
 \begin{array}{ll}
\frac{\partial}{\partial \nu}\delta v^{\rho,k,l}_i(\tau,x)+\alpha_i(t,x)\delta v^{\rho,k,l}_i(\tau,x)=0 \mbox{ on }[0,RT]\times S,\\
\\
\mathbf{v}^{\rho,k,l}(l-1,.)=0.
\end{array}
\end{equation}
This leads to representations of the solution in terms of the fundamental solutions $\Gamma^l_k$ of the equations
\begin{equation}\label{Navlerayboundklfund}
\begin{array}{ll}
\frac{\partial v^{\rho,k,l}_i}{\partial \tau}-\rho\nu\sum_{j=1}^n \frac{\partial^2 v^{\rho,k,l}_i}{\partial x_j^2} 
+\rho\sum_{j=1}^n v^{\rho,k-1,l}_j\frac{\partial v^{\rho,k,l}_i}{\partial x_j}=0.
\end{array}
\end{equation}
Let us start with the representation for $\mathbf{v}^{\rho,0,l}$. The solution is given in the form 
\begin{equation}
\begin{array}{ll}
v^{\rho,0,l}_i(\tau,x)=\int_{\Omega}v^{\rho,l-1}_i(l-1,y)\Gamma^l_0(\tau,y;0,y)dy\\
\\
+\int_{l-1}^{\tau}\int_{\Omega}\rho\sum_{j,m=1}^n\int_{\Omega}\left( \frac{\partial}{\partial x_i}K_n(y-z)\right) \left( \frac{\partial v^{\rho,l-}_m}{\partial x_j}\frac{\partial v^{\rho,l-1}_j}{\partial x_m}\right) (l-1,z)dz\Gamma^l_0(\tau,x;s,y)ds dy\\
\\
+\int_{l-1}^{\tau}\int_{S}\phi_i(s,y)\Gamma^l_0(\tau,x;s,y)dsdy.
\end{array}
\end{equation}
where $\phi_i$ is the solution of the integral equation
\begin{equation}
\frac{1}{2}\phi_i(\tau,x)=\int_{l-1}^{\tau}\int_{S}K_{\Gamma}(\tau,x;s,y)\phi_i(s,y)dsdy+f_i(\tau,x)
\end{equation}
along with the kernel
\begin{equation}
K_{\Gamma}(\tau,x;s,y)=\frac{\partial}{\partial \nu}\Gamma^l_0(\tau,x;s,y)+\alpha_i(\tau,x)\Gamma^l_0(\tau,x;s,y),
\end{equation}
and the functions $f_i$ which satisfy
\begin{equation}
 \begin{array}{ll}
f_i(\tau,x)=\int_{\Omega} K_{\Gamma}(\tau,x;0,y)v^{\rho,l-1}_i(l-1,y)dy-g_i(\tau,x)\\
\\
-\rho\int_{\Omega}\int_{l-1}^{\tau}\int_{\Omega} K_{\Gamma}(\tau,x;s,y)\\
\\
\times\sum_{j,m=1}^n\int_{\Omega}\left( \frac{\partial}{\partial x_i}K_n(y-z)\right) \left( \frac{\partial v^{\rho,l-1}_m}{\partial x_j}\frac{\partial v^{\rho,l-1}_j}{\partial x_m}\right) (l-1,z)dzdy 
 \end{array}
\end{equation}
Well for the corrections $\delta v^{\rho,k,l}_i$ these expressions become simplified. We 
\begin{equation}
\begin{array}{ll}
\delta v^{\rho,k,l}_i(\tau,x)=\\
\\
+\int_{l-1}^{\tau}\int_{\Omega}\rho\sum_{j,m=1}^n\int_{\Omega}\left( \frac{\partial}{\partial x_i}K_n(y-z)\right) \left( \frac{\partial v^{\rho,k-1,l}_m}{\partial x_j}\frac{\partial v^{\rho,k-1,l}_j}{\partial x_m}\right) (s,z)dz\Gamma^l_k(\tau,x;s,y)ds dy\\
\\
+\int_{l-1}^{\tau}\int_{S}\phi^k_i(s,y)\Gamma^l_k(\tau,x;s,y)dsdy.
\end{array}
\end{equation}
where $\phi^k_i$ is the solution of the integral equation
\begin{equation}
\frac{1}{2}\phi^k_i(\tau,x)=\int_{l-1}^{\tau}\int_{S}K^k_{\Gamma}(\tau,x;s,y)\phi^k_i(s,y)dsdy+f^k_i(\tau,x)
\end{equation}
along with the kernel
\begin{equation}
K^k_{\Gamma}(\tau,x;s,y)=\frac{\partial}{\partial \nu}\Gamma^l_k(\tau,x;s,y)+\alpha_i(\tau,x)\Gamma^l_k(\tau,x;s,y),
\end{equation}
and the functions $f^k_i$ which satisfy
\begin{equation}
 \begin{array}{ll}
f^k_i(\tau,x)=\\
\\
-\rho\int_{\Omega}\int_{l-1}^{\tau}\int_{\Omega} K_{\Gamma}(\tau,x;s,y)\\
\\
\times\sum_{j,m=1}^n\int_{\Omega}\left( \frac{\partial}{\partial x_i}K_n(y-z)\right) \left( \frac{\partial v^{\rho,k-1,l}_m}{\partial x_j}\frac{\partial v^{\rho,k-1,l}_j}{\partial x_m}\right) (s,z)dzdyds.
 \end{array}
\end{equation}
Well, the functions $\phi_i$ and $\phi^k_i$ have an explicit Levy-type expansion. For $k=0$ we have it in the form
\begin{equation}
\begin{array}{ll}
\frac{1}{2}\phi_i(\tau,x)=\\
\\
f_i(\tau,x)+\sum_{m=1}^{\infty}\int_{l-1}^{\tau}\int_SK^m_{\Gamma}(\tau,x;s,y)f_i(s,y)dsdy,
\end{array}
\end{equation}
where
\begin{equation}
K^1_{\Gamma}(\tau,x;s,y)=K_{\Gamma}(\tau,x;s,y),
\end{equation}
and
\begin{equation}
K^{m+1}_{\Gamma}(\tau,x;s,y)=\int_{l-1}^{\tau}\int_{\Omega}K^1_{\Gamma}(\tau,x;\sigma,z)K^m_{\Gamma}(\sigma,z;s,y)dyds.
\end{equation}
Similarly, for $k>0$ we have it in the form
\begin{equation}
\begin{array}{ll}
\frac{1}{2}\phi^k_i(\tau,x)=\\
\\
f^k_i(\tau,x)+\sum_{m=1}^{\infty}\int_{l-1}^{\tau}\int_SK^{m,k}_{\Gamma}(\tau,x;s,y)f^k_i(s,y)dsdy,
\end{array}
\end{equation}
where
\begin{equation}
K^{1,k}_{\Gamma}(\tau,x;s,y)=K^k_{\Gamma}(\tau,x;s,y),
\end{equation}
and
\begin{equation}
K^{m+1,k}_{\Gamma}(\tau,x;s,y)=\int_{l-1}^{\tau}\int_{\Omega}K^{1,k}_{\Gamma}(\tau,x;\sigma,z)K^{m,k}_{\Gamma}(\sigma,z;s,y)dyds.
\end{equation}

It is possible to extend this work to Navier-Stokes on manifolds. Another interesting problem is the extension to equations where the coefficients satisfy the H\"{o}rmander conditions and may have a stochastic force term. These problems will be studied in a subsequent paper.

\footnotetext[1]{
\texttt{{kampen@wias-berlin.de}, {kampen@mathalgorithm.de}}.}

\end{document}